\documentclass[leqno,11pt]{amsart}
\usepackage{etoolbox}

\makeatletter
\let\old@tocline\@tocline
\let\section@tocline\@tocline

\newcommand{\subsection@dotsep}{4.5}
\newcommand{\subsubsection@dotsep}{4.5}
\patchcmd{\@tocline}
  {\hfil}
  {\nobreak
     \leaders\hbox{$\m@th
        \mkern \subsection@dotsep mu\hbox{.}\mkern \subsection@dotsep mu$}\hfill
     \nobreak}{}{}
\let\subsection@tocline\@tocline
\let\@tocline\old@tocline

\patchcmd{\@tocline}
  {\hfil}
  {\nobreak
     \leaders\hbox{$\m@th
        \mkern \subsubsection@dotsep mu\hbox{.}\mkern \subsubsection@dotsep mu$}\hfill
     \nobreak}{}{}
\let\subsubsection@tocline\@tocline
\let\@tocline\old@tocline

\let\old@l@subsection\l@subsection
\let\old@l@subsubsection\l@subsubsection

\def\@tocwriteb#1#2#3{%
  \begingroup
    \@xp\def\csname #2@tocline\endcsname##1##2##3##4##5##6{%
      \ifnum##1>\c@tocdepth
      \else \sbox\z@{##5\let\indentlabel\@tochangmeasure##6}\fi}%
    \csname l@#2\endcsname{#1{\csname#2name\endcsname}{\@secnumber}{}}%
  \endgroup
  \addcontentsline{toc}{#2}%
    {\protect#1{\csname#2name\endcsname}{\@secnumber}{#3}}}%

\newlength{\@tocsectionindent}
\newlength{\@tocsubsectionindent}
\newlength{\@tocsubsubsectionindent}
\newlength{\@tocsectionnumwidth}
\newlength{\@tocsubsectionnumwidth}
\newlength{\@tocsubsubsectionnumwidth}
\newcommand{\settocsectionnumwidth}[1]{\setlength{\@tocsectionnumwidth}{#1}}
\newcommand{\settocsubsectionnumwidth}[1]{\setlength{\@tocsubsectionnumwidth}{#1}}
\newcommand{\settocsubsubsectionnumwidth}[1]{\setlength{\@tocsubsubsectionnumwidth}{#1}}
\newcommand{\settocsectionindent}[1]{\setlength{\@tocsectionindent}{#1}}
\newcommand{\settocsubsectionindent}[1]{\setlength{\@tocsubsectionindent}{#1}}
\newcommand{\settocsubsubsectionindent}[1]{\setlength{\@tocsubsubsectionindent}{#1}}

\renewcommand{\l@section}{\section@tocline{1}{\@tocsectionvskip}{\@tocsectionindent}{\@tocsectionnumwidth}{\@tocsectionformat}}%
\renewcommand{\l@subsection}{\subsection@tocline{1}{\@tocsubsectionvskip}{\@tocsubsectionindent}{\@tocsubsectionnumwidth}{\@tocsubsectionformat}}%
\renewcommand{\l@subsubsection}{\subsubsection@tocline{1}{\@tocsubsubsectionvskip}{\@tocsubsubsectionindent}{\@tocsubsubsectionnumwidth}{\@tocsubsubsectionformat}}%
\newcommand{\@tocsectionformat}{}
\newcommand{\@tocsubsectionformat}{}
\newcommand{\@tocsubsubsectionformat}{}
\expandafter\def\csname toc@1format\endcsname{\@tocsectionformat}
\expandafter\def\csname toc@2format\endcsname{\@tocsubsectionformat}
\expandafter\def\csname toc@3format\endcsname{\@tocsubsubsectionformat}
\newcommand{\settocsectionformat}[1]{\renewcommand{\@tocsectionformat}{#1}}
\newcommand{\settocsubsectionformat}[1]{\renewcommand{\@tocsubsectionformat}{#1}}
\newcommand{\settocsubsubsectionformat}[1]{\renewcommand{\@tocsubsubsectionformat}{#1}}
\newlength{\@tocsectionvskip}
\newcommand{\settocsectionvskip}[1]{\setlength{\@tocsectionvskip}{#1}}
\newlength{\@tocsubsectionvskip}
\newcommand{\settocsubsectionvskip}[1]{\setlength{\@tocsubsectionvskip}{#1}}
\newlength{\@tocsubsubsectionvskip}
\newcommand{\settocsubsubsectionvskip}[1]{\setlength{\@tocsubsubsectionvskip}{#1}}

\patchcmd{\tocsection}{\indentlabel}{\makebox[\@tocsectionnumwidth][l]}{}{}
\patchcmd{\tocsubsection}{\indentlabel}{\makebox[\@tocsubsectionnumwidth][l]}{}{}
\patchcmd{\tocsubsubsection}{\indentlabel}{\makebox[\@tocsubsubsectionnumwidth][l]}{}{}

\newcommand{\@sectypepnumformat}{}
\renewcommand{\contentsline}[1]{%
  \expandafter\let\expandafter\@sectypepnumformat\csname @toc#1pnumformat\endcsname%
  \csname l@#1\endcsname}
\newcommand{\@tocsectionpnumformat}{}
\newcommand{\@tocsubsectionpnumformat}{}
\newcommand{\@tocsubsubsectionpnumformat}{}
\newcommand{\setsectionpnumformat}[1]{\renewcommand{\@tocsectionpnumformat}{#1}}
\newcommand{\setsubsectionpnumformat}[1]{\renewcommand{\@tocsubsectionpnumformat}{#1}}
\newcommand{\setsubsubsectionpnumformat}[1]{\renewcommand{\@tocsubsubsectionpnumformat}{#1}}
\renewcommand{\@tocpagenum}[1]{%
  \hfill {\mdseries\@sectypepnumformat #1}}

\let\oldappendix\appendix
\renewcommand{\appendix}{%
  \leavevmode\oldappendix%
  \addtocontents{toc}{%
    \protect\settowidth{\protect\@tocsectionnumwidth}{\protect\@tocsectionformat\sectionname\space}%
    \protect\addtolength{\protect\@tocsectionnumwidth}{2em}}%
}
\makeatother

\makeatletter
\settocsectionnumwidth{2em}
\settocsubsectionnumwidth{2.5em}
\settocsubsubsectionnumwidth{3em}
\settocsectionindent{1pc}%
\settocsubsectionindent{\dimexpr\@tocsectionindent+\@tocsectionnumwidth}%
\settocsubsubsectionindent{\dimexpr\@tocsubsectionindent+\@tocsubsectionnumwidth}%
\makeatother

\settocsectionvskip{2pt}
\settocsubsectionvskip{0pt}
\settocsubsubsectionvskip{0pt}

\settocsectionformat{\bfseries}
\settocsubsectionformat{\mdseries}
\settocsubsubsectionformat{\mdseries}
\setsectionpnumformat{\bfseries}
\setsubsectionpnumformat{\mdseries}
\setsubsubsectionpnumformat{\mdseries}

\let\oldtableofcontents\tableofcontents
\renewcommand{\tableofcontents}{%
  \vspace*{-5\linespacing}
  \oldtableofcontents}

\usepackage{amsmath}
\usepackage{amssymb}
\usepackage{amsthm}
\usepackage{mathtools}
\usepackage{tikz-cd}
\usetikzlibrary{fit}
\usepackage{enumitem}
\setlist[enumerate]{label={\normalfont (\roman*)}, itemsep=1ex}

\usepackage[utf8]{inputenc}
\usepackage[T1]{fontenc}
\usepackage{lmodern}
\usepackage[babel]{microtype}
\usepackage[english]{babel}
\usepackage{relsize}
\usepackage{calc}

\usepackage{svg}

\usepackage{graphicx}
\usepackage{subcaption}

\usepackage{geometry}
\usepackage{xcolor} 
\colorlet{darkishRed}{red!60!black}
\colorlet{darkishBlue}{blue!60!black}
\colorlet{darkishGreen}{green!60!black}
\colorlet{darkblue}{blue!70!black}
\colorlet{darkishViolet}{violet}
\usepackage[linktoc=all]{hyperref}
\hypersetup{
	colorlinks,
    linkcolor={darkishBlue},
	citecolor={darkishGreen},
	urlcolor={darkishBlue}
}
\usepackage[nameinlink, capitalise, noabbrev]{cleveref}
\crefformat{enumi}{#2#1#3}
\crefformat{equation}{#2(#1)#3}
\crefname{equation}{}{}
\crefname{mainresult}{Theorem}{Theorems}
\let\setminus=\smallsetminus

\newcommand{\se}{\subseteq}
\newcommand{\sm}{\setminus}

\renewcommand{\leq}{\leqslant}
\renewcommand{\geq}{\geqslant}
\renewcommand{\ge}{\geq}
\renewcommand{\le}{\leq}

\DeclareMathOperator{\Aut}{Aut}

\newtheorem{theorem}{Theorem}[section] 
\newtheorem{proposition}[theorem]{Proposition}
\newtheorem{corollary}[theorem]{Corollary}
\newtheorem{lemma}[theorem]{Lemma}

\newtheorem{corlemma}[theorem]{Correspondence}
\crefname{corlemma}{Correspondence}{Correspondences}
\newtheorem{liftlemma}[theorem]{Lift}
\crefname{liftlemma}{Lift}{Lifts}
\newtheorem{projlemma}[theorem]{Projection}
\crefname{projlemma}{Projection}{Projections}

\newtheorem{observation}[theorem]{Observation}

\newtheorem{mainresult}{Theorem}

\newtheorem{mainproblem}{Problem}
\newtheorem{innercustomthm}{}

\newenvironment{customthm}[1]
  {\renewcommand\theinnercustomthm{#1}\innercustomthm}
  {\endinnercustomthm}

\theoremstyle{definition}
\newtheorem{example}[theorem]{Example}

\newtheorem{construction}[theorem]{Construction}
\theoremstyle{remark}

\newtheorem*{claim*}{Claim}

\crefname{claim}{Claim}{Claims}
\crefname{enumi}{}{}
\AtEndEnvironment{proof}{\setcounter{claim}{0}}

\usepackage{etoolbox}

\newcommand{\COMMENT}[1]{{}}
\definecolor{cMaroon}{HTML}{93152a}
\newcommand{\defnMath}[1]{{\color{darkishRed}{#1}}}
\newcommand{\defn}[1]{{\color{darkishRed}{\emph{#1}}}}
\newcommand{\casen}[1]{{\color{darkishViolet}{\emph{#1}}}}

\let\rho=\varrho
\let\phi=\varphi
\let\succeq=\succcurlyeq
\let\preceq=\preccurlyeq
\def\N{\mathbb N}

\def\Z{\mathbb Z}

\makeatletter

\def\calCommandfactory#1{%
  \expandafter\def\csname c#1\endcsname{\mathcal{#1}}}
\def\frakCommandfactory#1{%
  \expandafter\def\csname frak#1\endcsname{\mathfrak{#1}}}
   
\newcounter{ctr}
\loop
  \stepcounter{ctr}
  \edef\X{\@Alph\c@ctr}
  \expandafter\calCommandfactory\X
  \expandafter\frakCommandfactory\X
\ifnum\thectr<26
\repeat

\makeatletter
\patchcmd{\@tocline}
  {\hfil}
  {\leaders\hbox{\,.\,}\hfil}
  {}{}
\makeatother

\newcounter{mylabelcounter}

\makeatletter
\newcommand{\labelText}[2]{%
#1\refstepcounter{mylabelcounter}%
\immediate\write\@auxout{%
  \string\newlabel{#2}{{1}{\thepage}{{\unexpanded{#1}}}{mylabelcounter.\number\value{mylabelcounter}}{}}%
}%
}

\newbool{arXiv}
\boolfalse{arXiv} 

\newcommand{\arXivOrNot}[2]{\ifbool{arXiv}{{#1}}{{#2}}}

\newcommand{\lek}{({\le}\,k)}

\newcommand{\mindist}{\Delta}
\newcommand{\vx}{\text{\textnormal{vx}}}
\newcommand{\apart}{\text{\textnormal{part}}}

\newcommand{\td}{tree-decom\-pos\-ition}
\newcommand{\gd}{graph-decom\-pos\-ition}

\newcommand{\inv}{^{-1}}

\DeclareMathOperator{\orient}{ori}
\DeclareMathOperator{\id}{id}
\DeclareMathOperator{\interior}{int}

\newcommand{\orientN}[1]{\orient(#1)}
\newcommand{\tstar}{tri-star}

\newcommand{\ASSUMPTION}{\!\!\!\textsuperscript{\ref{ASSUMPTION}}}
\newtheoremstyle{mytheoremstyle}  
  {8pt}                     
  {8pt}                     
  {\itshape}               
  {}                        
  {\bfseries}              
  {.}                       
  { }                       
  {\thmname{#1} \thmnumber{#2}\,\textsuperscript{\ref{ASSUMPTION}}\ifthenelse{\equal{#3}{}}{}{\normalfont{ (#3)}}} 
\theoremstyle{mytheoremstyle}
\newtheorem{theoremASS}[theorem]{Theorem}
\crefname{theoremASS}{Theorem}{Theorems}
\crefformat{theoremASS}{#2Theorem #1\,\textsuperscript{\ref*{ASSUMPTION}}#3}
\newtheorem{propositionASS}[theorem]{Proposition}
\crefformat{propositionASS}{#2Proposition #1\,\textsuperscript{\ref*{ASSUMPTION}}#3}
\newtheorem{corollaryASS}[theorem]{Corollary}
\crefformat{corollaryASS}{#2Corollary #1\,\textsuperscript{\ref*{ASSUMPTION}}#3}
\newtheorem{lemmaASS}[theorem]{Lemma}
\crefformat{lemmaASS}{#2Lemma #1\,\textsuperscript{\ref*{ASSUMPTION}}#3}

\crefformat{keylemmaASS}{#2Key Lemma #1\,\textsuperscript{\ref*{ASSUMPTION}}#3}

\crefname{corlemmaASS}{Correspondence}{Correspondences}
\crefformat{corlemmaASS}{#2Correspondence #1\,\textsuperscript{\ref*{ASSUMPTION}}#3}
\newtheorem{liftlemmaASS}[theorem]{Lift}
\crefname{liftlemmaASS}{Lift}{Lifts}
\crefformat{liftlemmaASS}{#2Lift #1\,\textsuperscript{\ref*{ASSUMPTION}}#3}
\newtheorem{projlemmaASS}[theorem]{Projection}
\crefname{projlemmaASS}{Projection}{Projections}
\crefformat{projlemmaASS}{#2Projection #1\,\textsuperscript{\ref*{ASSUMPTION}}#3}

\usepackage{csquotes}
\usepackage[backend=biber, style=alphabetic, maxnames=99, maxalphanames=99, giveninits=true, sortcites=true]{biblatex}
\DeclareLabelalphaTemplate{
    \labelelement{
        \field[final]{shorthand}
        \field{label}
        \field[strwidth=1,strside=left]{labelname}
    }
    \labelelement{
        \field[strwidth=2,strside=right]{year}
    }
}

\title[Canonical graph decompositions via local separations]{Canonical graph decompositions via local separations}

\author[Raphael W. Jacobs]{Raphael~W.\ Jacobs$^\dagger$}%
\author[Paul Knappe]{Paul Knappe$^\ddagger$}%
\author[Jan Kurkofka]{Jan Kurkofka${}^{\includegraphics[height=.7\baselineskip]{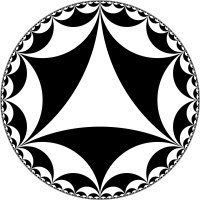}}$}

\thanks{${}^\dagger$ Supported by doctoral scholarships of the Studienstiftung des deutschen Volkes and the Cusanuswerk -- Bisch\"{o}fliche Studienf\"{o}rderung.}
\thanks{${}^\ddagger$ Supported by a doctoral scholarship of the Studienstiftung des deutschen Volkes.}
\thanks{$\includegraphics[height=.7\baselineskip]{Figures/MiniFarey.png}$
funded by the Deutsche Forschungsgemeinschaft (DFG, German
Research Foundation) -- 566118291.}

\keywords{Graph-decomposition, local separator, local separation, local connectivity, local covering, tree-decomposition, tangle, bottleneck, canonical}
\@namedef{subjclassname@2020}{\textup{2020} Mathematics Subject Classification}
\subjclass[2020]{05C83, 68R10, 57M15, 05C25, 05C40, 05C63, 05C38.}

\begin{document}
\thispagestyle{empty}

\begin{abstract}
    Every finite graph $G$ can be decomposed in a canonical way that displays its local connectivity-structure~\cite{canonicalGraphDec}.
    These decompositions are defined via a suitable more tree-like covering of $G$, whose tangle-tree structure is projected down to $G$.
    
    The covering graphs needed here are almost always infinite, and their tangle-tree structure is defined in terms of their (global) low-order separations. 
    The canonical decompositions they induce on $G$ are therefore not computable following their definition.

    We reconstruct these decompositions of $G$ from finite information in $G$ itself that is sufficiently local to be reflected in the cover.
    This involves the reconstruction of canonical tangle structure in terms of a new theory of local separations in finite graphs, which we develop for this purpose.
    
    As an application, we find that the canonical graph-decompositions from \cite{canonicalGraphDec} are computable.
\end{abstract}
\topskip0pt
\vspace*{\fill}
\maketitle

\vspace*{\fill}

\thispagestyle{empty}

\newpage
\thispagestyle{empty}
\tableofcontents

\newpage
\setcounter{page}{1}
\section{Introduction}

\subsection{From tree-decompositions to graph-decompositions} 
Tree-decompositions are a well-established concept to display the connectivity-structure of graphs when this resembles that of a tree.
Roughly speaking, a tree-decomposition (\cref{fig:TDCintro}) is a blueprint which shows how a graph $G$ (black) can be built from smaller graphs (grey), called \emph{parts}, by gluing them together in the shape of a tree (blue), called the \emph{decomposition tree}.
Each edge of the decomposition tree determines a \emph{separator} of~$G$ (circled in red), 
and more specifically induces a \emph{separation} which describes $G$ as a union of two halves that overlap precisely in that separator.

\begin{figure}[ht]
    \centering
    \includegraphics[height=5\baselineskip]{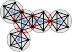}
    \caption{A tree-decomposition of a graph (black) into $K_5$'s (grey).}
    \label{fig:TDCintro}
\end{figure}

Tree-decompositions are a key concept in Robertson and Seymour's proof of their Graph Minor Theorem~\cite{GMXX}.
In computer science, many algorithmic problems that are NP-hard in general can be decided in linear time for classes of graphs admitting tree-decompositions into parts of bounded size, by Courcelle's Theorem~\cite{Courcelle}.
In group theory, and in Bass-Serre theory more specifically, tree-decompositions are used to detect product-structure in Cayley graphs such as amalgamated free products or HNN-extensions, for example in modern proofs of Stallings' Theorem~\cites{kron2010cutting,StallingsQuasi}.

One of the first applications of tree-decompositions sits at the interface of cluster analysis and graph theory.
Even though clusters are fuzzy in nature, there exist precise notions that describe clusters in graphs.
A well-known example is the abstract notion of a tangle, introduced by Robertson and Seymour alongside tree-decompositions~\cite{GMX}.
Their \emph{tangle-tree-decomposition theorem} states that every finite graph $G$ has a tree-decomposition such that every two tangles in~$G$ are minimally separated by some separation of the tree-decomposition.
Such tree-decompositions are referred to as tangle tree-decompositions.
An example is depicted in \cref{fig:TDCintro}, where each $K_5$-subgraph corresponds to a tangle.
The following recent work studies or constructs tangle tree-decompositions \cites{GroheTangles3,FiniteSplinters,RefiningToT,StructuralASS,AlbrechtsenRefiningToTinASS,AlbrechtsenRefiningTDCdisplayKblocks,AlbrechtsenOptimalToT,CarmesinToTshort,CDHH13CanonicalAlg,CDHH13CanonicalParts,confing,CG14:isolatingblocks,ProfilesNew,infinitesplinter,jacobs2023efficiently,entanglements,Kblocks} or even more specific decompositions \cites{Ugo,Tridecomp,TriAiC,Tetradecomp,Flo}.
Algorithmic applications include \cites{elbracht2020tangles,ComputingWithTangles,korhonen2024}, see \cite{TangleBook} for an overview.

Intuitively, tangle tree-decompositions capture the connectivity structure of graphs, as their separations represent the `bottlenecks' of the graph along which they decompose it into `highly connected parts', which in turn fit together as described by the decomposition tree.
However, there are many graphs for which their tangle tree-decompositions fail to detect much of their interesting connectivity structure. 
Indeed, when we perceive a graph to be formed from regions that are `locally highly connected' by arranging these regions in a `graph-like' way rather than a `tree-like' one, then trees of tangles fail to recover this arrangement.

\begin{enumerate}[label=(\Alph*)]
    \item\label{Intro:SparseNetworks} Large sparse networks form a wide class of examples, including networks from applications in Data Science such as infrastructure networks or social networks.
    For example, the orange regions in \cref{fig:RoadNetworkUK} are locally highly connected, and they are arranged in a `grid-like' way.
    \item\label{Intro:CayleyFinite} Cayley graphs of finite groups yield examples.
    The locally highly connected regions of the Cayley graph in \cref{fig:GroupDecomp} are copies of $K_4$ arranged in a `cycle-like' way.
    \item\label{Intro:LocallyChordal} Just as a chordal graph is obtained from cliques by gluing them together in the shape of a tree, a \emph{locally chordal} graph $G$ consists of cliques arranged in a `high-girth-like' way~\cite{locallychordal,localGlobalChordal,CanTDChordalGraphs}, see~\cref{fig:LocallyChordal}.
\end{enumerate}

\begin{figure}[ht]
    \centering
    \begin{subfigure}[b]{0.34\textwidth}
        \centering
        \includegraphics[height=10\baselineskip]{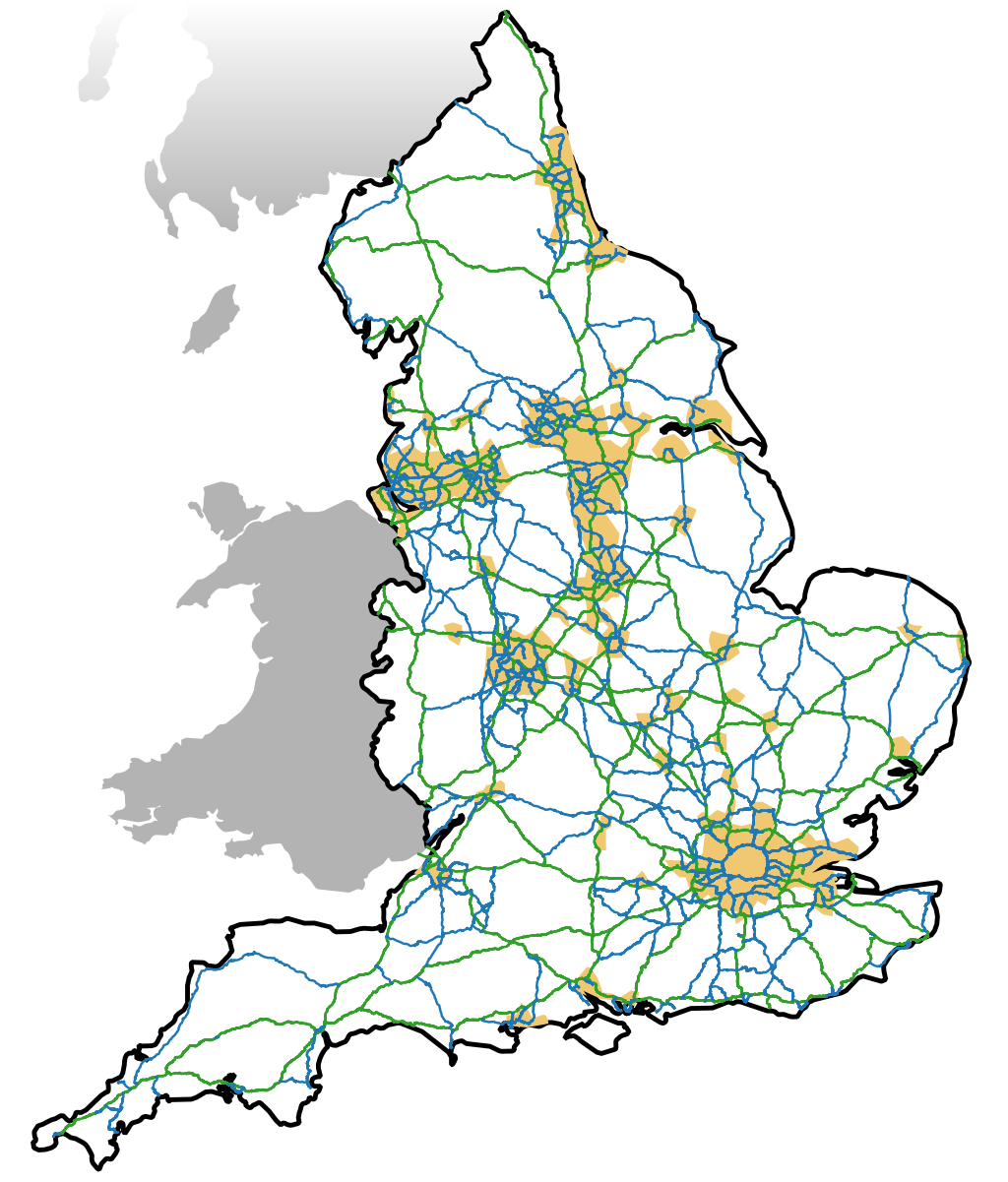}
        \caption{The main road network of the UK}
        \label{fig:RoadNetworkUK}
    \end{subfigure}
    \hfill
    \begin{subfigure}[b]{0.3\textwidth}
         \centering
         \raisebox{1.5\baselineskip}{\includegraphics[height=7\baselineskip]{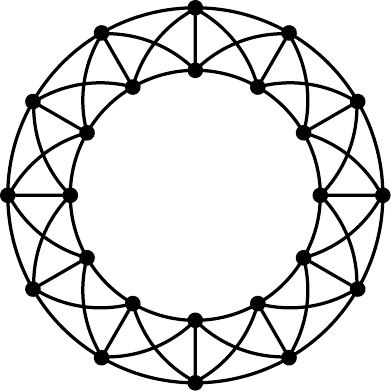}}
         \caption{A Cayley graph of $\Z_{12}\times\Z_2$}
         \label{fig:GroupDecomp}
    \end{subfigure}
    \hfill
    \begin{subfigure}[b]{0.34\textwidth}
        \centering
        \includegraphics[height=9\baselineskip]{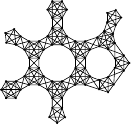}
        \caption{A graph that resembles two of the authors' favourite molecule}
        \label{fig:LocallyChordal}
    \end{subfigure}
    \caption{Graphs whose connectivity-structure eludes tangle tree-decompositions}
    \label{fig:IntroDecompExamples}
\end{figure}

A recent development in graph minor theory, initiated by Diestel and Kühn \cite{hierarchies} and followed up in \cite{canonicalGraphDec,StructuralDualityPathdecomposition,NoFatK4} and differently in \cite{Local2sep}, is the study of \emph{graph-decompositions}, which generalise tree-decompositions in that the decomposition tree may take the form of an arbitrary graph; see \cref{fig:IntroGraphDecompExamples}.
Diestel, Jacobs, Knappe and Kurkofka~\cite{canonicalGraphDec} have developed a `graph-decomposition version' of the tangle tree-decomposition theorem.

We briefly recall the construction; see \cref{sec:Terminology} for the precise definitions. 
Suppose we are given a connected graph~$G$.
Let $r >0$ be an integer parameter that will quantify the meaning of `local' in `locally highly connected'.
We construct a graph-decomposition $\cD_r(G)$ of~$G$ in three steps.

\begin{figure}[ht]
    \begin{subfigure}[b]{0.45\textwidth}
         \centering
         \includegraphics[height=8\baselineskip]{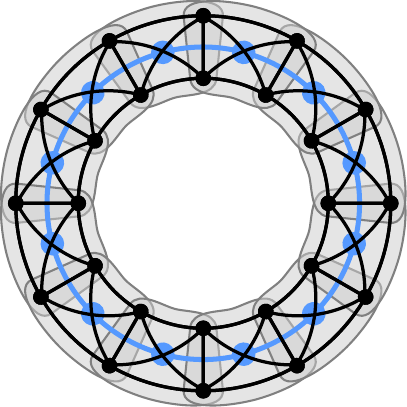}
    \end{subfigure}
    \hfill
    \begin{subfigure}[b]{0.45\textwidth}
        \centering
        \includegraphics[height=9\baselineskip]{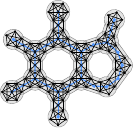}
    \end{subfigure}
    \caption{Graph-decompositions of black graphs into smaller parts (grey). The decomposition graphs are blue.}
    \label{fig:IntroGraphDecompExamples}
\end{figure}

First, we take the \emph{$r$-local covering} $p_r\colon G_r\to G$ of $G$ introduced in \cite{canonicalGraphDec}, which is the unique covering of~$G$ that preserves all $(r/2)$-balls of~$G$ but unfolds all other structure in $G$ as much as possible.
The graph $G_r$ is usually infinite, but it is highly symmetric; see \cref{fig:HGr}.
Intuitively, the $r$-locally highly connected regions in $G$ correspond to highly connected regions in~$G_r$, and in $G_r$ they are arranged in a `tree-like' way due to the nature of coverings -- so this arrangement is captured with a tangle tree-decomposition of~$G_r$.
In the second step, we apply a tangle-tree-decomposition theorem for infinite graphs~\cite{entanglements} to the covering $G_r$ to obtain a tangle tree-decomposition of~$G_r$ that is even canonical (i.e.\ invariant under graph isomorphisms such as the symmetries of~$G_r$).
In the third and final step, the tangle tree-decomposition of~$G_r$ is `folded back' via~$p_r$ to yield the graph-decomposition $\cD_r(G)$ of~$G$; see \cref{construction:orbittreedecompositionisgraphdecomposition}.
Following the above intuition, $\cD_r(G)$ captures the arrangement of the $r$-locally highly connected regions in~$G$.

\begin{figure}[ht]
    \centering
    \includegraphics[height=12\baselineskip]{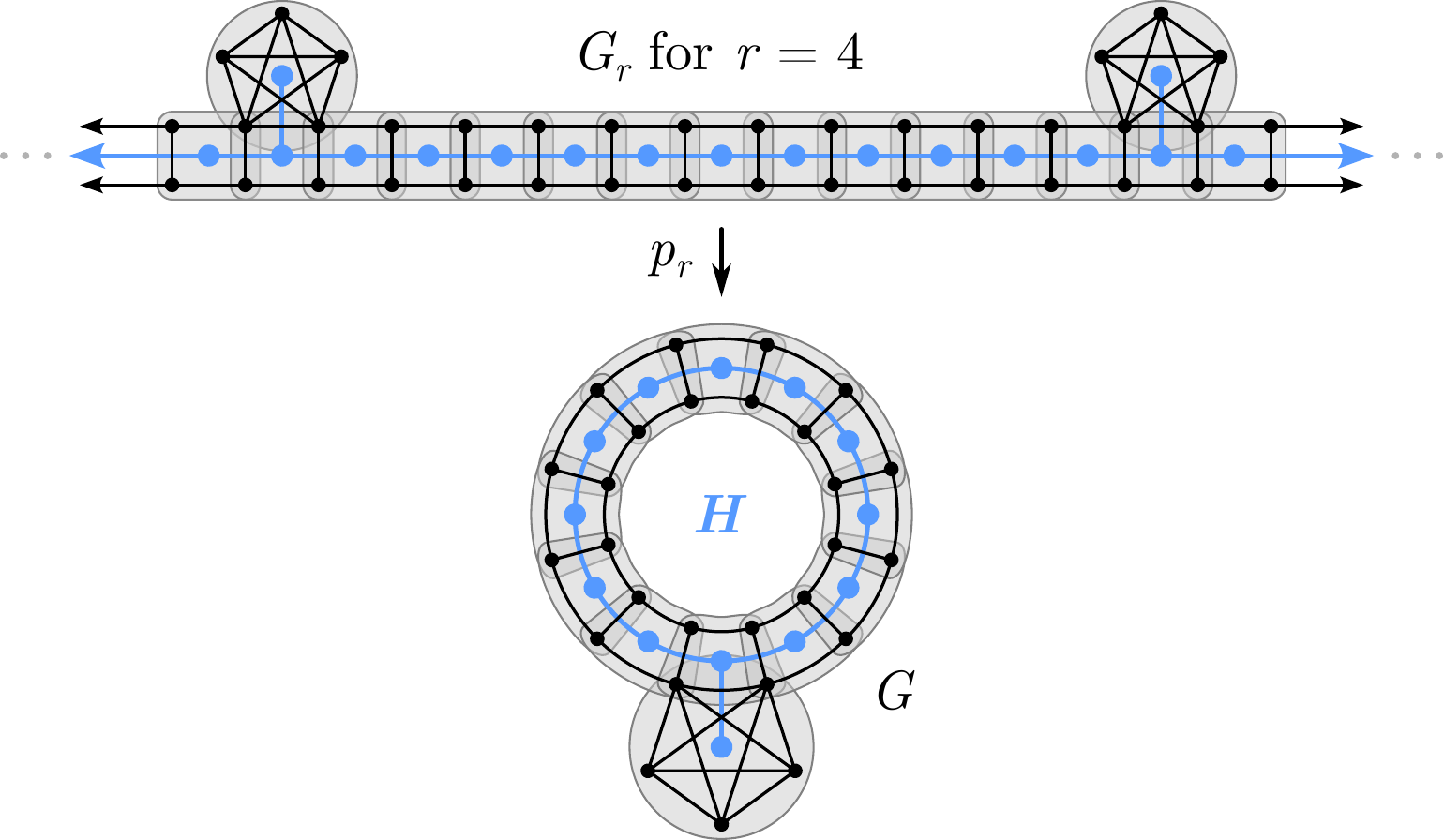}
         \caption{An $H$-decomposition of $G$ (black) into 4-cycles and a $K_5$ (grey) over~$H$ (blue), obtained from a tree-decomposition of the $r$-local cover~$G_r$ for $r=4$.}
    \label{fig:HGr}
\end{figure}

The graph-decomposition $\cD_r(G)$ is canonical, and unique given $G$ and~$r$.
In the case of $r=0$, the covering $G_r$ is the universal one, which is a tree: then $\cD_r(G)$ decomposes $G$ into its edges and vertices, and the decomposition graph is the $1$-subdivision of~$G$.
In the case of $r\ge |G|$, we have $G_r=G$, and so $\cD_r(G)$ defaults to a tangle tree-decomposition of~$G$.

The above construction of $\cD_r(G)$ involves topology and infinite graphs.
To compute such graph-decompositions, and thereby make them accessible to algorithms, we thus need to solve the following problem:
\enlargethispage{\baselineskip}
\begin{mainproblem}\label{Intro:questionDescription}
    Is there a finite combinatorial description of graph-decompositions like $\cD_r(G)$?
\end{mainproblem}

\subsection{From separations to local separations}

In a graph modelling a road network, small separators between large clusters form choke points that often get congested by traffic.
But not all such choke points come from small separators:
Consider for example the two red vertices in \cref{fig:BhamNham} that block the two shortest routes between Birmingham~(B) and Nottingham~(N).
Clearly, the two red vertices together form a choke point in between the two cities, as every detour is significantly longer.
However, the two red vertices do not separate the entire road network, but only do so locally (which is enough to make them a choke point):
they form an \emph{$r$-local separator}, that is, a vertex-set that separates roughly a ball of given radius $r/2>0$ around it.
In this paper, we develop a formal way to study such choke points via local separators.

\begin{figure}[ht]
    \centering
    \begin{subfigure}[b]{0.32\textwidth}
        \centering
        \includegraphics[height=6\baselineskip]{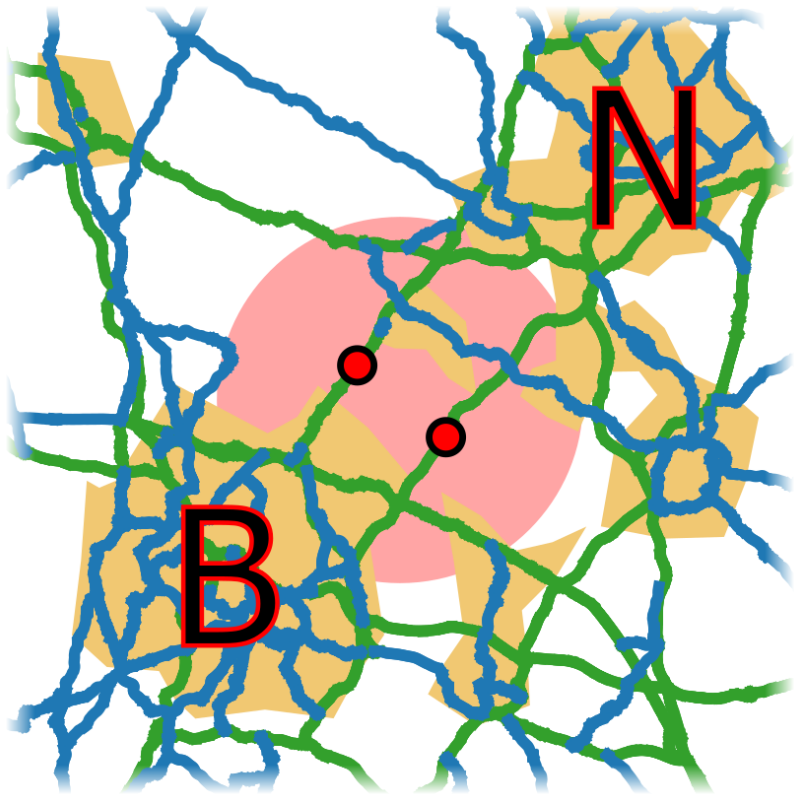}
        \caption{A portion of the main road network of the UK.}
        \label{fig:BhamNham}
    \end{subfigure}
    \hfill
    \begin{subfigure}[b]{0.32\textwidth}
         \centering
         \includegraphics[height=6\baselineskip]{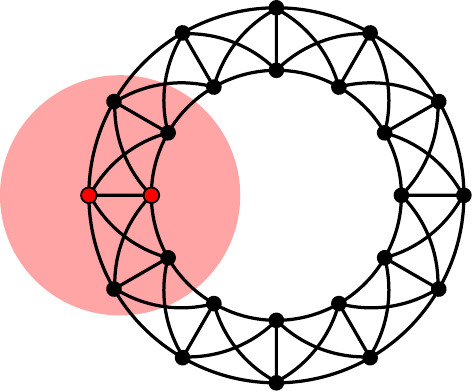}
         \caption{The red $r$-local 2-separator marks a factor of a direct product.}
         \label{fig:GroupLocalSeparator}
    \end{subfigure}
    \hfill
    \begin{subfigure}[b]{0.32\textwidth}
         \centering
         \includegraphics[height=9\baselineskip]{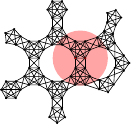}
         \caption{}
         \label{fig:LocallyChordalSeparator}
    \end{subfigure}
    \caption{Local separators (red) that separate the reddish balls around them.}
    \label{fig:IntroLocalSeparatorsExamples}
\end{figure}

Carmesin~\cite{Local2sep} initiated the study of $r$-local $1$- and $2$-separators, generalising the tree-like block-cutvertex decomposition and Tutte's 2-separator theorem~\cite{TutteGrTh} (which SPQR-trees are bases on) to the $r$-local setting.
Those early theorems provide canonical sets of pairwise non-crossing separations of order one and two, respectively, where the \emph{order} of a separation is the size of its separator.
The $r$-local $(\le 2)$-separations from~\cite{Local2sep} have since found applications to sparse networks and finite groups:
\begin{enumerate}[labelindent=0pt, labelwidth=\widthof{ad \cref{Intro:SparseNetworks}:}, leftmargin=!]
    \item[ad \cref{Intro:SparseNetworks}:] 
    Carmesin and Frenkel~\cite{Sarah} have developed an efficient algorithm for computing the canonical non-crossing $r$-local $(\le 2)$-separations and tested it on road networks (\cref{fig:BhamNham}).
    Their algorithm can use parallel computing, which is made possible by canonicity.
    \item[ad \cref{Intro:CayleyFinite}:]
    They were used in~\cite{StallingsNilpotent} to prove a low-order Stallings-type theorem for finite nilpotent groups (\cref{fig:GroupLocalSeparator}).
\end{enumerate}

\noindent Carmesin \cite[\S 10]{Local2sep} asked if his $r$-local $(\le 2)$-separations have analogues of orders $k > 2$:

\begin{mainproblem}\label{Intro:theoryLocalSeparators}
    Is there a decomposition theory for graphs based on $r$-local $k$-separations?
\end{mainproblem}

\subsection{Graph-decompositions and local separations}

Our aim in this paper is to present one unified solution to both \cref{Intro:questionDescription} and \cref{Intro:theoryLocalSeparators}:
We develop a finite combinatorial theory of local separations, and we use this theory to come up with a finite combinatorial description of graph-decompositions like~$\cD_r(G)$.
Our work is fully compatible with both the existing perspectives of $r$-local coverings~\cite{canonicalGraphDec} and of $r$-local $(\le 2)$-separations~\cite{Local2sep}.\footnote{We show this for $r$-local coverings in this paper, whereas for $r$-local $(\le 2)$-separations this will be proved in~\cite{Blueprints}.}
This answers \cref{Intro:questionDescription} and~\cref{Intro:theoryLocalSeparators}, up to the following caveat.
While the concepts that we develop as the basis for our theory of $r$-local $k$-separations are defined for all~$k$, the rest of the theory is only guaranteed to work for $k<K(G,r)$, where $K(G,r)$ is a graph invariant which we introduce next.

Given a graph $G$ and an integer $r>0$, we define the \emph{displacement} $\mindist_r(G)$ to be the least distance between any two distinct vertices in $G_r$ that $p_r$ maps to the same vertex in~$G$ (setting $\mindist_r(G):=\infty$ if $G_r=G$).\footnote{Both $G_r$ and $p_r$ will be formally defined in \cref{sec:Terminology}. Equivalent definitions for $K(G,r)$ are discussed in \cref{lem:midistKeyLemma}.}
We guarantee that our theory works for $r$-local separations of order $<K(G,r)$ where $K(G,r):=\frac{\mindist_r(G)}{r}+1$ (setting $K(G,r):=2$ if $r=0$).
The value of $K(G,r)$ is quite large in many graphs, and we provide a construction scheme for examples supporting this claim, see \cref{ex:BlocksCoincideGraphClass}.
Even in the worst case, the theory works for $r$-local $(\le 2)$-separations, see \cref{mindistGeR}.

The definition of the displacement $\mindist_r(G)$ involves the $r$-local covering, which is usually infinite.
Perhaps surprisingly, the displacement is computable \cite{AlgorithmLocalCover}.
We only use the displacement to guarantee that our theory works for $r$-local $k$-separations if $k < K(G,r)$; examples suggest that the theory often works beyond this guarantee.

\subsection{From trees of tangles to graphs of local bottlenecks}

The work \cites{FiniteSplinters,infinitesplinter} initiated a shift of perspective from tangles to the minimum-order separations that separate them.
Recent research~\cite{entanglements} has developed this perspective further by axiomatising these minimum-order separations through the new concept of \emph{entanglements}, and then used entanglements to come up with a new canonical construction of tangle tree-decompositions.\!\footnote{For ends instead of tangles, a similar construction but without entanglements had also been found by~\cite{Tim}.}
As this construction works without referencing tangles, we refer to it as an \emph{entanglement tree-decomposition}.
The characteristic property of an entanglement tree-decomposition of~$G$ is that it \emph{represents} all entanglements in~$G$, that is, for every entanglement in $G$ the decomposition tree has an edge that corresponds to some separation in that entanglement.

Hence, to find a finite combinatorial description of graph-decompositions like~$\cD_r(G)$, which are obtained from tangle tree-decompositions of the $r$-local covering~$G_r$ by folding them to graph-decompositions of~$G$, one might try to come up with a notion of `$r$-local entanglements' based on $r$-local separations, and then try to find a canonical construction of `$r$-local entanglement graph-decompositions'.
Roughly speaking, an `$r$-local entanglement graph-decomposition' of $G$ would be a graph-decomposition of $G$ that represents the $r$-local entanglements in~$G$.

Entanglements, however, present a challenge from an algorithmic standpoint. 
To determine whether a given set of $k$-separations constitutes an entanglement, one must compare these $k$-separations with all separations of $G$ across all orders, including those of order much larger than~$k$. 
This requirement complicates iterative algorithmic processes, which typically seek to build an entanglement tree-decomposition layer by layer: by starting with examining all entanglements formed by $\lek$-separations for $k = 1$, before progressively considering higher values of~$k$.
This requirement also causes a theoretical problem, as we can only guarantee that our finite combinatorial theory works as intended for $r$-local separations of order $<K(G,r)$.

To overcome the challenge associated with entanglements, we develop a new axiomatisation of the minimum-order separations between any two given tangles, which we call \emph{bottlenecks}.
All separations within a bottleneck share the same order, which is the \emph{order} of the bottleneck.
The main difficulty in conceptualising bottlenecks was to strike a delicate balance:
On the one hand, bottlenecks need to be specific enough to make it possible for us to find an $r$-local version of them. 
In particular, whether a given set of $k$-separations forms a bottleneck of order~$k$ must only depend on their relationships with the $\lek$-separations of~$G$.
On the other hand, bottlenecks need to be general enough to ensure that all minimum-order separations distinguishing any pair of tangles form a bottleneck.

As a main result, we show that every finite connected graph $G$ has for every $r$ and $k<K(G,r)$ an \emph{$r$-local bottleneck graph-decomposition} $\cH_r^{\le k}(G)$.
That is, $\cH_r^{\le k}(G)$ is a graph-decomposition of~$G$ that represents all the $r$-local bottlenecks $\beta$ in~$G$ of order~$\le k$: for every $\beta$ there is an edge of the decomposition graph $H$ of $\cH_r^{\le k}(G)$ that corresponds to an $r$-local separation in~$\beta$.
The parameter $k<K(G,r)$ allows us to set the maximum order for all $r$-local separations and all $r$-local bottlenecks considered in the construction of $\cH_r^{\le k}(G)$.
For fixed $G$ and~$r$, the decompositions $\cH_r^{\le 0}(G)$, $\cH_r^{\le 1}(G),\ldots$ then refine their predecessors, and hence allow for iterative algorithmic constructions.
The constructions may even continue for $k\ge K(G,r)$ if additional checks for the results are in place. 
\begin{enumerate}[labelindent=0pt, labelwidth=\widthof{ad \cref{Intro:LocallyChordal}:}, leftmargin=!]
    \item[ad \cref{Intro:LocallyChordal}:] The graph-decomposition $\cH_r^{\le k}(G)$ has been applied for values of $k$ beyond the guarantee $K(G,r)$ to structurally characterise the $r$-locally chordal graphs~$G$~\cite{CanTDChordalGraphs}.
\end{enumerate}
The graph-decompositions $\cH_r^{\le k}(G)$ are unique given $G,r,k$, and they are canonical.

If $r\ge |G|$, and more generally if $G$ itself is an $r$-local covering of some graph, then $\cH_r^{\le k}(G)$ defaults to a tangle tree-decomposition of~$G$.
For all~$r\in\N$, the graph-decomposition $\cH_r^{\le k}(G)$ coincides with the graph-decomposition obtained from the tangle tree-decomposition $\cH_r^{\le k}(G_r)$ of the $r$-local covering~$G_r$ by folding via~$p_r$.
In this sense, $\cH_r^{\le k}(G)$ exhibits the same qualities as the graph-decomposition $\cD_r(G)$ that is defined via $r$-local coverings, up to the parameter~$k$ that bounds the order of the $r$-local separations considered in the construction of~$\cH_r^{\le k}(G)$.

\begin{mainresult}\label{MainIntro}
    Let $G$ be a finite connected graph and $r,k \in \N$ with $k<K(G,r)$.
    \begin{enumerate}
        \item $\cH_r^{\le k}(G)$ is a canonical graph-decomposition of~$G$ that represents all the $r$-local bottlenecks in~$G$ of order~$\le k$.
        \item $\cH_r^{\le k}(G_r)$ is a canonical tangle tree-decomposition of scope~$k$ of the $r$-local covering~$G_r$.
        \item $\cH_r^{\le k}(G)$ is canonically isomorphic to the graph-decomposition of $G$ that is obtained from the canonical tree-decomposition $\cH_r^{\le k}(G_r)$ of the $r$-local covering~$G_r$ by folding via~$p_r$.
        \item $\cH_r^{\le k}(G)$ refines $\cH_r^{\le \ell}(G)$ for all $\ell\in\N$ with $\ell\le k$.
    \end{enumerate}
\end{mainresult}

\noindent See \cref{sec:Bottlenecks} and \cref{sec:MainProof} for definitions.

\begin{mainresult}\label{MainAlgo}
There is an algorithm that determines $\cH_r^{\le k}(G)$ given $G,r,k$ as in \cref{MainIntro}.
\end{mainresult}

\cref{MainIntro} and \cref{MainAlgo} together constitute one main result of this paper.
However, we hope that the finite combinatorial theory of $r$-local $k$-separations will be of independent interest.
\begin{enumerate}
    \item[ad \cref{Intro:CayleyFinite}:]
    Our work has been used just recently by Köhl and Salarian~\cite{Reza} to establish combinatorial characterisations of virtually torsion-free and virtually free groups.
\end{enumerate}

\subsection{Organisation of the paper}
\cref{sec:Terminology} recalls terminology about graph-decompositions.
In~\cref{sec:Bottlenecks}, we introduce bottlenecks and generalise a canonical version of the tangle-tree-decomposition theorem from~\cite{infinitesplinter} to bottlenecks.
In~\cref{sec:LocalSeparations}, we introduce the notion of local separations and study their interplay with local coverings.
In~\cref{sec:Displacement}, we introduce the notion of displacement and study how large displacement can be harnessed.
In~\cref{sec:LinksCrossing}, we expand the crossing relation from the usual separations to local separations and study its interplay with local coverings.
In~\cref{sec:LocalBottlenecks}, we introduce a local version of bottlenecks, study its interplay with local coverings, and prove a local-separations version of the tangle-tree-decomposition theorem.
In~\cref{sec:constructinggddirectly}, we show how graph-decompositions can be built from sets of pairwise non-crossing local separations.
In~\cref{sec:MainProof}, we prove \cref{MainIntro,MainAlgo}.
In~\cref{secCorners}, we discuss corners of local separations.

\subsection{Acknowledgements}

We thank Reinhard Diestel for reading the entire paper very carefully and sharing his detailed feedback, which has improved the paper across the board.
We are especially grateful for his suggestion of a more elegant definition of bottlenecks.
We are grateful to Nathan Bowler for pointing out an error in an early version.
We thank Johannes Carmesin for his inspiration in the early stages of this project.

\section{Review of graph-decompositions and coverings}\label{sec:Terminology}

Our notations and definitions follow \cite{bibel}.
Throughout this paper, let $G$ be a (simple and not necessarily finite) graph and $r \in \N$ an integer.
When $r$ is clear from context, which it usually will be, we refer to the cycles in $G$ of length~$\le r$ as \defn{short} cycles.
Let $X,Y,A$ be sets of vertices in~$G$.
An $X$--$Y$ path $P$ in $G$ is called a path \defn{through} $A$ if all internal vertices of $P$ are contained in $A$ and $P$ has at least one internal vertex.
Recall that a \defn{near-partition} of a set $M$ is a family $(M_i:i\in I)$ of subsets~$M_i$, possibly empty, that are disjoint and satisfy $M=\bigcup_{i\in I}M_i$.
If $|I|=2$, we speak of a \defn{near-bipartition}.

We remark that most of what follows is a recap of definitions and concepts from \cite{canonicalGraphDec} relevant to the present paper; in~\cite{canonicalGraphDec} you will find proofs or reference to proofs.

\subsection{Local coverings}\label{subsec:localcov}

A \defn{walk} (of \defn{length $k$}) in a graph $G$ is a non-empty alternating sequence $W=v_0 e_0 v_1\ldots v_{k-1} e_{k-1} v_k$ of vertices and edges in $G$ such that each $e_i$ joins $v_i$ and $v_{i+1}$.
If $v_0=v_k$ we call $W$ \defn{closed} and say that $W$ is a closed walk \defn{at}~$v_0$.
The subsequences of the form $v_i e_i\ldots v_j$ are called \defn{subwalks} of~$W$.
The walk that traverses $W$ backwards is denoted by~$W^-$.
We denote the set $\{v_0, \dots, v_k\}$ of vertices on the walk $W$ by \defn{$V(W)$}.
A walk $W$ is an \defn{$X$--$Y$ walk} for vertex sets $X,Y\se V(G)$ if $W$ has precisely its first vertex in $X$ and precisely its last vertex in~$Y$.
An \defn{$X$-walk} is a walk that has precisely its first and last vertex in~$X$.

We say that a walk $W$ is \defn{reduced} if it contains no subwalk of the form $ueveu$.
The \defn{reduction} of $W$ is the unique walk obtained by iteratively replacing its subwalks of the form $ueveu$ with the trivial walk~$u$.
Two walks are \defn{homotopic} if their reductions are equal.
Let \defn{$\sim$} denote the equivalence relation this defines on the set of all walks in~$G$.

For a vertex $x_0$ of $G$, we denote by \defn{$\cW(G,x_0)$} the set of all closed walks in $G$ at~$x_0$.
Then $\defnMath{\pi_1(G,x_0)}:=(\cW(G,x_0)/{\sim},\cdot)$ is the (\defn{combinatorial}) \defn{fundamental group} of $G$ at~$x_0$, where $[W_1]\cdot [W_2]:=[W_1 W_2]$ and $W_1 W_2$ is the concatenation of the two walks $W_1$ and $W_2$.
We often do not distinguish between $W$ and $[W]$ when it is clear from context what we mean,

A closed walk $W$ in $G$ at~$x_0$ \defn{stems} from a closed walk $Q$ if $W$ can be written as $W=W_0 Q W_0^-$.
A closed walk \defn{once around} a cycle $O$ is a closed walk in $O$ that traverses every edge of $O$ exactly once.
We say that $W$ \defn{stems} from a cycle $O$ if $W$ stems from a closed walk $Q$ once around~$O$.
The \defn{$r$-local subgroup} \defnMath{$\pi_1^r(G,x_0)$} of $\pi_1(G,x_0)$ is the subgroup generated by the closed walks at $x_0$ that stem from short cycles in~$G$ (those of length~$\le r$).
The $r$-local subgroup is normal.

A \defn{covering} of a graph $G$ is an epimorphism $p\colon C\to G$ for some graph~$C$ which exhibits the following property: for every vertex $v\in C$, the map $p$ bijectively sends the edges incident to $v$ to the edges incident to $p(v)$ in~$G$.
The vertices in $p\inv(v)$ are the \defn{lifts} of~$v$. 
Lifts of edges are defined similarly.
A \defn{lift} of a set $X$ of vertices of~$G$ to~$C$ is a vertex set $\hat X\se V(C)$ that contains for each $x\in X$ exactly one lift $\hat x$ of~$x$.
Equivalently, $p$ restricts to a bijection $\hat X\to X$.
A \defn{lift} of a walk $W$ in~$G$ is a walk in~$C$ which $p$ projects to~$W$.

There exists a covering $p_r\colon G_r\to G$ that induces an isomorphism from the fundamental group $\pi_1(G_r,\hat x_0)$ of $G_r$ to the $r$-local subgroup $\pi_1^r(G,x_0)$, where $\hat x_0$ is any lift of~$x_0$.
The covering $p_r\colon G_r\to G$ is unique up to isomorphism of coverings given $G$ and~$r$, and it is independent of the choice of~$x_0$.
In particular, the graph $G_r$ is determined by just $G$ and $r$.  
We call $p_r\colon G_r\to G$ the \defn{$r$-local covering} of~$G$.
Recall that the $0$-local, $1$-local and $2$-local covering of a simple graph $G$ are all its universal covering.

The group of \defn{deck transformations} of $G_r$, those automorphisms of $G_r$ whose action commutes with~$p_r$, is denoted by~\defnMath{$\Gamma_r(G)$}.
Since $\pi_1^r(G, x_0)$ is a normal subgroup of~$\pi_1(G, x_0)$, this group~$\Gamma_r(G)$ acts transitively on every fibre $p_r^{-1}(x)$ of the covering map~$p_r$ \cite[\S 1.3]{hatcher}.

\begin{lemma}\label{shortCyclesGenerate}
    {\normalfont\cite[Lemma~4.6]{canonicalGraphDec}}
    The binary cycle space of~$G_r$ is generated by the short cycles.
\end{lemma}

\noindent We will often use this lemma implicitly when working with~$G_r$.

Let $\rho$ be a non-negative integer.
The \defn{(combinatorial) $(\rho/2)$-ball} $\defnMath{B_G(v,\rho/2)}$ around a vertex $v$ in $G$ is formed by all vertices and edges traversed by closed walks in $G$ of length at most $\rho$ based at~$v$.
We remark that if~$\rho$ is odd, then the $(\rho/2)$-ball around $v$ is precisely the subgraph induced on all vertices of distance at most $\rho/2$ from $v$, while for even~$\rho$ any edges between vertices of distance $\rho/2$ from~$v$ will be missing.
A covering $p \colon C \to G$ of $G$ \defn{preserves $(\rho/2)$-balls} if $p$ restricts to an isomorphism $B_C(v, \rho/2) \to B_G(p(v),\rho/2)$ for every vertex $v$ of~$C$.
Note that the surjectivity of these restrictions always holds and the definition requires additional injectivity.\footnote{A notion that could be regarded as a precursor of the preservation of $(\rho/2)$-balls has been studied by Georgakopoulos~\cite{georgakopoulos2017covers}.}
By \cite[Lemma~4.3 and~4.4]{canonicalGraphDec}, the $r$-local covering $p_r$ may equivalently be described as the covering that is universal among the coverings $q \colon \hat G \to G$ preserving the $(r/2)$-balls.
Thus, the $r$-local covering $p_r$ maps every short cycle in $G_r$ of length at most $r$ isomorphically to $G$ \cite[Lemma~4.3]{canonicalGraphDec}.
Vice versa, any short cycle in $G$ lifts to a short cycle in $G_r$.

\subsection{Separations}

Let $G$ be any graph.
For $X \subseteq V(G)$, a component $K$ of $G-X$ is \defn{tight} at $X$ if $N_G(K) = X$.
A \defn{separation} of~$G$ is a set $\{A,B\}$ such that $A\cup B=V(G)$ and $G$ contains no edge between $A\sm B$ and $B\sm A$.
We refer to $A$ and $B$ as the \defn{sides} of~$\{A,B\}$, and we call $A\cap B$ the \defn{separator} of~$\{A,B\}$.
The size $\vert A\cap B\vert$ of the separator is the \defn{order} of~$\{A,B\}$.
Also we call a separation with separator $X$ a \defn{$k$-separation} if $k = |X|$.
A~separation~$\{A,B\}$ is \defn{proper} if $A\sm B$ and $B\sm A$ are non-empty.
If there are tight components $C_A, C_B$ of $G - (A \cap B)$ whose vertex sets are contained in $A$ and $B$, respectively, then the separation $\{A,B\}$ is \defn{tight}.

Every separation $\{A,B\}$ has two \defn{orientations} $(A,B)$ and $(B,A)$, and we refer to $(A,B)$ and $(B,A)$ as \defn{oriented separations}.
We denote the \defn{inverse} $(B,A)$ of an oriented separation $(A,B)$ by \defnMath{$(A,B)^*$}.
Given a set $S$ of (unoriented) separations, we denote the set of all orientations of separations in $S$ by \defn{$\orientN{S}$.}
The \defn{order} of $(A,B)$ and $(B,A)$ is the order of~$\{A,B\}$.
A partial ordering $\geq$ is defined on the ordered separations of $G$ by letting\footnote{We remark that we use the partial order as defined in the 6th but not the 5th edition of Diestel's book \cite{bibel}. That is, we write $\geq$ where readers familiar with previous papers would expect $\leq$.}
\[
    \defnMath{(A,B)\ge (C,D)}\quad :\Longleftrightarrow \quad A\se C\quad\text{and}\quad B\supseteq D
\]

Two separations $\{A,B\}$ and $\{C,D\}$ of~$G$ are \defn{nested}~if they have $\geq$-comparable orientations.
Then also every orientation of $\{A,B\}$ is \defn{nested} with every orientation of $\{C,D\}$.
If $\{A,B\}$ and $\{C,D\}$ are not nested, then they are said to \defn{cross}, and then also every orientation of $\{A,B\}$ \defn{crosses} every orientation of $\{C,D\}$.
A set of separations of~$G$ is \defn{nested} if its elements are pairwise nested.

The automorphism group $\Aut(G)$ of~$G$ acts on the oriented separations of~$G$ by~$(A, B) \mapsto (\phi(A), \phi(B))$ for~$\phi \in \Aut(G)$.
This action is compatible with~$\ge$.
Thus, the respective action on the (unoriented) separations of~$G$ also preserves the notions of nested and crossing.
We say that a set~$S$ of separations of~$G$ is~\defn{$\Gamma$-canonical} for some subgroup~$\Gamma$ of~$\Aut(G)$ if $S$ is invariant under the action of~$\Gamma$ on the separations of~$G$.
In the case of~$\Gamma = \Aut(G)$, we just say that~$S$ is \defn{canonical}.

The following definitions are illustrated in \cref{fig:LinksAndCorners}.
Let $\{A_1,A_2\}$ and $\{C_1,C_2\}$ be two separations of a graph~$G$ with separators $X$ and~$Y$, respectively.
The \defn{$X$-link for~$C_i$} is the vertex set $X\sm C_{3-i}$ (which is included in~$C_i$).
Similarly, the \defn{$Y$-link for~$A_i$} is the vertex set $Y\sm A_{3-i}\se A_i$.
The two $X$-links are said to be \defn{opposite} of each other, and similarly the $Y$-links are \defn{opposite} of each other.
The \defn{centre} of~$\{A_1, A_2\}$ and~$\{C_1, C_2\}$ is the intersection $X\cap Y$.
The separator $X$ is near-partitioned into its two $X$-links and the centre.

\begin{figure}[ht]
    \centering
    \includegraphics[height=6\baselineskip]{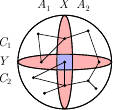}
    \caption{Four red links and a blue centre}
    \label{fig:LinksAndCorners}
\end{figure}

\begin{lemma}\label{char:NestedViaLinks}
    Let $G$ be a connected graph, and let $\{A_1,A_2\}$ and $\{C_1,C_2\}$ be separations of $G$ with separators $X$ and~$Y$, respectively.
    The following assertions are equivalent:
    \begin{enumerate}
        \item $\{A_1,A_2\}$ and $\{C_1,C_2\}$ cross;
        \item for every pair of indices $i,j\in [2]$, at least one of the following three sets is non-empty:\\
                the $Y$-link for $A_i$, the $X$-link for $C_j$, and $(A_i\sm A_{3-i})\cap (C_j\sm C_{3-j})$.
        \item \label{item:crossviacornersedges} for every pair of indices $i,j\in [2]$, at least one of the following three sets is non-empty:\\
                the $Y$-link for $A_i$, the $X$-link for $C_j$, and $E(X\cap Y,(A_i\sm A_{3-i})\cap (C_j\sm C_{3-j}))$.
    \end{enumerate}
    Moreover, \emph{(ii)} and \emph{(iii)} are violated for a pair of indices $i,j \in [2]$ if and only if $(A_i, A_{3-i}) \geq (C_{3-j}, C_j)$.
\end{lemma}
\begin{proof}
    (i)$\leftrightarrow$(ii) is \cite[Statement~(6)]{confing}.
    (iii)$\to$(ii) is trivial.
    It remains (ii)$\to$(iii).
    
    For this, let $i,j\in [2]$ and assume that the $Y$-link for~$A_i$ and the $X$-link for~$C_j$ are empty.
    By~(ii), there is a vertex $v$ in $(A_i\sm A_{3-i})\cap (C_j\sm C_{3-j})$.
    There is a vertex $u$ of $G$ outside of $(A_i\sm A_{3-i})\cap (C_j\sm C_{3-j})$, since otherwise (ii) is violated for $i':=3-i$ and $j':=3-j$.
    Since $G$ is connected, $G$ contains a $u$--$v$ path~$P$.
    As $\{A_i\cap C_j,A_{3-i}\cup C_{3-j}\}$ is a separation, the path $P$ must intersect its separator, which is equal to $X\cap Y$ by the assumption that the $Y$-link for~$A_i$ and the $X$-link for~$C_j$ are empty.
    Hence, the edge on $P$ before its first vertex in $X\cap Y$ lies in $E(X\cap Y,(A_i\sm A_{3-i})\cap (C_j\sm C_{3-j}))$.

    The moreover-part follows immediately from the definition of $\geq$ and (ii).
\end{proof}

\begin{figure}[ht]
    \centering
    \includegraphics[height=6\baselineskip]{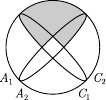}
    \caption{$(A_1,A_2)\ge (C_1,C_2)$ if and only if the grey area is empty.
    The grey area corresponds to $i=1$ and $j=2$ in \cref{char:NestedViaLinks}~(ii)}
    \label{fig:NestedViaLinks}
\end{figure}

\begin{figure}[ht]
    \centering
    \includegraphics[height=6\baselineskip]{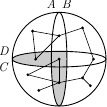}
    \caption{$\{B\cup D, A\cap C\}$ is one of the four corners of $\{A,B\}$ and~$\{C,D\}$}
    \label{fig:CrossingDiagram}
\end{figure}

For a depiction of the setting for the next definitions, see \autoref{fig:CrossingDiagram}.
If $\{A,B\}$ and $\{C,D\}$ cross, then their four \defn{corners} are the separations $\{B\cup D,A\cap C\}$, $\{B\cup C,A\cap D\}$, $\{A\cup C,B\cap D\}$ and $\{A\cup D,B\cap C\}$.
The corners $\{B\cup D,A\cap C\}$ and $\{A\cup C,B\cap D\}$ are \defn{opposite}, and so are the corners $\{B\cup C,A\cap D\}$ and~$\{A\cup D,B\cap C\}$.
Any two corners that are not opposite are \defn{adjacent}.
The two adjacent corners $\{B\cup D,A\cap C\}$ and $\{B\cup C,A\cap D\}$ are said to \defn{lie on the same side} of~$\{A,B\}$, namely \defn{on the $A$-side}.
Similarly, the other two adjacent corners $\{A\cup C,B\cap D\}$ and $\{A\cup D,B\cap C\}$ are said to \defn{lie on the same side} of~$\{A,B\}$, namely \defn{on the~$B$-side}.

The separations of $G$ satisfy the following condition~\cref{Fish}, see for example \cite[Lemma~12.5.5]{bibel}.
\begin{enumerate}[label=\rm{(F)}]
\item\label{Fish} For every two crossing separations $s$ and $s'$, every two opposite corners $c$ and $d$ of $s$ and $s'$, and every separation $t$, the following three assertions hold:
\begin{enumerate}[label=\rm{(F\arabic*)}]
        \item \label{item:F1} If $t$ crosses at least one of $c$ and $d$, then $t$ crosses at least one of $s$ and $s'$.
        \item \label{item:F2} If $t$ crosses both $c$ and $d$, then $t$ crosses both $s$ and $s'$.
        \item \label{item:F3} Neither $s$ nor $s'$ crosses $c$ or~$d$.
\end{enumerate}
\end{enumerate}

The order-function $\{A,B\}\mapsto |A\cap B|$ is \defn{subdmodular}\footnote{This is compatible with how the term `submodular' is used in the literature, as the set $\orientN{N}$ equipped with $\geq$ forms a lattice in which the infimum and supremum of two oriented separations are given by orientations of opposite corners.} on the separations of $G$ if it satisfies
\begin{enumerate}[label={\rm{(}$\dagger$\rm{)}}]
    \item\label{submod} If two separations $s$ and $s'$ cross and if $c$ and $d$ are opposite corners, then $|c|+|d|\le |s|+|s'|$.
\end{enumerate}

A set $\sigma$ of oriented separations of $G$ is a \defn{star} if $(A,B)\ge (D,C)$ for every distinct two $(A,B),(C,D) \in \sigma$, and its \defn{interior $\interior(\sigma)$} is $\bigcap_{(A,B) \in \sigma} B$.
Let $N$ be a nested set of proper separations of~$G$.
A \defn{splitting star} of~$N$ is a star $\sigma$ such that $\{A,B\}\in N$ and $(B,A)\notin\sigma$ for every $(A,B)\in\sigma$ and such that every $\{C,D\}\in N$ has an orientation $(C,D)$ with $(C,D)\ge (A,B)$ for some $(A,B)\in\sigma$.
This $(A,B) \in \sigma$ is clearly unique.

Note that every $(A,B) \in \orientN{N}$ lies in at most one splitting star.
Moreover, if $G$ is a finite graph, then the splitting stars of $N$ form a partition of $\orientN{N}$ \cite[\S 3]{TreeSets}.
Kneip and Gollin~\cite[Corollary~3.3]{infinitetreesets} showed that this also holds for infinite graphs, if additionally $\orientN{N}$ has no $(\omega+1)$-chain.

\subsection{Decompositions}
In this section, we mostly follow~\cite{canonicalGraphDec}.
Let $G$ be a graph, and $H$ a multigraph.
Let $\cV = (V_h : h \in V(H))$ be a family of sets $V_h$ of vertices of~$G$ indexed by the nodes of~$H$. 
The pair $(H, \cV)$ is an \defn{$H$-decomposition} or \defn{\gd} of $G$ if
\begin{enumerate}[label=\rm{(H\arabic*)}]
    \item\label{H1} $G = \bigcup_{h \in V(H)} G[V_h]$, and
    \item\label{H2} for every $v \in V(G)$, the subgraph $H[W_v]$ of $H$ induced by $W_v := \{h \in V(H) : v \in V_h \}$ is connected.
\end{enumerate}

\noindent We refer to the $V_h$ as its \defn{bags} and to the $W_v$ as its \defn{co-bags}, and we call $H$ its \defn{decomposition graph}.
If $H$ is a tree, then a \gd\ $(H, \cV)$ is a \defn{\td}.

Note the duality $v \in V_h \iff h \in W_v$.
We may thus reconstruct the families $\cV = (V_h : h \in V(H))$ and $\cW = (W_v : v \in V(G))$ from each other.

We often flesh out the bags and co-bags by choosing subgraphs~$G_h$ of $G$ on $V_h$ and $H_v$ of $H$ on $W_v$ in such a way that \cref{H1,H2} still hold, as
\begin{enumerate}[label=\rm{(H\arabic*')}]
    \item\label{H1'} $G = \bigcup_{h \in V(H)} G_h$, and
    \item\label{H2'} for every $v \in V(G)$, the subgraph $H_v$ of $H$ is connected.
\end{enumerate}

\noindent These~$G_h$, then, are our chosen \defn{parts} of the decomposition $(H, \cV)$, and the $H_v$ our chosen \defn{co-parts}.
We may then refer to the decomposition with these choices simply by citing the families $\cG = (G_h)_{h\in H}$ and $\cH = (H_v)_{v\in G}$ together with $G$ and~$H$.
Note that $\cV$ can be reconstructed from either of these families, and that it satisfies \cref{H1} and \cref{H2} if $\cG$ and~$\cH$ satisfy \cref{H1'} and~\cref{H2'}.
We often introduce graph-decompositions (with parts) as a pair $(H,\cV)$ (or~$(H, \cG)$) and tacitly assume that $\cV=(V_h:h\in V(H))$ (or $\cG = (G_h: h \in V(H))$), respectively.

In the case of \td s~$(T, \cV)$, we generally allow parts~$G_t$ other than the induced subgraphs~$G[V_t]$ on the bags.
However, throughout this paper, parts of a \td\ will always be the induced subgraphs~$G[V_t]$.
This agrees with the traditional definition of parts of \td s, as for example in~\cite{bibel}.

For a category-theoretic perspective on graph-decompositions and tree-decompositions, see~\cites{CategoryDecomp,WillTDC}.\nocite{Planken}

\begin{construction}[Tree-decompositions from nested sets of separations]\label{construction:candidatetreedecomposition}
    Let $G$ be a graph with a nested set $N$ of proper separations.
    We define the \defn{candidate} $\defnMath{\cT(N)} = (T,\cV)$ for a \td\ as follows.
    $T$ is a graph whose nodes are the splitting stars of $N$ and whose edges are the separations in~$N$, where $\{A,B\}\in N$ joins the two unique nodes $t_1,t_2$ of $T$ with $(A,B) \in t_1$ and $(B,A) \in t_2$.
    The bag $V_t$ of a node $t$ of $T$ is $\interior(t)$, with $t$ viewed as a splitting star of~$N$, and we usually flesh out its part as~$G_t := G[V_t]$.
\end{construction}

\noindent If $N$ is finite, then $T$ is a tree; moreover, $(T,\cV)$ is a \td\ of $G$ \cite[\S 6]{TreeSets}.
If $N$ is infinite, elements of $\orientN{N}$ may fail to lie in a splitting star, and thus $T$ may not be a tree.
But even if~$T$ is a tree, then the candidate~$\cT$ need in general not be a \td\ of a connected and locally finite graph~$G$~\cite[Example~4.9]{infinitesplinter}.
However, we have

\begin{lemma}\label{lemma:candidateisTDtightbounded}{\normalfont\cite[Lemmas~7.2 \& 7.5]{canonicalGraphDec}}
    Assume that $G$ is connected and locally finite.
    Let $N$ be a nested set of tight separations of $G$ which all have order $\leq k$ for some $k \in \N$.
    Then $\cT(N)$ is a \td.
    Moreover, if~$N$ is $\Gamma$-canonical for some subgroup $\Gamma$ of $\Aut(G)$, then so is~$\cT(N)$.
\end{lemma}

For a tree~$T$, we denote by \defn{$\orientN{E(T)}$} the set of all orientations $(u,v)$, $(v,u)$ of the edges $e = uv$ of~$T$.
The \defn{tree-order} on $\orientN{E(T)}$ is given by letting $e_1 \ge e_2$ for $e_i = (u_i,v_i) \in \orientN{E(T)}$ and $i \in \{1,2\}$ if $e_1=e_2$ or the (unique) $v_1$--$u_2$ path in $T$ avoids $u_1$ and $v_2$. 

Let $(T,\cV)$ be a \td\ of a graph~$G$.
Each oriented edge $e=(t_1,t_2) \in \orientN{E(T)}$ corresponds to a separation $(A^e_1,A^e_2)$ of $G$ where $A^e_i := \bigcup_{t \in V(T_i)} V_t$ and $T_i$ is the component of $T-t_1t_2$ which contains~$t_i$, see \cite[Lemma~12.3.1]{bibel}.

\begin{lemma}\label{lemma:withthehelpofTD}
    Let $N$ be a nested set of proper separations of a connected graph $G$ such that $\cT(N)=:(T,\cV)$ is a \td\ of $G$.
    Then $e \mapsto (A^e_1,A^e_2)$ is an order-isomorphism from $(\orientN{E(T)},{\ge})$ to $(\orientN{N},{\ge})$.
    The splitting stars of $N$ correspond, under this isomorphism, to the stars of all edges of $T$ at a fixed vertex, oriented towards that vertex.
\end{lemma}

\begin{proof}
    This is well known; 
    see for example \cite[Proof of Lemma 2.7]{infinitesplinter}.
\end{proof}

An \defn{action} of a subgroup $\Gamma$ of $\Aut(G)$ on a \gd\ $(H,\cV)$ of $G$ is an action of $\Gamma$ on the graph $H$ as automorphisms of $H$ that commutes with $h \mapsto V_h$, i.e.\ every $\phi \in \Gamma$ acts on $H$ as an automorphisms of $H$ so that $\phi(V_h) = V_{\phi(h)}$ for all~$h\in V(H)$.
If such an action exists, then $(H, \cV)$ is \defn{$\Gamma$-canonical}.
In the case of $\Gamma = \Aut(G)$, we just say that $(H,\cV)$ is \defn{canonical}.
If the \gd\ $(H, \cV)$ comes with parts~$G_h$, then it is \defn{$\Gamma$-canonical} if the action also respects these parts in that~$\phi(G_h) = G_{\phi(h)}$ for all~$\phi \in \Gamma$ and $h\in V(H)$, and in the case of~$\Gamma = \Aut(G)$, we call it \defn{canonical}.

If $\cT(N)$ is a tree-decomposition and $N$ is $\Gamma$-canonical, then the $\phi \in \Gamma$ act on $\orientN{N}$ as bijections $(A,B) \mapsto (\phi(A),\phi(B))$, so the action is compatible with $\geq$.
Thus, $\cT(N)$ is $\Gamma$-canonical by construction (which immediately yields the `moreover'-part of~\cref{lemma:candidateisTDtightbounded}).
Moreover, every automorphism $\phi$ of $G$ commutes with at most one automorphism $\psi$ of the decomposition tree $T$ in that $\phi(V_t) = V_{\psi(t)}$ \cite[Lemma~3.7]{canonicalGraphDec}.
So if $(T,\cV)$ is  $\Gamma$-canonical, then its $\Gamma$-canonicity is witnessed by a unique action of $\Gamma$ on $T$ as automorphisms of $T$ that commutes with $t \mapsto V_t$.

\begin{construction}[Graph-decompositions via local coverings {\cite[Construction~3.8]{canonicalGraphDec}}]\label{construction:orbittreedecompositionisgraphdecomposition}
    {\,}\\
    Let $G$ be a connected graph and let $r\in\N$.
    Let $\Gamma_r := \Gamma_r(G)$ be the group of deck transformations of the $r$-local covering of~$G$, and let $\cT = (T, \cV)$ be a $\Gamma_r$-canonical \td\ of~$G_r$.
    We define the \gd\ $\cH = (H, \cV)$ of~$G$, as follows.
    The decomposition graph $H$ is the orbit graph of $T$ under the action of $\Gamma_r$ on $(T, \cV)$, that is, the vertex and edge set of $H$ consist of the $\Gamma_r$-orbits of the vertex and edge set of $T$ where the endvertices of the orbit of $e = st$ are the orbits of $s$ and~$t$.
    The bag $V_h$ of a node $h$ of $H$ is defined as $V_h := p_r(V_t)$ for $t \in h$.
    If $(T, \cV)$ comes with parts~$G_t$, then we flesh out the family~$\cG$ of parts~$G_h$ of~$(H, \cV)$ as~$G_h := p_r(G_t)$ for~$t \in h$.
    We say that $(H,\cV)$ and $(H, \cG)$ are \defn{defined} by $(T,\cV)$ (and its parts) \defn{via~$p_r$}, or that $(H,\cV)$ and $(H, \cG)$ are \defn{obtained} from~$(T,\cV)$ (and its parts) by \defn{folding via~$p_r$}.
\end{construction}

\noindent Pairs $(H,\cV)$ from \cref{construction:orbittreedecompositionisgraphdecomposition} are graph-decompositions by \cite[Lemma~3.9]{canonicalGraphDec}.
We remark that even though $G$ is simple, $H$ does not have to be simple.

\section{Bottlenecks}\label{sec:Bottlenecks}

In this section we introduce \emph{bottlenecks}, which will later become one of our main tools to transfer the Tree-of-Tangles Theorem to local separations (see~\cref{sec:LocalBottlenecks}).
Our main result of this section, \cref{BottleneckNestedSet}, asserts that every locally finite and connected graph admits a nested set of separations meeting all bottlenecks.
This implies the existence of trees of tangles, see \cref{BottlenecksImpliesBoundedToT} and \cref{BottlenecksImpliesToT}.

Bottlenecks build on a recent trend to construct trees of tangles by working with the minimum-order separations between the tangles instead of with the tangles themselves.
This shift of perspective was initiated by~\cites{FiniteSplinters,infinitesplinter} and led to a first axiomatisation of such sets of separations through \emph{entanglements}~\cite{entanglements}.
With bottlenecks we here give a second, alternative, axiomatisation, which in contrast to entanglements allows for an iterative construction of a tangle tree-decompositions, see \cref{construction:inductiveNestedSet}.
This in turn makes possible the formulation of an $r$-local version of a tree of tangles that is compatible with the guarantee $K(G,r)$, see \cref{sec:LocalBottlenecks}.
We do not develop a local version of tangles, as the need did not arise in this project.
Still, we think it would be exciting to have a local version of tangles.

\subsection{Bottlenecks}
For $n \in \N$, we denote $\{1, \dots, n\}$ by \defn{$[n]$}.
A \defn{\tstar } of~$G$ is a set $\sigma=\{\,(A_i,B_i):i\in [3]\,\}$ of three distinct oriented separations $(A_i,B_i)$ of~$G$ with separators~$X_i$ such that $B_i = A_j \cup A_k$ whenever $\{i,j,k\} = [3]$.
Note that \tstar s are stars.
The unoriented separations $\{A_i,B_i\}$ are the \defn{constituents} of~$\sigma$.
The \tstar\ $\sigma$ has \defn{order~$\le k$} if all its elements have order~$\le k$.
We remark that $\bigcup_{i \in [3]} X_i$ has size~$\le\lfloor 3k/2 \rfloor$ if $\sigma$ has order~$\leq k$.
If a \tstar\ $\sigma$ is introduced as $\{\,(A_1,B_1),\,(A_2,B_2),\,(A_3,B_3)\,\}$, then we denote by $X_i$ the separator $A_i \cap B_i$ of $\{A_i,B_i\}$, let $X$ be their union $\bigcup_{i \in [3]} X_i$, let $Z$ be the \defn{centre} $\bigcap_{i \in [3]} X_i$ of $\sigma$, and denote the \defn{$ij$-link} $(X_i \cap X_j) \setminus Z = X \setminus X_k$ by \defn{$X_{ij}$} whenever $\{i,j,k\} = [3]$.

These definitions are inspired by their usual occurrence in a corner diagram:

\begin{example}\label{tristarExample}
    This example is depicted in \cref{fig:tristarExample}.
    Let $\{A,B\}$ and $\{C,D\}$ be crossing separations of a graph~$G$ with separators $X$ and~$Y$, respectively.
    Let $(A_1,B_1):=(A,B)$ and let $(A_2,B_2):=(B\cap C,A\cup D)$ and $(A_3,B_3):=(B\cap D,A\cup C)$ be the two corners on the $B$-side.
    Then $\sigma:=\{\,(A_i,B_i):i\in [3]\,\}$ is a \tstar\ of~$G$.
    The $12$-link of~$\sigma$ is equal to the $X$-link for~$C$.
    The $13$-link of~$\sigma$ is equal to the $X$-link for~$D$.
    The $23$-link of~$\sigma$ is equal to the $Y$-link for~$B$.
    The centre of~$\sigma$ is equal to the centre $X\cap Y$ of the two crossing separations.
\end{example}

\begin{figure}[ht]
    \centering
    \includegraphics[height=7\baselineskip]{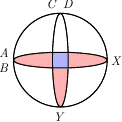}
    \caption{The situation in \cref{tristarExample}: the three $ij$-links are red, the centre of $\sigma$ is blue}
    \label{fig:tristarExample}
\end{figure}

It is immediate from the definition of \tstar\ that
\begin{equation}\label{tristarLinkPartition}
    X_i=X_{ij}\sqcup Z\sqcup X_{ik}\text{ whenever }\{i,j,k\}= [3]
\end{equation}

A \tstar\ $\sigma = \{\,(A_1,B_1),\,(A_2,B_2),\,(A_3,B_3)\,\}$ is \defn{relevant with base} $\{A_1,B_1\}$ if $\{A_1,B_1\}$ is tight and for every $x \in X_{23}$ and every $i \in \{2,3\}$ there exists an $x$--$X_{1i}$ edge in $G$ or there exist $y\in X_{1i}\cup Z$ and a component $K$ of $G[A_i\sm B_i]$ with $x,y\in N(K)$; see \cref{fig:RelevantTristar}.

\begin{figure}[ht]
    \centering
    \includegraphics[height=6\baselineskip]{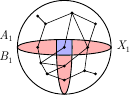}
    \caption{$(A_1,B_1)$ is part of a \tstar\ that is relevant with base $\{A_1,B_1\}$. The links are red and the centre is blue. The red link on the right is the $23$-link}
    \label{fig:RelevantTristar}
\end{figure}

\begin{lemma}\label{crossingTightLinkAnalysis}
    Let $\{A,B\}$ and $\{C,D\}$ be crossing tight separations of a connected graph~$G$, with separators $X$ and~$Y$, respectively.
    Assume that the $X$-link for~$C$ is empty while the $Y$-link for~$B$ is non-empty.
    Then the $Y$-link for~$A$ is empty while $X\cap Y$ is non-empty.
\end{lemma}
\begin{proof}
    Since $\{C,D\}$ is tight, there is a component $K\se G[C\sm D]$ with $N(K)=Y$.
    Since the $X$-link for~$C$ is empty, and since $K$ has a neighbour in the non-empty $Y$-link for~$B$, we get that $K$ avoids~$A$.
    So $Y=N(K)$ is included in $Y\cap B$, and in particular the $Y$-link for~$A$ is empty.
    Since $\{A,B\}$ and $\{C,D\}$ cross, and since the $X$-link for~$C$ and the $Y$-link for~$A$ are empty, the centre $X\cap Y$ is non-empty by \cref{char:NestedViaLinks}.
\end{proof}

\begin{lemma} \label{CrossingTightSeparationsFormRelevantYStar}
    Let $\{A,B\}$ and $\{C,D\}$ be two tight crossing separations of a connected graph $G$.
    Then their two oriented corners $(B \cap C, A \cup D)$ and $(B \cap D, A \cup C)$ on the $B$-side together with $(A,B)$ form a relevant \tstar\ with base $\{A,B\}$.
\end{lemma}

\begin{proof}
    Let $(A_1,B_1) := (A,B)$, $(A_2,B_2) := (B \cap C, A \cup D)$ and $(A_3,B_3) := (B \cap D, A \cup C)$.
    Then $\sigma := \{\,(A_i,B_i) : i \in [3]\,\}$ is a \tstar\ by \cref{tristarExample}.
    Let $X_i:=A_i\cap B_i$ for all~$i$.
    We will implicitly use the correspondences between the links and centres provided by \cref{tristarExample}.
    
    It remains to show that $\sigma$ is relevant with base~$\{A_1,B_1\}$.
    We may assume that $X_{23}$ is non-empty, as we are done otherwise.
    Let $x \in X_{23}$ and $i\in \{2,3\}$, say $i=2$, as the other case follows by symmetry.
    We have to show that $G$ contains either an $x$--$X_{12}$ edge or that there are a vertex $y\in X_{12}\cup Z$ and a component $K$ of $G[A_2\sm B_2]$ with $x,y\in N(K)$.
    
    Since $\{C,D\}$ is tight, there exists a component $K$ of $G- (C \cap D)$ with $V(K) \subseteq C \setminus D$ and $N(K)=C\cap D$.
    Recall that $Z:=\bigcap_{i\in [3]}X_i$ satisfies $Z=A\cap B\cap C\cap D$ by the definition of~$\sigma$.
    Hence, $N(K) = X_{23} \cup Z \cup Y$, where $Y$ is the link $(C \cap D) \setminus B$.
    We distinguish three cases.

    First, we assume that $X_{12}$ is empty.
    Since $X_{23}$ is non-empty, \cref{crossingTightLinkAnalysis} gives that $Z$ is non-empty while $Y$ is empty.
    Since $K$ is included in $G[C\sm D]$ and has a neighbour $x\in X_{23}$, it must also be included in $G[A_2\sm B_2]$.
    Thus, $K$ is a component of $G[A_2\sm B_2]$ with a neighbour in~$Z$.

    Second, we assume that $Y$ is non-empty.
    Since $x \in X_{23} \subseteq B \setminus A$ and $Y \subseteq A \setminus B$, there is an $x$--$Y$ path $Q$ through $K \subseteq G[C \setminus D]$, which meets $(A \cap B) \setminus D = X_{12}$.
    Let $P$ be the $x$--$X_{12}$ path in~$Q$.
    Then $P$ either is an $x$--$X_{12}$ path through $A_2\sm B_2$ or contains an $x$--$X_{12}$ edge.
    In the former case, the internal vertices of $P$ are contained in a desired component of~$G[A_2\sm B_2]$.

    Finally, we assume that $X_{12}$ is non-empty while $Y$ is empty.
    \cref{crossingTightLinkAnalysis} (applied to $(A,B) := (D,C)$, $X := A \cap B$, $(C,D) := (A,B)$ and $Y := C \cap D$) yields that 
    the centre $Z$ is non-empty.
    Consider an $x$--$Z$ path $Q$ through the component $K \subseteq G[C \setminus D]$.
    Let $P$ be the $x$--$(X_{12}\cup Z)$ subpath of~$Q$.
    If $P$ has no internal vertex, then $P$ ends in $X_{12}$ because $Q$ goes through~$K$.
    Hence, $P$ either contains an $x$--$X_{12}$ edge or is an $x$--$(X_{12}\cup Z)$ path through~$A_2\sm B_2$.
    In the latter case, the internal vertices of $P$ are contained in a desired component of~$G[A_2\sm B_2]$.
\end{proof}

A \defn{bottleneck of order~$k\in\N$} in~$G$ is a non-empty set $\beta$ of tight $k$-separations of $G$ satisfying:
\begin{enumerate}[label={\rm(B)}]
    \item\label{Entangle1}
    Whenever $\beta$ contains the base of a relevant \tstar\ $\sigma$ of $G$ of order~$\le k$, at least one of the other two constituents of $\sigma$ is contained in $\beta$ as well.
\end{enumerate}
A bottleneck of order~$k$ is also called a \defn{$k$-bottleneck} for short.

Recall that a separation \defn{distinguishes} two tangles if the two tangles orient the separation to different sides.
A separation distinguishes two given tangles \defn{efficiently} if it distinguishes them with minimal order.
Two tangles in $G$ are \defn{$\lek$-distinguishable} if there is a separation of $G$ of order~$\le k$ that distinguishes the two tangles.
Two tangles in $G$ are \defn{distinguishable} if some separation distinguishes them.

\begin{example}\label{TanglesInduceBottleneck}
    For every pair of distinguishable tangles $\tau,\tau'$ in a graph~$G$, the set $\beta=\beta(\tau,\tau')$ of all separations efficiently distinguishing $\tau$ and~$\tau'$ forms a bottleneck~in~$G$; see \cite[§12.5]{bibel} for definitions regarding tangles.
\end{example}

\begin{proof}
    As $\tau$ and $\tau'$ are distinguishable, there is a $k$-separation of~$G$ that efficiently distinguishes $\tau$ and~$\tau'$.
    Hence, $\beta$ is non-empty, and every separation in~$\beta$ has the same order~$k$.
    We claim that $\beta$ is $k$-bottleneck.
    By their efficiency, every separation in $\beta$ is tight; compare \cite[Lemma~6.1 and its proof]{infinitesplinter}.
    To show that $\beta$ satisifes~\cref{Entangle1},
    let $\sigma=\{\,(A_i,B_i):i\in [3]\,\}$ be a relevant \tstar\ of~$G$ of order~$\le k$ with base $\{A_1,B_1\}\in\beta$.
    Then $\tau$ and $\tau'$ orient $\{A_1,B_1\}$ differently; we may assume that $\tau$ towards its $A_1$-side, and $\tau'$ towards its $B_1$-side, say.
    Both tangles $\tau$, $\tau'$ also orient the other constituents of~$\sigma$.
    By the tangle property for $\tau'$, there is $j\in\{2,3\}$ such that $\tau'$ orients $\{A_j,B_j\}$ towards its $A_j$-side.
    Moreover, $\tau$ orients $\{A_j,B_j\}$ towards its $B_j$-side by the tangle property for $\tau$.
    Hence, $\{A_j,B_j\}$ distinguishes $\tau$ and~$\tau'$, and it must do so efficiently since $\{A_j,B_j\}$ has order~$\le k$.
    Therefore, $\{A_j,B_j\}\in\beta$.
\end{proof}

An oriented separation $(A,B)$ \defn{points towards} a vertex set $Y$ of $G$ if $Y \subseteq B$.
Let $Y$ be a clique in $G$ of size $k$, i.e.\ a set of vertices which induces a complete subgraph in $G$.
Then precisely one orientation of a given separation $\{A,B\}$ of order $<k$ points towards $Y$.
By \defn{$\tau_Y$}, we denote the set of oriented separations of order $< |Y| = k$ pointing towards $Y$.
We warn the reader that $\tau_Y$ might not be a tangle of order $k$.
A separation $\{A,B\}$ of $G$ \defn{distinguishes} two cliques $Y,Z$ of $G$ if $(A,B) \in \tau_Y$ and $(B,A) \in \tau_Z$.
A separation distinguishes $Y$ and $Z$ \defn{efficiently} if it distinguishes them with minimal order.

\begin{example}\label{CliquesInduceBottlenecks}
    Let $Y$ and $Z$ be two cliques in a graph $G$ which are distinguished by a separation of~$G$ of order $< \min \{\,|Y|,|Z|\,\}$.
    Then $Y$ and $Z$ induce a bottleneck $\beta = \beta(Y,Z)$ in $G$, which consists of the separations of~$G$ that efficiently distinguish $Y$ and $Z$. 
\end{example}

\begin{proof}
    The set $\beta$ is non-empty and consisits of by assumption.
    Denote by $k$ the minimal order of a separation distinguishing $Y$ and $Z$.
    Then all separations in~$\beta$ have order~$k$.
    We claim that every separation in $\beta$ is tight.
    Indeed, let $\{A,B\}$ be a separation in $\beta$.
    We may assume that $(A,B) \in \tau_Y$ and $(B,A) \in \tau_Z$.
    Then each vertex in $Y \cap (A \setminus B)$ from each vertex in $Z \cap (B \setminus A)$ order-minimally.
    Since $k < \min \{|Y|,|Z|\}$ by assumption, both these sets are non-empty.
    Thus, every $\{A,B\}$ in $\beta$ is tight.

    It remains to show that $\beta$ satisfies \cref{Entangle1}.
    For this, let $\sigma=\{\,(A_i,B_i):i\in [3]\,\}$ be a relevant \tstar\ of~$G$ of order~$\le k$ with base $\{A_1,B_1\}\in\beta$.
    As the base $\{A_1,B_1\}$ is in $\beta$, we may assume by symmetry that $Y \subseteq A_1$ and $Z \subseteq B_1$; in particular, $Y \subseteq B_2,B_3$.
    Since the order of $\sigma$ is $< \min \{|Y|,|Z|\}$, the clique $Z$ does not lie in the union of the separators of the $(A_i,B_i)$.
    Thus, $\{A_1,B_1\}\in\beta$ yields that either $A_2$ or $A_3$ contains $Z$, and thus $\{A_2,B_2\}$ or $\{A_3,B_3\}$, respectively, is in $\beta$.
\end{proof}

\subsection{A tree of bottlenecks}

The following construction transfers the construction introduced in the proof of \cite[Theorem~1.2]{infinitesplinter} to bottlenecks:

\begin{construction}[Nested separations from bottlenecks]\label{construction:inductiveNestedSet}
    Let $G$ be a graph, and let $\cB$ be a set of bottlenecks in $G$.
    For every $k \in \N$, let $\defnMath{\cB^k(G)}$ be the set of all $k$-bottlenecks in $\cB$ in $G$ and $S^k:=\bigcup\cB^k(G)$.
    Recursively for all $k\in\N$, define nested sets $N^k(G) \subseteq S^k$ of separations, as follows.
    Let $\defnMath{x^k}\colon S^k\to\N$ assign to each separation $s\in S^k$ the number of separations in $S^k$ that $s$ crosses, assuming for now that this number is finite.
    Write $\defnMath{N^{<k}(G)}:=\bigcup_{j < k} N^j(G)$.
    For each $\beta \in \cB^k(G)$, let \defnMath{$N^k(G,\beta)$} consist of the separations $s\in \beta$ which are nested with $N^{<k}(G)$ and among those have minimal~$x^k(s)$.
    Finally, set $\defnMath{N^k(G)} := \bigcup\,\{\,N^k(G,\beta):\beta\in\cB^k(G)\,\}$.
    We further write $\defnMath{N(G)}:=\bigcup_{k\in\N} N^k(G)$.
\end{construction}

\noindent Whenever we use the notation of $N^k(G)$ or $N(G)$ without reference to any set $\cB$ of bottlenecks, we set $\cB$ to be the set of all bottlenecks in $G$. 

\begin{theorem}\label{BottleneckNestedSet}
    Let $G$ be a locally finite, connected graph. 
    Given any set $\cB$ of bottlenecks in $G$, the sets $N(G)$ and $N^{<k}(G)$ (for $k\in\N$) are well-defined and nested, and $N^k(G,\beta)$ is non-empty for every $k\in\N$ and every $k$-bottleneck $\beta$ in~$G$.
    Moreover, if~$\cB$ is $\Gamma$-invariant for a subgroup~$\Gamma$ of $\Aut(G)$, then~$N(G)$ and $N^{<k}(G)$ (for $k\in\N$) are $\Gamma$-canonical.
\end{theorem}

\noindent Here, a set~$\cB$ of bottlenecks is \defn{$\Gamma$-invariant} for a subgroup~$\Gamma$ of $\Aut(G)$ if the action of~$\Gamma$ on the bottlenecks of~$G$ (which is inherited from the action of~$\Aut(G)$ on the separations of~$G$) is invariant on~$\cB$.

\begin{lemma}\label{CrossingNumberFinite}
    Let $G$ be a locally finite, connected graph and $k\in\N$. 
    Then every tight finite-order separation of $G$ is crossed by only finitely many tight separations of $G$ of order at most~$k$.
\end{lemma}
\begin{proof}
    This follows from a lemma of Halin \cite[2.4]{halinlatticescuts} as in the proof of \cite[Proposition~6.2]{infinitesplinter}.
\end{proof}

\begin{lemma}\label{ThreeGreenCorners}
    Let $\beta_i$ be $k_i$-bottlenecks for $i=1,2$ in a locally finite, connected graph~$G$ (possibly $\beta_1=\beta_2$).
    If $s_1\in\beta_1$ crosses $s_2\in\beta_2$, then at least three corners have order~$\le \max\{k_1,k_2\}$.
\end{lemma}

\begin{proof}
    The short proof is analogue to that of \cite[Sublemma~3.6]{entanglements}, but we include it for convenience.
    Without loss of generality we have $k_1\le k_2$.
    Suppose for a contradiction that at most two corners of $s_1$ and $s_2$ have order~$\le k_2$.
    By submodularity~\cref{submod}, at least one of every two opposite corners has order~$\le k_2$.
    So there are adjacent corners $c,d$ of order~$\le k_2$.
    Let $c'$ and $d'$ be the corners opposite of $c$ and~$d$, respectively.
    By assumption, $c'$ and $d'$ have order~$>k_2$.
    Then $c$ and $d$ have order~$<k_1$ by submodularity~\cref{submod}.
    Let $i$ be the index for which $c$ and $d$ lie on the same side of~$s_i$.
    Then $c,d$ and $s_i$ are the constituents of a relevant \tstar\ with base $s_i$ of order~$\le k_i$ by \cref{CrossingTightSeparationsFormRelevantYStar}.
    Since $s_i\in\beta_i$, \cref{Entangle1} yields that either $c$ or $d$ is contained in~$\beta_i$, say $c\in\beta_i$.
    But then $c$ has order equal to~$k_i$, contradicting that, as shown above, $c$ has order~$<k_1$.
\end{proof}

\begin{lemma}\label{MkGnonempty}
    For every $k$-bottleneck $\beta$ in any given set $\cB$ of bottlenecks in a locally finite, connected graph $G$, the set $N^k(G,\beta)$ is non-empty.
\end{lemma}
\begin{proof}
    It suffices to show for $k>0$ that $\beta$ contains separations that are nested with $N^{<k}(G)$.
    Let $s\in\beta$ cross as few separations in $N^{<k}(G)$ as possible.
    \cref{CrossingNumberFinite} ensures that these are only finitely many.
    Assume for a contradiction that $s$ crosses some $s'\in N^{<k}(G)$.
    Then at least three corners have order~$\le k$ by \cref{ThreeGreenCorners}.
    Thus, two corners on the same side of $s$ have order~$\le k$.
    By \cref{CrossingTightSeparationsFormRelevantYStar}, these two corners and $s$ form a relevant \tstar\ $\sigma$ based at~$s$.
    Hence, \cref{Entangle1} applied to $\sigma$ yields a corner~$c\in\beta$.
    The corner $c$ crosses strictly fewer separations in $N^{<k}(G)$ than $s$ by~\cref{item:F1} and the fact that $N^{<k}(G)$ is nested.
\end{proof}

\begin{lemma}
    Let $G$ be a locally finite, connected graph.
    Let $s,s'$ be two crossing separations in tight $k$-bottlenecks in~$G$, and let $c,d$ be two opposite corners of~$s,s'$ that are contained in $\bigcup\cB^k(G)$.
    Then 
    \begin{equation}\label{xksubmod}
        x^k(c)+x^k(d)<x^k(s)+x^k(s').
    \end{equation}
\end{lemma}
\begin{proof}
    This follows from \cref{Fish}, as in the proof of \cite[Corollary~3.3]{entanglements} or \cite[Lemma~12.5.6]{bibel}.
\end{proof}

\begin{proof}[Proof of \cref{BottleneckNestedSet}]
    \cref{CrossingNumberFinite} ensures that $x^k$ in \cref{construction:inductiveNestedSet} assigns finite values.
    The sets $N^k(G,\beta)$ are non-empty by \cref{MkGnonempty}.
    To show that $N(G)$ and the sets $N^{<k}(G)$ are nested, it suffices to show that $N^k(G)$ is nested for every~$k\in\N$.
    Suppose for a contradiction that there are two separations $s,s'\in N^k(G)$ that cross.
    Then $s \in N^k(G,\beta)$ and $s' \in N^k(G,\beta')$ for some $k$-bottlenecks $\beta,\beta'$ in~$G$.
    Since $s$ and $s'$ are nested with $N^{<k}(G)$, \cref{item:F1} implies that
    \begin{equation}\label{cornersGood} 
        \textit{all four corners of }s,s'\textit{ are nested with }N^{<k}(G).
    \end{equation}
    
    By \cref{ThreeGreenCorners}, at least three corners have order~$\le k$.
    By symmetry, we may assume without loss of generality that the corners $c,c',d'$ depicted in \cref{fig:CornersToT} have order~$\le k$.

    \begin{figure}[ht]
        \centering
        \includegraphics[height=6\baselineskip]{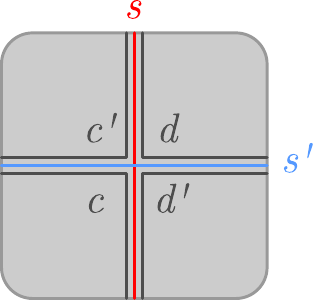}
        \caption{The four corners in the proof of \cref{BottleneckNestedSet}}
        \label{fig:CornersToT}
    \end{figure}

    \casen{Case~1: The corner $d$ has order~$>k$.}
    By \cref{CrossingTightSeparationsFormRelevantYStar} and \cref{Entangle1}, $\beta$ meets $\{c,c'\}$ and $\beta'$ meets $\{c,d'\}$.
    If $c$ is contained in $\beta$ or~$\beta'$, then $c$ has order~$k$, and then $d$ has order~$\le k$ by submodularity~\cref{submod} contradicting our assumption; hence, $c$ lies in neither $\beta$ nor~$\beta'$.
    We therefore get $c'\in\beta$ and $d'\in\beta'$.
    But $x^k(c')+x^k(d')<x^k(s)+x^k(s')$ from \cref{xksubmod} implies that $x^k(c')<x^k(s)$ or $x^k(d')<x^k(s')$, say $x^k(c')<x^k(s)$.
    This together with \cref{cornersGood} shows that $c'$ contradicts the fact that $s$ lies in $N^k(G,\beta)$.

    \casen{Case~2: The corner $d$ has order~$\le k$.}
    Applying \cref{CrossingTightSeparationsFormRelevantYStar} and \cref{Entangle1} to both sides of both $s \in \beta$ and $s' \in \beta'$ yields that $\beta$ meets $\{c,c'\}$ and $\{d,d'\}$ and $\beta'$ meets $\{c,d'\}$ and $\{c',d\}$.
    If for a pair of opposite corners one is in $\beta$ and the other in~$\beta'$, then we obtain a contradiction from \cref{xksubmod} as in Case~1.
    So we may assume that this does not happen.
    By symmetry, we may assume without loss of generality that $c\in\beta$, and hence $d \notin \beta$.
    Since $\beta'$ meets $\{c',d\}$, we have $c'\in\beta'$, which in turn gives $d\in\beta$, and finally $d'\in\beta'$.
    By symmetry, we may assume without loss of generality $x^k(s)\le x^k(s')$.
    Then $x^k(c')+x^k(d')<x^k(s)+x^k(s')$ from \cref{xksubmod} implies that $x^k(c')<x^k(s')$ or $x^k(d')<x^k(s')$.
    This in combination with \cref{cornersGood} shows that one of $c',d'\in\beta'$ contradicts the fact that $s'$ lies in $N^k(G,\beta')$.

    The `moreover'-part is immediate from the canonicity of~\cref{construction:inductiveNestedSet}.
\end{proof}

A \defn{tree of tangles} $N$ of $G$ of \defn{scope~$k$} is a nested set of separations of $G$ which contains, for every pair of $\lek$-distinguishable tangles $\tau_1,\tau_2$ in $G$, a separation that efficiently distinguishes $\tau_1$ and~$\tau_2$.
A \defn{tangle tree-decomposition} $\cT$ of $G$ of \defn{scope~$k$} is a tree-decomposition of~$G$ such that the set of induced separations of~$\cT$ is a tree of tangles of $G$ of scope~$k$.

\begin{theorem}\label{BottlenecksImpliesBoundedToT}
    Let $G$ be a locally finite, connected graph, and let $k\in\N$.
    Then $N^{\le k}(G)$ is a canonical tree of tangles of $G$ of scope~$k$.
    Moreover, $\cT(N^{\le k}(G))$ is a canonical tangle tree-decomposition of $G$ of scope~$k$.
\end{theorem}
\begin{proof}
    Apply \cref{BottleneckNestedSet} with $\cB$ the set of all bottlenecks in $G$.
    Then \cref{TanglesInduceBottleneck} yields that $N^{\le k}(G)$ is a canonical tree of tangles.
    By \cref{lemma:candidateisTDtightbounded}, $\cT(N^{\le k}(G))$ is a canonical \td .
\end{proof}

A tree of tangles that has scope~$k$ for every $k\in\N$ has \defn{infinite} scope.
Likewise, a tangle tree-decomposition that has scope~$k$ for every $k\in\N$ has \defn{infinite} scope.

\begin{theorem} \label{BottlenecksImpliesToT}
    Let $G$ be finite connected graph.
    Then $N(G_r)$ is a canonical tree of tangles of the $r$-local covering~$G_r$ of infinite scope.
    Moreover, $\cT(N(G_r))$ is a canonical tangle tree-decomposition of $G_r$ of infinite scope.
\end{theorem}
\begin{proof}
    Again apply \cref{BottleneckNestedSet} with $\cB$ the set of all bottlenecks in $G$.
    Then \cref{TanglesInduceBottleneck} yields that $N(G_r)$ is a canonical tree of tangles.
    To see that~$\cT(N(G_r))$ is indeed a canonical \td, we combine \cref{lemma:candidateisTDtightbounded} with the following results of \cite{canonicalGraphDec}: Theorem~3, Theorem~6.2, Lemma~6.3 and Lemma~8.1.
\end{proof}

\section{Local separations}\label{sec:LocalSeparations}

In this section, we define \emph{$r$-local separations} formally, along with notions of \emph{$r$-local separators} and \emph{$r$-local components}.
Ideally, we would like the $r$-local notions to be compatible with the perspective of $r$-local coverings in two ways:
On the one hand, each $r$-local notion should generalise its global counterpart so that both notions are equivalent in all $r$-local covers.
On the other hand, each $r$-local notion should get along with every $r$-local covering $p_r\colon G_r\to G$ in that we can move in between the $r$-local separations (say) of $G$ and the $r$-local separations of~$G_r$ by means of lifting or projecting.

Our notions will not satisfy these properties in general, but we shall find conditions under which they do and which cover all relevant cases.
The equivalences between the $r$-local notions and their global counterparts in $r$-local covers are shown in~\cref{prop:tightseporsaretightOsepors,keylemma:correspondenceofcomponents,correspondenceofseparationswithfixedseparator}.
\cref{keylemma:projectionofcomponents,keylemma:liftofcomponents} show how $r$-local components can be lifted and projected, respectively, between graphs their $r$-local covers.
\cref{ProjectionLocalSeparation,keylemma:liftofseparations} achieve this for $r$-local separations, which includes $r$-local separators as a special case.

\subsection{Local separations and their counterparts in the local cover}

Let $G$ be a connected graph and $\emptyset \neq X\se V(G)$.
Recall that $\defnMath{\partial_G(X)}$ or $\defnMath{\partial X}$ denotes the set of all edges in $G$ with precisely one end in~$X$.
The cuts of the form $E_G(C,X)$ for a component $C$ of $G-X$ are the \defn{componental cuts at~$X$} in~$G$.
As $G$ is connected, the components $C$ of $G-X$ partition $\partial X$ into the componental cuts $E_G(C,X)$.
This partition is also defined by the equivalence relation on $\partial X$ where we declare two edges $e,f\in \partial X$ to be equivalent if there exists an $X$-walk (see \cref{subsec:localcov}) in $G$ which starts with $e$ and ends with~$f$.
The notion of \emph{$r$-local components} below is motivated by this equivalence relation.

An \defn{$X$-edge} is an edge with both endvertices in~$X$; note that $X$-edges are not in $\partial X$.
An \defn{$r$-local $X$-walk} in $G$ is an $X$-walk in some short cycle in~$G$, one of length $\leq r$, or an $X$-walk consisting of a single $X$-edge; an $r$-local $X$-walk which is a path is called an \defn{$r$-local $X$-path}.
We remark that a non-trivial reduced $r$-local $X$-walk is either an $r$-local $X$-path or a walk once around a short cycle that meets $X$ in precisely its base vertex.
The \defn{$r$-local components at~$X$} in $G$ are the equivalence classes of the transitive closure of the relation on $\partial X$ defined by letting
\medskip
\begin{center}
    \defn{$e\sim_r f$ at $X$ in $G$} $:\Leftrightarrow$ $e$ and $f$ are the first and last edge on the same $r$-local $X$-walk in~$G$, or $e=f$.
\end{center}
\medskip
The $r$-local components at $X$ thus further partition the componental cuts at~$X$.
Every $r$-local $X$-walk meets every $r$-local component at $X$ in precisely its first and last edge or in none, and thus
\begin{equation}\label{eq:LocalCompsShortCycles} 
    \text{every short cycle $O$ meets every $r$-local component at $X$ in an even number of edges.}
\end{equation}
We remark that, if a short cycle $O$ meets $X$ in two vertices, then $O$ splits into edge-disjoint $r$-local $X$-paths.

We call $X\se V(G)$ an \defn{$r$-local separator} of $G$ if there are at least two $r$-local components at~$X$.
For the right intuition, note that the qualification of being `$r$-local' refers to the way in which $X$ `separates' (namely, only locally), not to the position of $X$ in $G$. 
In particular, every separator of $G$ is also an $r$-local separator.
\cref{fig:OcomponentsExample} shows an example of an $r$-local separator that is not a (global) separator.
The size of an $r$-local separator, which we note cannot be zero, is its \defn{order}.
An $r$-local separator of order~$k$ is also called an \defn{$r$-local $k$-separator.}

\begin{figure}[ht]
    \centering
    \includegraphics[height=6\baselineskip]{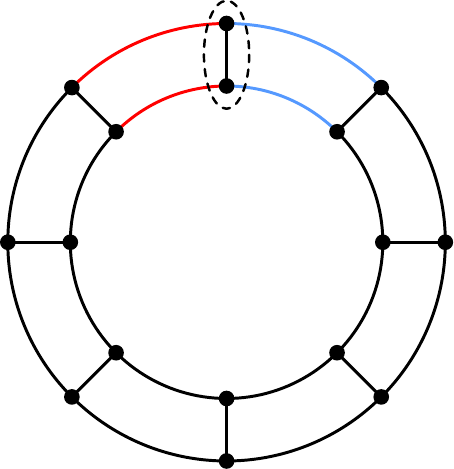}
    \caption{There are two $4$-local components (red and blue) at $X$ (encircled)}
    \label{fig:OcomponentsExample}
\end{figure}

Our aim in introducing local separations was that, for every $r$, the relevant $r$-local $k$-separations of $G$ should lift to (global) $k$-separations of $G_r$.
Before we discuss lifting properties, let us show that the $r$-local separations of $G_r$ (to which we hope our $r$-local separation of $G$ will lift) are indeed (global) separations of $G_r$ -- as long as they are relevant in the following natural sense of being 'locally tight'.

Let us call an $r$-local component $F$ at $X$ \defn{tight} if every vertex in~$X$ is incident with some edge in~$F$, and let us call an $r$-local separator \defn{tight} if there are at least two tight $r$-local components at~$X$.
Note that although every separator is also an $r$-local separator, it can be tight as a separator but not as an $r$-local separator, see \cref{fig:TightGlobalNotLocal}. 
Conversely, an $r$-local separator need not be a separator, but if it is, then the tightness of the $r$-local separator clearly implies the tightness of the separator.
In the $r$-local cover these notions coincide:

\begin{figure}[ht]
    \centering
    \includegraphics[height=6\baselineskip]{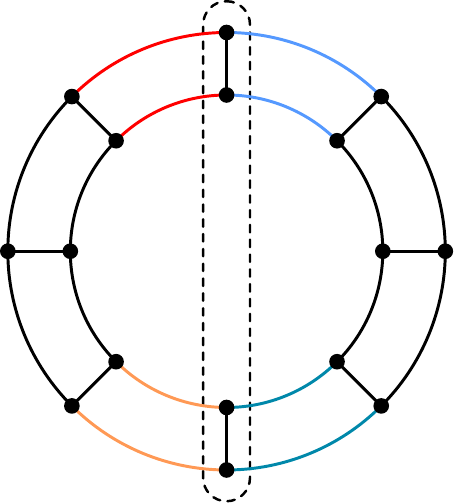}
    \caption{The encircled vertex set $X$ is a tight separator of~$G$. However, there are four $4$-local components at~$X$, none of which is tight. So while $X$ is a $4$-local separator, it is not a tight one.}
    \label{fig:TightGlobalNotLocal}
\end{figure}

\begin{corlemma}[Tight separator]\label{prop:tightseporsaretightOsepors}
    Assume that the binary cycle space of a connected graph $G$ is generated by cycles of length~$\le r$.
    A set $X\se V(G)$ is a tight separator of $G$ if and only if it is a tight $r$-local separator of~$G$.
\end{corlemma}

The proof requires some preparation.
For a vertex set $\emptyset \neq X\se V(G)$ we denote by \defnMath{$\cW_r(X)$} the set of all walks that are concatenations of $r$-local $X$-paths.
The \defn{$r$-local atoms} of~$X$, or \defn{$r$-toms} for short, are the $\subseteq$-maximal subsets $X'\se X$ such that every two vertices in $X'$ are linked by a walk in~$\cW_r(X)$.
Then $X$ is partitioned by its $r$-toms.
We say that $X$ is \defn{$r$-locally atomic}, or \defn{$r$-tomic} for short, if $X$ has only one $r$-local atom.

\begin{lemma}\label{TightLocalSeparatorAreOConnected}
    Every tight $r$-local separator of a connected graph $G$ is $r$-tomic.
\end{lemma}

\begin{proof}
    This is immediate from the definitions.
\end{proof}

\begin{lemma}\label{TightSeparatorAreOConnected}
    Assume that the binary cycle space of a connected graph $G$ is generated by cycles of length~$\le r$.
    Then every tight separator of $G$ is $r$-tomic.
\end{lemma}

\begin{proof}
    Suppose for a contradiction that some tight separator $X$ of~$G$ is not $r$-tomic.
    Fix some $y$ in an $r$-local atom $Y$ of $X$ and some $x \in X \setminus Y$.
    Since $X$ is a tight separator of $G$, there are two tight components $C$ and $C'$ of $G-X$.
    Let $P$ and $P'$ be $x$--$y$ paths through $C$ and~$C'$, respectively. 
    The path $P$ meets the edge set $E := E_G(C,Y)$ precisely once while $P'$ avoids it.
    Thus, the cycle $P\cup P'$ meets $E$ precisely once; in particular, oddly.
    Since the binary cycle space of $G$ is generated by short cycles, there exists a short cycle $O\se G$ that meets $E$ oddly; in particular, $O$ meets $Y$.
    As the cycle $O$ meets the cut $E_G(C,X)$ evenly and $E_G(C,X)$ is the disjoint union of $E = E_G(C,Y)$ and $E_G(C, X \setminus Y)$, the cycle $O$ has an edge $f \in E_G(C, X \setminus Y)$; in particular, $O$ meets $X \setminus Y$.
    Hence, the cycle $O$ of length~$\le r$ meets both $Y$ and $X \setminus Y$, which contradicts the fact that $Y$ is an $r$-local atom of~$X$. 
\end{proof}

\begin{corlemma}[Local components]\label{keylemma:correspondenceofcomponents}
    Assume that the binary cycle space of a connected graph $G$ is generated by cycles of length~$\le r$.
    The $r$-local components at an $r$-tomic set $X\se V(G)$ are precisely the componental cuts at~$X$.
\end{corlemma}

\begin{proof}
    By definition, every $r$-local component at~$X$ is included in a componental cut at~$X$.
    Conversely, let $C$ be a component of $G-X$ with non-empty componental cut.
    Let $e,f$ be $C$--$X$ edges in $G$ and let $E$ be the $r$-local component at $X$ containing~$e$.
    If $E$ always contains $f$ as well, we are done; so we assume for a contradiction that $f \notin E$.
    Using the component~$C$, we find an $X$-walk $P$ that starts in $X$ and traverses the edge $e$ first, then uses vertices and edges in~$C$, and finally traverses the edge $f$ to end in~$X$; in particular, $P$ passes through $E$ precisely once, and thus oddly often.
    Since $X$ is $r$-tomic, there is a walk $Q$ in $G$ which joins the endvertices of $e,f$ in $X$ and is a concatenation of $r$-local $X$-paths.
    Since $r$-local $X$-paths traverse $r$-local components at~$X$ evenly, so do concatenations of them, such as the walk~$Q$.
    Now concatenating the walks $P$ and $Q$ yields a closed walk $W$ which passes through $E$ oddly often.
    Consider the element $W'$ of the binary cycle space induced by the closed walk $W$: an edge of $G$ maps to $0$ if $W$ passes it evenly often, and otherwise, the edge maps to $1$; in particular, the values of $W'$ on the edges in $E$ sum to an odd number. 
    Since the binary cycle space of $G$ is generated by short cycles, $W'$ is the symmetric difference of the edges sets of short cycles $O_1, \dots, O_n$.
    Hence, some short cycle $O_i$ meets $E$ oddly, and therefore contradicts~\cref{eq:LocalCompsShortCycles}.
\end{proof}

\begin{proof}[Proof of \cref{prop:tightseporsaretightOsepors}]
    Let $X\se V(G)$.
    If $X$ is a tight separator, or if $X$ is a tight $r$-local separator, then $X$ is $r$-tomic by \cref{TightSeparatorAreOConnected} or \cref{TightLocalSeparatorAreOConnected}, respectively.
    Hence, \cref{keylemma:correspondenceofcomponents} completes the proof.
\end{proof}

\begin{lemma}\label{relevantTriStarRtomicInCovering}
    Assume that the binary cycle space of $G$ is generated by cycles of length $\leq r$ and that $G$ is connected.
    Let $\sigma=\{\,(A_i,B_i) : i\in [3]\,\}$ be a relevant \tstar\ in $G$ with base~$\{A_1,B_1\}$.
    Then $X:=\bigcup_{i\in [3]}X_i$ is $r$-tomic.
\end{lemma}
\begin{proof}
    By \cref{TightSeparatorAreOConnected}, $X_1$ is $r$-tomic.
    Let $Y$ be the $r$-tom of $X$ which contains~$X_1$.
    Suppose for a contradiction that there is some $x \in Y':= X \setminus Y = X_{23}\sm Y$.
    Consider $F := E(Y', X_{12} \cup (A_2 \setminus B_2))$.
    
    We claim that there is a (not necessarily short) cycle $O'$ in $G$ which meets $F$ precisely once.
    For $i \in \{2,3\}$, let $P_i$ be an $x$--$(X_{1i} \cup Z)$ path through $A_i \sm B_i$ or a single $x$--$X_{1i}$ edge, given by the relevance of $\sigma$. 
    Since $\{A_1,B_1\}$ is tight, we may join the endvertices that these $P_i$ have in $X_1$ by a path through $A_1 \sm B_1$ to obtain the desired cycle $O'$ in~$G$.
    
    As the binary cycle space of $G$ is generated by short cycles, we infer from the existence of $O'$ that there exists a short cycle $O\se G$ which meets $F$ oddly.
    If $O$ meets $F$ in an $X_{23}$--$X_{12}$ edge, then this $X$-edge witnesses that some vertex in $Y'$ is in a common $r$-tom with some vertex in $X_{12} \subseteq Y$, which contradicts the maximality of $Y$ as an $r$-tom.
    Thus, we may assume that $O$ does not meet $F$ in an $X_{23}$--$X_{12}$ edge.
    Hence, $O$ meets $E(Y', A_2 \setminus B_2)$ oddly.
    Since $X_2=(X_2\cap Y)\sqcup Y'$ and $X_2$ is the separator of the separation $(A_2,B_2)$,
    the cut $E(X_2,A_2\sm B_2)$ is bipartitioned by $E(X_2\cap Y,A_2\sm B_2)$ and $E(Y',A_2\sm B_2)$.
    Since $O$ meets the cut $E(X_2,A_2\sm B_2)$ evenly, it must meet $E(X_2\cap Y,A_2\sm B_2)$.
    Thus, $O$ witnesses that a vertex in $X_2 \cap Y$ and a vertex in $X_2 \setminus Y$ are in a common $r$-tom, which contradicts the maximality of $Y$ as an $r$-tom of~$X$.
\end{proof}

A separation could alternatively be defined as a triple $\{U,X,W\}$ such that $U\sqcup X\sqcup W=V(G)$ and every component of $G-X$ has its vertex set included in $U$ or in~$W$.
We use this perspective to introduce local separations.

An \defn{$r$-local separation} of~$G$ is a triple $\{E_1,X,E_2\}$ where $X$ is a non-empty set of vertices of~$G$ and $E_1,E_2$ near-bipartition $\partial X$ such that every $r$-local component at~$X$ is included in $E_1$ or in~$E_2$.
We refer to $X$ as its \defn{separator}, to the two classes $E_1,E_2$ as its \defn{sides} and to $|X|$ as its \defn{order}.
Also we call $\{E_1,X,E_2\}$ an \defn{$r$-local $k$-separation} for $k=|X|$.
The ordered triples $(E_2,X,E_1)$ and $(E_1,X,E_2)$ are the \defn{orientations} of $\{E_1,X,E_2\}$ \defn{towards} $E_1$ and \defn{towards}~$E_2$, respectively.
We also call these \defn{oriented $r$-local separations}, or just \defn{$r$-local separations} by a slight abuse of notation.
The \defn{inverse} $(E_2,X,E_1)$ of an oriented $r$-local separation $(E_1,X,E_2)$ is denoted by \defnMath{$(E_1,X,E_2)^*$}.
An $r$-local separation with separator $X$ is \defn{tight} if each side includes a tight $r$-local component at~$X$.

Every oriented separation $(A_1,A_2)$ of $G$ \defn{induces} the oriented $r$-local separation $(E_1,X,E_2)$ with $X:=A_1\cap A_2$ and $E_i:=E_G(X,A_i\sm X)$.
Then we also say that $\{A_1,A_2\}$ \defn{induces} $\{E_1,X,E_2\}$.
If $(A_1,A_2)$ or $\{A_1,A_2\}$ is introduced as~$s$, then we denote the (oriented) $r$-local separation that $s$ induces by~$\defnMath{s_r}$.
If an $r$-local separation is introduced as $\{E_1,X,E_2\}:=s_r$ when $\{A_1,A_2\}=s$ was previously given, we tacitly assume that $E_i=E_G(X, A_i\sm X)$ and not $E_i=E_G(X,A_{3-i}\sm X)$.
We remark that by definition $s\mapsto s_r$ is an injective map from the (oriented) separations of $G$ to the (oriented) $r$-local separations of~$G$.

\begin{corlemma}[Separations]\label{correspondenceofseparationswithfixedseparator}
    Assume that the binary cycle space of $G$ is generated by cycles of length~$\le r$ and that $G$ is connected.
    The map $s\mapsto s_r$ is a bijection between the (oriented) separations of $G$ with $r$-tomic separator and the (oriented) $r$-local separations of $G$ with $r$-tomic separator.
\end{corlemma}
\begin{proof}
    This is immediate from \cref{keylemma:correspondenceofcomponents}.
\end{proof}

\begin{example}
    The assumption in \cref{correspondenceofseparationswithfixedseparator} that the separators are $r$-tomic is necessary.
    For example, let $G$ be the graph depicted in \cref{fig:LocalFlip} and $r:=4$.
    The binary cycle space of $G$ is generated by cycles of length~$\le r$.
    Let $X$ denote the set of encircled vertices, which is not $r$-tomic.
    Let $E_1,E_2$ consist of the red edges and the blue edges, respectively.
    Then $\{E_1,X,E_2\}$ is an $r$-local separation of~$G$ that is not induced by a separation of~$G$.
\end{example}

\begin{figure}[ht]
    \centering
    \includegraphics[height=6\baselineskip]{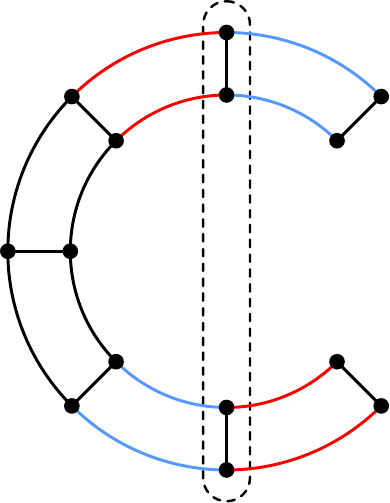}
    \caption{An $r$-local separation that is not induced by a separation, as the sides of the $r$-local separation restricted at the $r$-local atoms of its separator are flipped.}
    \label{fig:LocalFlip}
\end{figure}

\begin{corlemma}[Tight separations]\label{keylemma:correspondenceofseparations}
    Assume that the binary cycle space of $G$ is generated by cycles of length~$\le r$ and that $G$ is connected.
    Then the map $s\mapsto s_r$ restricts to a bijection between the tight (oriented) separations of $G$ and the tight (oriented) $r$-local separations of~$G$.
\end{corlemma}

\begin{proof}
    If $s$ or $s_r$ is tight, then their separator is $r$-tomic by \cref{TightSeparatorAreOConnected} or \cref{TightLocalSeparatorAreOConnected}, so the other of~$s,s_r$ is tight as well by \cref{keylemma:correspondenceofcomponents}.
    Thus, $s \mapsto s_r$ is a bijection from the tight (oriented) separations to the tight (oriented) $r$-local separations by \cref{correspondenceofseparationswithfixedseparator}.
\end{proof}

\subsection{Interplay with the covering map}\label{subsec:interplayWithCov}

Let $X$ be a vertex set of $G$.
If $X$ is $r$-tomic, then the deck transformations of $p_r$ act transitively on the set of $r$-toms of $p_r\inv (X)$ in $G_r$ by $Y \mapsto \gamma(Y)$ for any deck transformation $\gamma$.
In \cref{fig:rtomicExample}, the set $X$ of two vertices of $G$ is $4$-tomic. It's preimage $p_4\inv(X)$ has infinitely many $4$-toms in the $4$-local cover $G_4$ of~$G$.

\begin{figure}[ht]
    \centering
    \includegraphics[width=.9\textwidth]{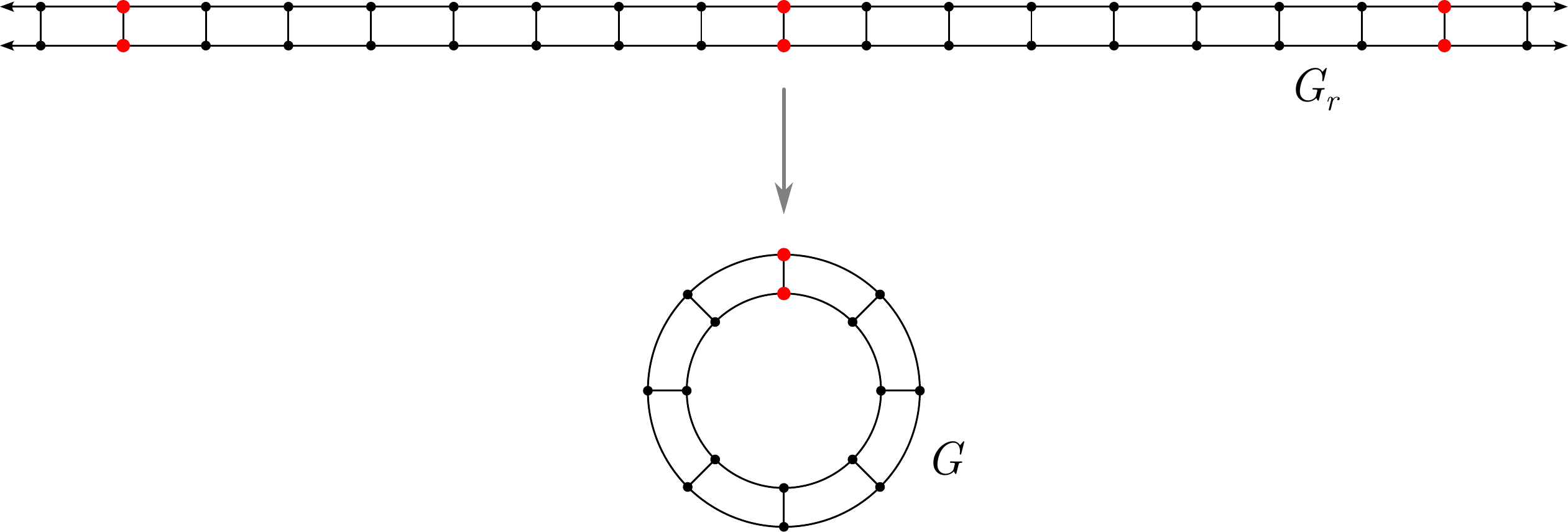}
    \caption{The set $X$ is red in $G$ (bottom) and $p_4\inv (X)$ is red in $G_4$ (top).}
    \label{fig:rtomicExample}
\end{figure}

\begin{projlemma}[Locally atomic]\label{ProjectionRtomic}
    Let $\hat X$ be a vertex set of $G_r$.
    If $\hat X$ is $r$-tomic, then its projection $p_r(\hat X)$ is $r$-tomic.
\end{projlemma}

\begin{proof}
    Let $\hat x$ and $\hat y$ be two arbitrary vertices in $\hat X$ that are linked by an $r$-local $\hat X$-path $\hat W$ in $G_r$.
    Since $p_r$ maps short cycles of $G_r$ isomorphically to short cycles in $G$,
    the projection $W := p_r(\hat W)$ is a concatenation of $r$-local $X$-paths in $G$ that joins $x := p_r(\hat x)$ to $y := p_r(\hat y)$.
    As we have chosen $\hat x,\hat y$ arbitrarily in $\hat X$, it follows immediately that $p_r(\hat X)$ is $r$-tomic.
\end{proof}

\begin{projlemma}[Local atoms]\label{ProjectionLocalAtom}
    Let $X$ be a vertex set of $G$.
    Then $p_r$ maps the $r$-toms of $\hat X := p_r\inv(X)$ to the $r$-toms of $X$.
    In particular, if $X$ is $r$-tomic, then $p_r$ maps the $r$-toms of $\hat X$ onto~$X$.
\end{projlemma}
\begin{proof}
    Let $\hat Y$ be an $r$-tom of $\hat X$.
    By \cref{ProjectionRtomic}, its projection $Y := p_r(\hat Y)$ is $r$-tomic.
    Suppose $Y$ is not an $r$-tom of $X$, i.e.\ there is an $r$-local $X$-path $W$ in $G$ that joins a vertex $y \in Y$ to a vertex $x \in X \setminus Y$.
    Since any short cycle of $G$ lfits to a short cycle of the $r$-local cover $G_r$, the $r$-local $X$-path $W$ lifts to an $r$-local $\hat X$-path $\hat W$, as $\hat X = p_r\inv(X)$.
    But $\hat W$ witnesses that $\hat Y$ is not an $r$-tom of $\hat X$, which is a contradiction.
\end{proof}

Let $\hat X$ be a set of vertices of~$G_r$ and $X:=p_r(\hat X)$.
The vertex set $\hat X$ is \defn{$r$-locally closed} (with respect to~$p_r$) if no $r$-local $p_r\inv(X)$-path links a vertex in $\hat X$ to a vertex in $p_r\inv (X) \setminus \hat X$.
We remark that if $\hat X$ is $r$-locally closed, then $p_r$ maps edges of $\partial \hat X$ to $\partial X$ (rather than to $X$-edges).

\begin{observation}\label{LocalatomsAreClopen}
    For every $X\se V(G)$, the $r$-local atoms of $p_r\inv (X)$ are its $\subseteq$-minimal non-empty $r$-locally closed subsets.
    More generally, $\hat X \subseteq V(G_r)$ is $r$-locally closed if and only if it is a union of $r$-local atoms of $p_r\inv(p_r(\hat X))$.
\end{observation}

\begin{liftlemma}[$e\sim_r f$ at $X$ in $G$]\label{localEdgeRelationLift}
    Given a connected graph $G$, let $\hat X$ be an $r$-locally closed set of vertices of~$G_r$, and let $X:=p_r(\hat X)$.
    The following assertions are equivalent for every two edges $e,f\in\partial X$:
    \begin{enumerate}
        \item $e\sim_r f$ at $X$ in~$G$;
        \item for every lift $\hat e\in \partial \hat X$ of~$e$ there exists a lift $\hat f\in \partial\hat X$ of~$f$ such that $\hat e\sim_r \hat f$ at~$\hat X$ in~$G_r$;
        \item there exist lifts $\hat e,\hat f\in \partial\hat X$ of $e$ and~$f$, respectively, such that $\hat e\sim_r \hat f$ at~$\hat X$ in~$G_r$.
    \end{enumerate}
\end{liftlemma}
\begin{proof}
    (i)$\to$(ii).
    Let $e \sim_r f$ at $X$ in $G$ and let $\hat e\in\partial\hat X$ be a lift of~$e$.
    If $e = f$, then we set $\hat f := \hat e$.
    Otherwise, there is an $r$-local $X$-walk $W$ that traverses $e$ first and $f$ last. 
    Let $O$ be a short cycle in~$G$ that contains~$W$.
    Let $\hat O$ be the lift of $O$ that contains~$\hat e$(which is a short cycle in $G_r$), and let $\hat W$ be the lift of $W$ that starts with~$\hat e$, which is thus a walk in $\hat O$.
    Let $\hat f$ be the lift of $f$ that is the last edge of~$\hat W$.
    As $\hat X$ is $r$-locally closed, $\hat W$ cannot end in $p_r\inv(X) \setminus \hat X$: it is an $r$-local $\hat X$-walk that traverses $\hat e$ first and $\hat f$ last.
    Thus, $\hat e \sim_r \hat f$ at~$\hat X$ in $G_r$.
    
    (ii)$\to$(iii) is straightforward, as every edge $e\in\partial X$ lifts to some edge $\hat e\in\partial \hat X$.
    
    (iii)$\to$(i).
    Since $p_r$ maps short cycles of $G_r$ isomorphically to short cycles of $G$, the $r$-local $\hat X$-walk witnessing $\hat e\sim_r\hat f$ at $\hat X$ in $G_r$ projects to an $r$-local $X$-walk witnessing $e\sim_r f$ at $X$ in~$G$, as $\hat X$ is $r$-locally closed.
\end{proof}

\begin{projlemma}[Local components]\label{keylemma:projectionofcomponents}
    Given a connected graph $G$, let $\hat X$ be an $r$-locally closed set of vertices of~$G_r$.
    Then every $r$-local component at $\hat X$ projects onto an $r$-local component at $X:=p_r(\hat X)$.
\end{projlemma}
\begin{proof}
    Let $\hat C$ be an $r$-local component at $\hat X$ and put $C:=p_r(\hat C)$.
    Since $\hat X$ is $r$-locally closed, all of $C$ is included in~$\partial X$.
    Then $C$ is included in an $r$-local component $C'$ at~$X$ by \cref{localEdgeRelationLift} (iii)$\to$(i).
    Conversely, $C'$ is included in $C$ by \cref{localEdgeRelationLift} (i)$\to$(ii), so $C=C'$ as desired.
\end{proof}

\begin{figure}[ht]
    \centering
    \includegraphics[height=14\baselineskip]{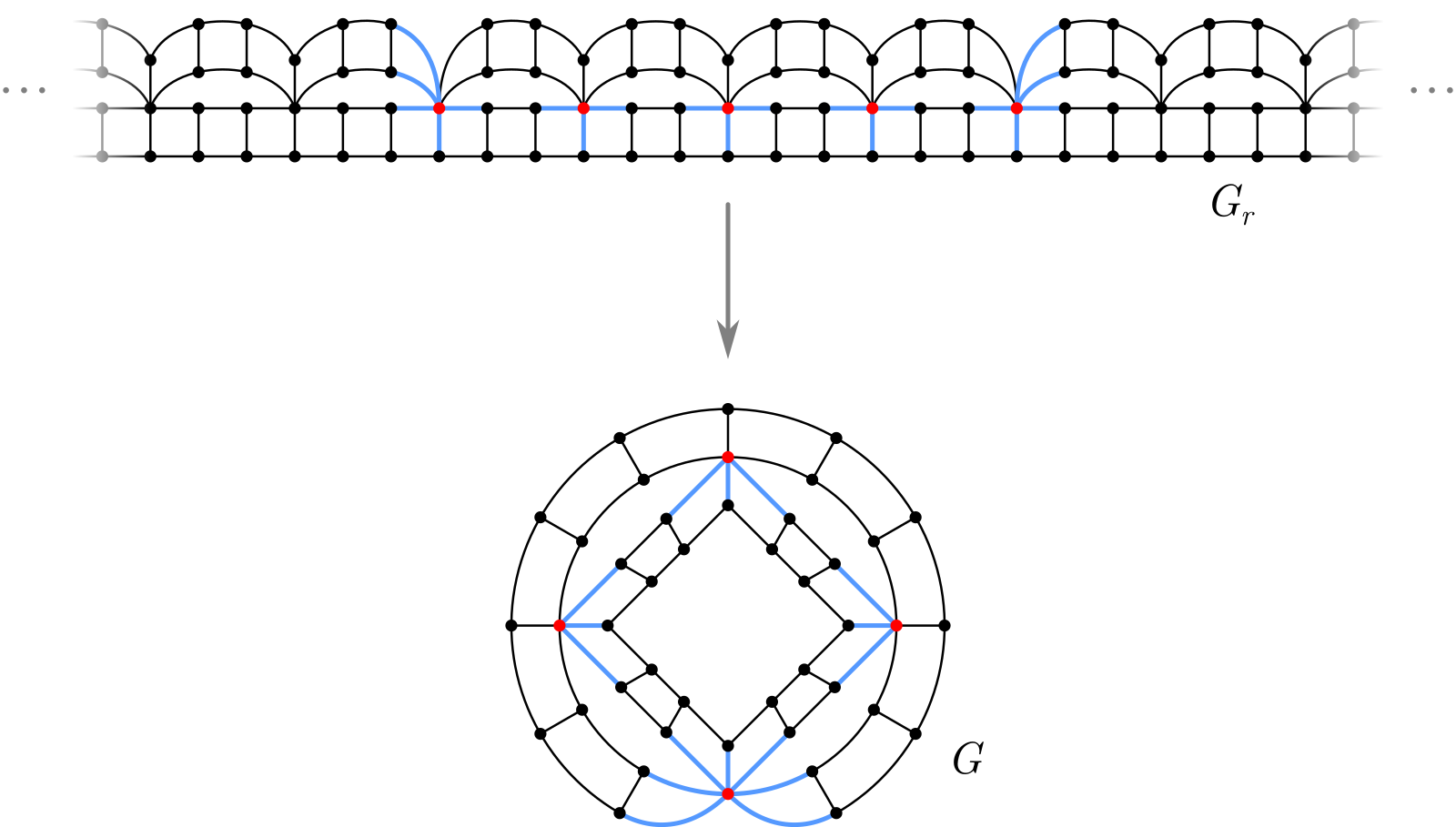}
    \caption{Let $r:=6$. The figure shows an $r$-local component (blue, top) at a non-$r$-locally-closed vertex set $\hat X$ (red) in $G_r$ whose projection to~$G$ (blue, bottom) fails to be an $r$-local component at $p_r(\hat X)$ (red)}
    \label{fig:rlocallyclopennecessary}
\end{figure}

\begin{example}
    The assumption in \cref{keylemma:projectionofcomponents} that $\hat X$ is $r$-locally closed is necessary, see \cref{fig:rlocallyclopennecessary}.
\end{example}

\cref{keylemma:projectionofcomponents} implies that for every $r$-local separation $(E_1,X,E_2)$ of $G$ there is an $r$-local separation $(\hat E_1,\hat X, \hat E_2)$ of $G_r$ such that $p_r(\hat X) = X$ and $p_r(\hat E_i) = E_i$ for $i=1,2$: just take $\hat X := p_r\inv(X)$ and $\hat E_i := p_r\inv(E_i)$ for $i=1,2$.

If $X$ is $r$-tomic, we can find such an $(\hat E_1,\hat X, \hat E_2)$ with $\hat X$ $r$-tomic too: just take as $\hat X$ an $r$-tom of $p_r\inv(X)$ instead of the entire set $p_r\inv(X)$ (cf.\ \cref{ProjectionLocalAtom}), and define $\hat E_i := p_r\inv(E_i) \cap \partial \hat X$ for $i=1,2$, i.e.\ an $r$-local component $\hat C$ at $\hat X$ is in $\hat E_i$ if its projection $p_r(\hat C)$ is in $E_i$.
We have thus shown the following:

\begin{corollary}\label{FirstLiftOfLocSep}
    Every $r$-local separation $(E_1,X,E_2)$ of $G$ is the projection of some $r$-local separation $(\hat E_1,\hat X, \hat E_2)$ of $G_r$.
    If $X$ is $r$-tomic, then $\hat X$ can be chosen as any $r$-tom of $p_r\inv(X)$. \qed
\end{corollary}

Even if the set $X$ in \cref{FirstLiftOfLocSep} is $r$-tomic, the $r$-tomic set $\hat X$ can be much larger than $X$.
\cref{fig:rlocally1sheetednecessary} below shows an example where $\hat X$ is infinite but $X$ is finite.
Also $G_r$ has more $r$-local components at $\hat X$ than $G$ has at $X$ in this example.
Hence, $G_r$ has $r$-local separations with separator $\hat X$ that do not project to $r$-local separations of $G$.

However, it turns out that the latter problem will go away when we solve the former:
if we require that $\hat X$ is no larger than $X$, then $r$-local separations of $G_r$ with separator $\hat X$ do project to $r$-local separations of $G$ (\cref{ProjectionLocalSeparation}).

Recall that a lift of $X\se V(G)$ to $G_r$ is a vertex set $\hat X\se V(G_r)$ such that $p_r$ restricts to a bijection $\hat X\to X$.
The following lemmas consider lifts that are $r$-locally closed.
However, not every $X\se V(G)$ has a lift that is $r$-locally closed.
Sufficient conditions will be discussed in \cref{sec:Displacement}.

\begin{liftlemma}[Local components]\label{keylemma:liftofcomponents}
    Let $X$ be a set of vertices in a connected graph~$G$, and let $\hat X$ be an $r$-locally closed lift of~$X$ to~$G_r$.
    \begin{enumerate}
        \item For every $r$-local component $\hat C$ at $\hat X$ in~$G_r$, the covering $p_r$ restricts to a bijection $\hat C \to C$ to an $r$-local component $C$ at~$X$ in~$G$.
        \item The above map $\hat C\mapsto C$ is a bijection between the set of $r$-local components at~$\hat X$ in~$G_r$ and the set of $r$-local components at~$X$ in~$G$.
        \item $\hat C$ is tight (with regard to~$\hat X$) if and only if $C$ is tight (with regard to~$X$), for all~$\hat C$.
    \end{enumerate}
\end{liftlemma}

\begin{proof}
    (i) and (ii).
    Let $\hat C$ be an $r$-local component at~$\hat X$ in~$G_r$.
    As $\hat X$ is $r$-locally closed, \cref{keylemma:projectionofcomponents} yields that $C := p_r(\hat C)$ is an $r$-local component at~$X$ in~$G$.
    The assumption that $\hat X$ is a lift of~$X$ implies by the definition of coverings that the restriction of $p_r$ to~$\hat C$ is also injective, giving~(i), and that these bijections~$\hat C \to C$ satisfy~(ii).
     
    (iii) follows from (i) and (ii) as $\hat X$ is a lift of $X$.
\end{proof}

\begin{figure}[ht]
    \centering
    \includegraphics[height=14\baselineskip]{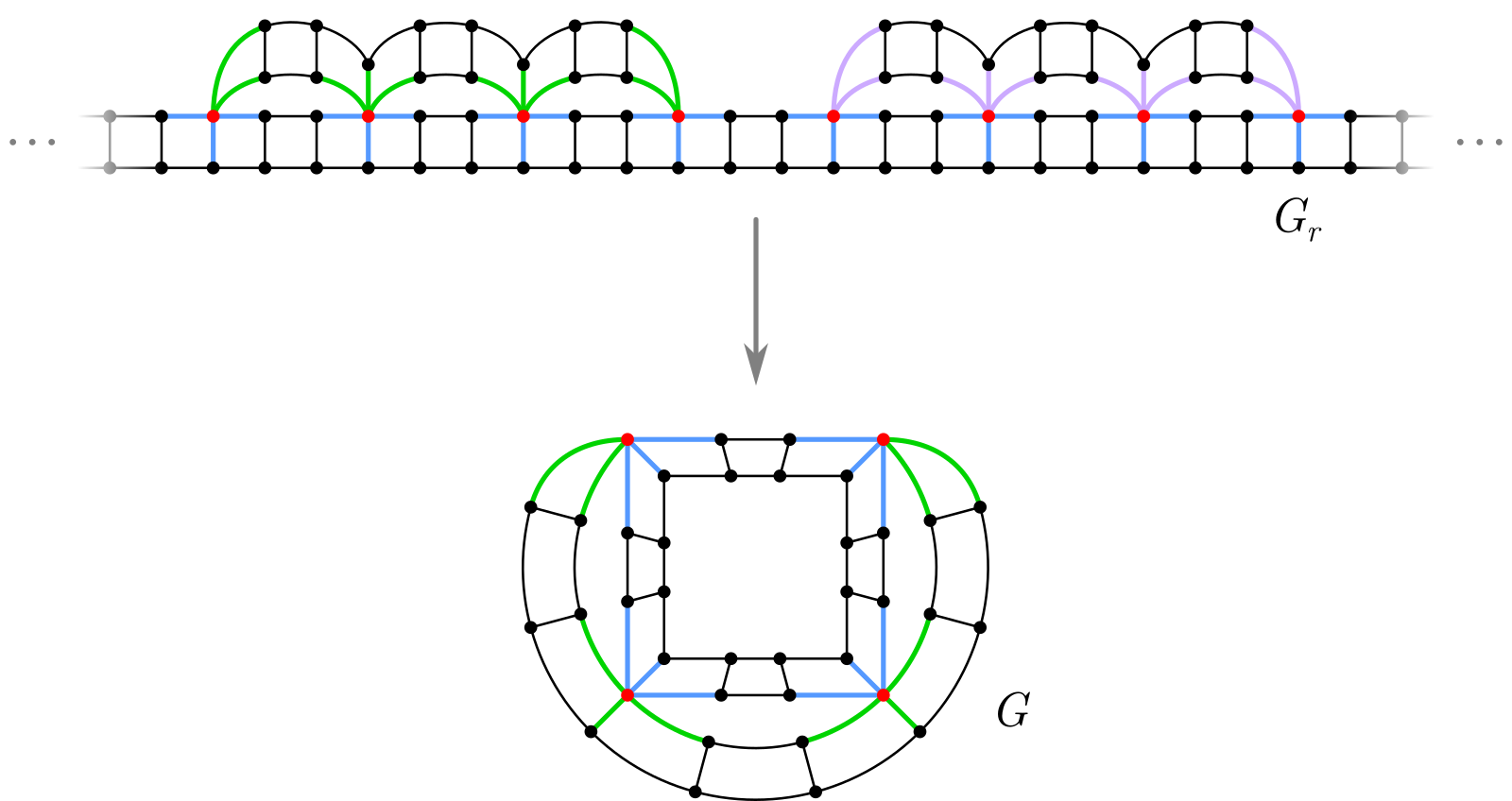}
    \caption{Let $r:=8$.
    The figure shows an $r$-locally closed vertex set $\hat X$ (red, top) in $G_r$ which is not a lift of $X=p_r(\hat X)$ (red, bottom).
    The coloured edge sets in $G_r$ are $r$-local components at~$\hat X$, and the coloured edge sets in $G$ are $r$-local components at~$X$.}
    \label{fig:rlocally1sheetednecessary}
\end{figure}

\begin{example}
    The assumption in \cref{keylemma:liftofcomponents} that $\hat X$ is a lift of $X$ to $G_r$ is necessary.
    For example, let $r:=6$ and consider the graph $G$ in \cref{fig:rlocally1sheetednecessary}.
    Let $X$ be the red set of vertices in~$G$.
    Then $\hat X:=p_r\inv (X)$ is the only non-empty $r$-locally closed subset of $p_r\inv(X)$, but $\hat X$ is not a lift of~$X$.
    The blue $r$-local component at $\hat X$ projects to the blue $r$-local component at~$X$, but not injectively, so (i) of \cref{keylemma:liftofcomponents} fails.
    Both the green and the purple $r$-local component at $\hat X$ project to the green $r$-local component at~$X$, so (ii) fails.
    The green $r$-local component at $\hat X$ is not tight, but its green projection at~$X$ is, so (iii) fails.
\end{example}

\begin{projlemma}[Local separations]\label{ProjectionLocalSeparation}
    Let $X$ be a set of vertices in a connected graph~$G$, and let $\hat X$ be an $r$-locally closed lift of $X$ to~$G_r$.
    Suppose $\hat X$ is the separator of an $r$-local $k$-separation $\{\hat E_1,\hat X,\hat E_2\}$ of $G_r$.
    Then $\{E_1,X,E_2\}$ for $E_i:=p_r(\hat E_i)$ is an $r$-local $k$-separation of~$G$.
    Moreover, $\{E_1,X,E_2\}$ is tight if and only if $\{\hat E_1,\hat X,\hat E_2\}$ is tight.
\end{projlemma}
\begin{proof}
    This follows immediately from \cref{keylemma:liftofcomponents}.
\end{proof}

In the context of \cref{ProjectionLocalSeparation}, when $\{\hat E_1,\hat X,\hat E_2\}$ is introduced as~$\hat s$, we often denote $\{E_1,X,E_2\}$ by~\defnMath{$p_r(\hat s)$} and refer to it as the \defn{projection} of~$\hat s$, and similarly for oriented local separations $(\hat E_1, \hat X, \hat E_2)$.

\begin{liftlemma}[Local separations]\label{keylemma:liftofseparations}
    Let $(E_1,X,E_2)$ be an $r$-local $k$-separation of a connected graph~$G$.
    Let $\hat X$ be an $r$-locally closed lift of $X$ to~$G_r$.
    For $i=1,2$ let $\hat E_i := p_r\inv (E_i)\cap \partial \hat X$.
    Then the following assertions hold:
    \begin{enumerate}
        \item $(\hat E_1, \hat X ,\hat E_2)$ is an $r$-local $k$-separation of~$G_r$.
        \item $p_r$ restricts to a bijection $\hat E_i\to E_i$ for both $i=1,2$.
        \item $(E_1,X,E_2)$ is tight if and only if $(\hat E_1, \hat X, \hat E_2)$ is tight.
    \end{enumerate}
\end{liftlemma}

\begin{proof}
    This follows immediately from \cref{keylemma:liftofcomponents}.
\end{proof}

In the context of \cref{keylemma:liftofseparations}, we call $(\hat E_1, \hat X, \hat E_2)$ the \defn{$\hat X$-lift} of $(E_1, X, E_2)$ to $G_r$.
A \defn{lift} of $(E_1,X,E_2)$ is an $\hat X$-lift of $(E_1,X,E_2)$ for some~$\hat X$ as in~\cref{keylemma:liftofseparations}.
We emphasise that if $\hat X$ is not only $r$-locally closed but even $r$-tomic, then, by \cref{correspondenceofseparationswithfixedseparator}, a lift $\hat s$ of an $r$-local separation $s$ of $G$ is induced by a (global) separation of $G_r$.
Usually, we will later find $\hat X$ as an $r$-tom of $p_r\inv(X)$ via \cref{RtomicGives1sheeted}.

\cref{ProjectionLocalSeparation} and \cref{keylemma:liftofseparations} are inverses of each other in the following sense:
Let $X$ be a set of vertices in a connected graph~$G$ and let~$\hat X$ be an $r$-locally closed lift of~$X$ to~$G_r$.
Then an $r$-local separation of~$G_r$ with separator~$\hat X$ equals the $\hat X$-lift of its projection.
Conversely, an $r$-local separation of~$G$ equals the projection of any of its lifts.

\section{Displacement: a guarantee for the existence of lifts}\label{sec:Displacement}

While~\cref{ProjectionLocalSeparation} and~\cref{keylemma:liftofseparations} explain the interplay of $r$-local separations with the covering map~$p_r$, they both rest on the assumption that the ($r$-local) separator~$\hat X$ of~$G_r$ is an $r$-locally closed lift of the $r$-local separator $X$ of~$G$.
But not all $r$-local separators $X$ admit such $r$-locally closed lifts~$\hat X$.
For example, it is easy to check that every lift of the $r$-local $4$-separator of $G$ shown in~\cref{fig:rlocallyclopennecessary} fails to be $r$-locally closed in~$G_r$.

We now introduce the \emph{displacement~$\mindist_r(G)$ of $p_r$}, a graph invariant with which we formulate a guarantee for the existence of such lifts:
If $X$ is $r$-tomic and $|X| < 2 \mindist_r(G) / r$, then $X$ has an $r$-locally closed lift $\hat X$ (\cref{RtomicGives1sheeted}).
This then yields the desired lifting/projection lemmas for $r$-local separations, \cref{sufficientForProjection} and \cref{sufficientForLifting}.
We also give various characterisations of the displacement (\cref{lem:midistKeyLemma}) as well as a simple combinatorial lower bound (\cref{cor:combinatoriallowerbound}).
Finally, we provide a construction scheme for graphs (\cref{example:locallychordal,ex:BlocksCoincideGraphClass}) that have high displacement, and on which our theory of $r$-local separation is hence guaranteed to capture much of their $r$-local connectivity structure.

\subsection{Displacement}

The \defn{displacement} of $p_r$ is defined as
\[
    \defnMath{\mindist_r(G)}:=\min\,\big\{\,\text{dist}_{G_r}(\hat u,\hat v):\hat u\text{ and }\hat v\text{ are distinct lifts of the same vertex of }G\,\}
\]
where we follow the convention that $\min(\emptyset) =\infty$.
We remark that $\mindist_r(G)=\infty$ if and only if~$G=G_r$.

\begin{lemma}\label{mindistGeR}
    $\mindist_r(G)>r$ for every connected graph $G$ and $r\in\N$.
\end{lemma}
\begin{proof}
    The proof of \cite[Lemma~4.3]{canonicalGraphDec} shows this.
\end{proof}

\begin{lemma}\label{lem:midistKeyLemma}
    Let $G$ be a connected graph, $\rho \geq r \geq 0$ integers, and $x_0$ a vertex of~$G$.
    Then the following assertions are equivalent:
    \begin{enumerate}
        \item \label{Ginvariant:rlocalhasdistance>rho} $\mindist_r(G) > \rho$;
        \item \label{Ginvariant:rlocalisrhohpres} $p_r$ preserves $(\rho/2)$-balls;
        \item \label{Ginvariant:rgeneratesrho} $\pi_1^r(G,x_0) \supseteq \pi_1^{\rho}(G,x_0)$;
        \item \label{Ginvariant:rlocalisrholocal} $p_r$ and $p_\rho$ are isomorphic as coverings of~$G$;
    \end{enumerate}
    In particular, either $\mindist_r(G) = \infty$ and \cref{Ginvariant:rlocalhasdistance>rho} through \cref{Ginvariant:rlocalisrholocal} are satisfied for every integer~$\rho \geq r$, or $\mindist_r(G) - 1$ is the maximum $\rho \geq r$ that satisfies \cref{Ginvariant:rlocalhasdistance>rho} through~\cref{Ginvariant:rlocalisrholocal}.
\end{lemma}

\begin{proof}
    \cref{Ginvariant:rlocalhasdistance>rho}$\leftrightarrow$\cref{Ginvariant:rlocalisrhohpres} is \cite[Lemma~4.2]{canonicalGraphDec}.

    \cref{Ginvariant:rlocalisrhohpres}$\to$\cref{Ginvariant:rgeneratesrho}.
    Since the covering $p_r$ preserves $(\rho/2)$-balls, the cycles in $G$ of length $\leq \rho$ lift to cycles of~$G_r$. Walks at~$x_0$ that stem from such cycles, those that generate $\pi_1^\rho(G, x_0)$, therefore lie in the characteristic subgroup of~$p_r$.
    By definition of~$p_r$, this is the group~$\pi^r_1(G, x_0)$.

    \cref{Ginvariant:rgeneratesrho}$\to$\cref{Ginvariant:rlocalisrholocal}.
    By definition, $\pi^r_1(G, x_0) \subseteq \pi_1^\rho(G, x_0)$ for $r \leq \rho$, so \cref{Ginvariant:rgeneratesrho} implies that these groups coincide.
    This implies \cref{Ginvariant:rlocalisrholocal} by~\cite[Proposition~1.37]{hatcher}. 

    \cref{Ginvariant:rlocalisrholocal}$\to$\cref{Ginvariant:rlocalhasdistance>rho}. By \cref{mindistGeR}, every two lifts of the same vertex of $G$ to $G_\rho$ have distance greater than $\rho$ in~$G_\rho$.
    Since $p_r$ and $p_\rho$ are isomorphic coverings, the same is true for lifts of the same vertex of~$G$ to $G_r$, that is, $\mindist_r(G) > \rho$.
\end{proof}

\begin{lemma}\label{cor:combinatoriallowerbound}
    Let $G$ be a connected graph and $r \geq 0$ an integer. Then $\mindist_r(G)$ is at least the minimum of the lengths of induced cycles in $G$ longer than $r$.
\end{lemma}

\begin{proof}
    Let $\ell$ be the minimum of all lengths of induced cycles in $G$ longer than~$r$ (letting $\ell:=\infty$ if no such cycle exists), and fix a vertex $x_0$ of~$G$.
    By the implication \cref{Ginvariant:rgeneratesrho} $\to$\cref{Ginvariant:rlocalhasdistance>rho} in~\cref{lem:midistKeyLemma}, it suffices to show that every closed walk $W$ in $G$ based at $x_0$ that stems from a cycle $O$ of length at most $\rho := \ell-1$ is in $\pi_1^r(G,x_0)$.
    We proceed by induction on $|O|$.
    If $O$ has length $\leq r$, the walk $W$ is in $\pi_1^r(G,x_0)$ by definition.
    Thus, we may assume that $|O| > r$.
    Since $W$ stems from $O$, we may write $W = W_0 Q W_0^-$ where $Q = v_0e_1v_1\dots v_n$ is a closed walk once around $O$.
    The definition of $\ell$ implies that $O$ is not induced, i.e.\ there exists a chord $e= v_iv_j$ in $G$ with $i<j$.
    Now $W$ is homotopic to the concatenation of two closed walks $R = W_0 Q_R W_0^-$ and $S = S_0 Q_S  S_0^-$ stemming from shorter cycles where $Q_R := v_0 Q v_i e v_j Q v_n$, $S_0 := W_0 v_n Q^- v_j$ and $Q_S := v_j e v_i Q v_j$.
    By the induction assumption, we have $R,S \in \pi_1^r(G,x_0)$, and thus $W \in \pi_1^r(G,x_0)$.
\end{proof}

\begin{lemma}\label{pathFlipping}
    Assume $r \geq 2$.
    Given a connected graph $G$, let $X\se V(G)$ and let $W\in\cW_r(X)$ use exactly $\ell$ vertices of~$X$.
    Then the endvertices of $W$ have distance~$\le (\ell-1)r/2$ in~$G$.
\end{lemma}
\begin{proof}
    Since $W$ is contained in $\cW_r(X)$, we can write $W$ as a concatenation of $r$-local $X$-paths $P_1,\ldots,P_n$.
    Let $i \in [n]$.
    If $P_i$ has a single edge, set $P_i' := P_i$.
    Otherwise, $P_i$ is contained in a short cycle $O_i$, and we let $P'_i$ be a shortest path in $O_i$ that joins the ends of~$P_i$.
    In both cases, $P'_i$ has length~$\le r/2$ (using the assumption that $r\ge 2$).
    By possible choosing $W \in \cW_r$ such that it connects the same endvertices but $n$ is minimum, we may assume that $W$ repeats no vertices in~$X$ by definition of $\cW_r(X)$, so we have $n\le \ell-1$.
    Hence, the concatenation of all~$P'_i$ is a walk that links the same vertices as $W$ does, and has length~$\le (\ell-1)r/2$.
\end{proof}

Due to the assumption $r \geq 2$ in \cref{pathFlipping}, in statements we will often assume
\begin{equation}\tag{$*$}\label{ASSUMPTION}
    \textit{$r \geq 2$ or $k = 1$}
\end{equation}
by adding \textsuperscript{\ref{ASSUMPTION}} to the respective statement.
We also keep \textsuperscript{\ref{ASSUMPTION}} at each reference to a statement assuming~it.

\begin{lemma}\label{lem:Tight2LocalSepAreSize1}
    Let $G$ be connected and $r\le 2$. Then every tight $r$-local separator of~$G$ has size~$1$.
\end{lemma}

\begin{proof}
    Since $G$ is simple, $G$ contains no cycles of length $\leq 2$. 
    But a tight $r$-local separator of size $\geq 2$ would witness the existence of such a cycle. 
    Thus, tight $r$-local separators have size $1$.
\end{proof}

An $r$-tomic set $X$ is \defn{$r$-locally $1$-covered}
if some (equivalently: every) $r$-local atom $\hat X$ of $p_r\inv (X)$ is a lift of~$X$
(so $p_r$ restricts to a bijection $\hat X\to X$).
Recall from \cref{subsec:interplayWithCov} that the deck transformations $\gamma$ of $p_r$ act on the set of $r$-toms of $p_r\inv(X)$ by $Y \mapsto \gamma(Y)$.
Thus, an $r$-tomic set is $r$-locally $1$-covered if and only if we replace `some' by `every' in the definition of $r$-locally $1$-covered.

\begin{observation}\label{1SheetedFixedPointFree}
    If $X\se V(G)$ is $r$-locally $1$-covered, then the action of the group $\Gamma_r(G)$ of deck transformations of $p_r$ on the $r$-toms of $p_r\inv(X)$ is transitive and fixed-point free, 
    i.e.\ the identity $\id_{G_r}$ is the only deck transformation that maps any $r$-tom on itself.
\end{observation}

\begin{lemmaASS}\label{RtomicGives1sheeted}
    Assume that $kr/2<\mindist_r(G)$ for a connected graph $G$ and $k \in \N$.
    Then every $r$-tomic vertex set $X$ of size at most $k$ is $r$-locally $1$-covered.
    Moreover, the $r$-toms of $p_r\inv(X)$ are precisely the $r$-locally closed lifts of $X$.
\end{lemmaASS}
\begin{proof}
    Let $\hat X$ be an arbitrary $r$-tom of $p_r\inv (X)$.
    As $X$ is $r$-tomic, $p_r$ maps $\hat X$ onto $X$ by \cref{ProjectionLocalAtom}.
    Thus, it remains to show that $p_r$ is injective on~$\hat X$.
    So suppose for a contradiction that there are two distinct vertices $\hat x_1,\hat x_2\in\hat X$ with $p_r(\hat x_1)=p_r(\hat x_2)$.
    Since $\hat X$ is $r$-tomic, there is an $\hat x_1$--$\hat x_2$ walk $\hat W\in\cW_r(\hat X)$.
    We choose the pair of vertices $\hat x_i$ and $\hat W$ so that $\hat W$ has minimum length.
    Then $p_r$ is injective on $(V(\hat W)\cap\hat X) \setminus \{\hat x_2\}$.
    Hence, $|V(\hat W)\cap\hat X|=:\ell\le k+1$ because $X$ has size~$\leq k$.
    If $r \geq 2$, then the $\hat x_i$ have distance $\le (\ell-1)r/2\le kr/2$ by \cref{pathFlipping}.
    This contradicts our assumption that $kr/2<\mindist_r(G)$.
    Otherwise our assumption of \cref{ASSUMPTION} implies $k=1$ and $r \le 1$.
    Then the $\hat x_i$ have distance $1$ but are contained in the same fibre of $p_r$. This contradicts that $G$ has no loops.

    Moreover, since $X$ is $r$-locally $1$-covered, its $r$-toms are $r$-locally closed lifts of $X$. 
    Conversely, every $r$-locally closed lift $\hat X$ of $X$ contains an $r$-tom $\hat Y$ of $p_r\inv(X)$.
    As $X$ is $r$-tomic, $\hat Y$ meets $p_r\inv(x)$ for every $x \in X$ by \cref{ProjectionLocalAtom}, and huts $\hat X = \hat Y$.
\end{proof}

\begin{example}
    \cref{RtomicGives1sheeted} fails if we do not assume \cref{ASSUMPTION}: Consider as $G$ a cycle of length $k$, and the $1$-tomic set $X = V(G)$, then $p_1\inv(X) = V(G_1)$ is $1$-tomic but has infinite size, as $G_1$ is a double~ray.
\end{example}

\begin{projlemmaASS}\label{sufficientForProjection}
    Assume that $G$ is connected.
    Let $\hat s$ be an $r$-local separation of $G_r$ with $r$-tomic separator $\hat X$ of size $\le k$ such that $kr/2<\mindist_r(G)$.
    Then $X:=p_r(\hat X)$ is $r$-tomic and $r$-locally $1$-covered, and $\hat X$ is an $r$-tom of $p_r\inv(X)$. 
    In particular, $\hat X$ is an $r$-locally closed lift of $X$, and so the projection $s:=p_r(\hat s)$ is defined.
\end{projlemmaASS}
\begin{proof}
    The projection $X=p_r(\hat X)$ has size~$\le k$ and is $r$-tomic by \cref{ProjectionRtomic}.
    Hence, $X$ is $r$-locally $1$-covered by \cref{RtomicGives1sheeted}. 
    Since $\hat X$ is $r$-tomic, it is contained in some $r$-tom $\hat Y$ of $p_r\inv(X)$.
    Since $X$ is $r$-locally $1$-covered, $p_r$ restricts to a bijection $\hat Y \to X$.
    As $p_r$ restricted to $\hat X \subseteq \hat Y$ is already onto $X$, we have $\hat X = \hat Y$, so $\hat X$ is an $r$-tom of~$p_r\inv(X)$.
\end{proof}

\begin{liftlemmaASS}\label{sufficientForLifting}
    Assume that $G$ is connected.
    Let $s$ be an $r$-local separation of $G$ with $r$-tomic separator $X$ of size $\le k$ such that $kr/2<\mindist_r(G)$.
    Then $X$ is $r$-locally $1$-covered and $s$ can be lifted.
    In particular, the lifts of $s$ to $G_r$ form a $\Gamma_r$-orbit of the set of $r$-local separations of~$G_r$.
\end{liftlemmaASS}
\begin{proof}
    The separator $X$ is $r$-locally $1$-covered by \cref{RtomicGives1sheeted}.
    Hence, $X$ lifts to an $r$-tom $\hat X$ of $p_r\inv(X)$, which thus is $r$-locally closed. 
    By \cref{keylemma:liftofseparations}, $s$ lifts to an $r$-local separation of $G_r$ with separator $\hat X$.
    The `in particular'-part follows with \cref{1SheetedFixedPointFree}.
\end{proof}

\subsection{A class of examples with high displacement}

Recall that the \defn{girth $g(G)$} of a graph $G$ is the minimal length of a cycle in~$G$ (and $g(G):=\infty$ if $G$ is a forest).
Let $(H,\cG)$ be a \gd\ (with parts~$G_h$) of~$G$.
We associate with every edge $e = h_0h_1$ of $H$ its \defn{adhesion graph} $G_e := G_{h_0} \cap G_{h_1}$.

Assume for now that the adhesion graphs of $(H,\cG)$ are pairwise disjoint and non-empty.
For every edge $e = hh'$ of~$H$ we define $\defnMath{s^e := \{E^e_h,X^e,E^e_{h'}\}}$, where $X^e := V(G_e)$, $E^e_h := \partial_{G_h} X^e$ and $E^e_{h'} := \partial_{G_{h'}} X^e$.
The axioms for \gd s are easily seen to imply that, since our adhesion graphs are disjoint, the edge sets $E^e_h$ and $E^e_{h'}$ form a near-partition of $\partial_G X^e$.

In this context, a \defn{traversal} of $s^e$ is a walk in $G$ whose first edge is in $E^e_h$, whose last edge is in~$E^e_{h'}$, and all whose internal vertices are in~$X^e$.
A closed walk $W = v_1e_1v_2 \dots v_ne_nv_1$ in $G$ \defn{traverses} $s^e$ with $e \in E(H)$ if (a cyclic shift of) $W$ contains a traversal of~$s^e$.
For a closed walk $W$ in $G$ we let $\defnMath{H^W}$ denote the subgraph of $H$ formed by all the edges $e\in E(H)$ (and their endvertices) for which $W$ traverses~$s^e$. 

\begin{example}\label{example:locallychordal}
    Assume that $G$ is a connected graph. Let $(H, \cG)$ be a \gd\ (with parts~$G_h$) of $G$ such that
    \begin{enumerate}
        \item every part $G_h$ is complete, and
        \item the adhesion graphs $G_e$ are disjoint.
    \end{enumerate}
    Then all induced cycles of $G$ of length~$>3$ have length $\geq g(H)$.
    By \cref{cor:combinatoriallowerbound}, therefore, $\mindist_r(G) \geq g(H)$ for every $r \geq 3$.
\end{example}

\begin{proof}
    Let $O$ be an induced cycle in~$G$.
    In what follows, we interpret $O$ also as a closed walk once around~$O$.
    If $H^O$ is empty then, by the definition of \gd s, $O$ is contained in a part~$G_h$; thus, $O$ is a triangle, because the parts are complete by~(i).
    Assume now that $H^O$ is not empty.
    Note that, for every traversal $P$ of $s^e$ for an edge $e\in E(H)$, the two endvertices of $P$ send edges of $G$ to all internal vertices of~$P$.
    Since $O$ is induced and the parts of $(H, \cG)$ are complete, the intersection $O \cap G_e$ thus consists of a single vertex, and for every node $h$ of~$H^O$ the intersection $O \cap G_h$ consists of a single edge joining different adhesion graphs.
    Thus, $H^O$ is a cycle $h_1 \dots h_nh_1$.
    The unique edges of $O$ in $G_{h_i}$ are distinct for different $i$, by~(ii).
    Hence, $O$ has at least~$n\geq g(H)$ edges.
\end{proof}

Let $(H,\cG)$ be a \gd\ (with parts~$G_h$) of~$G$, and assume that its adhesion graphs are pairwise disjoint.
We say that a closed walk $W$ traverses $s^e$ \defn{$k$ times} (where $e\in E(H)$) if the number of cyclic shifts of $W$ which contain a traversal of~$s^e$ as an initial segment equals $k$.

For every node $h$ of~$H$, let $\mathring G_h$ denote the graph obtained from $G_h$ by deleting all edges that are contained adhesion graphs.
Let $W = v_1e_1v_2 \dots v_ne_nv_1$ be a closed walk in~$G$.
If $v_i e_i \ldots e_k v_{k+1}$ is a subwalk of (a cyclic shift of) $W$ that is a traversal of $s^f$ for some $f\in E(H)$, then we refer to the indices $j$ with $i<j<k$ as \defn{bad indices} of the walk~$W$.
All indices of $W$ that are not bad are \defn{good}.
We denote by $\defnMath{f(W)}$ the number of all good indices of~$W$.
We then let $\defnMath{\ell (H, \cG)}$ be the minimum~$f(W)$ among all closed walks~$W$ in~$G$ which traverse at least one~$s^e$ and which traverse~$s^e$ at most once for each edge~$e \in H$.
If no such walk exists, we let $\ell(H,\cG):=\infty$.
Further, we write $\defnMath{d(H,\cG)}$ for the minimum distance $d_{G_h}(G_e,G_f)$ over all nodes $h$ of $H$ and distinct edges $e,f$ of $H$ at~$h$.
If no such distinct edges $e,f$ exist, we let $d(H,\cG):=\infty$.

\begin{corollary}\label{cor:HW}
    Let $(H,\cG)$ be a \gd\ (with parts~$G_h$) of a finite connected graph $G$ with pairwise disjoint adhesion graphs.
    Then $\ell (H,\cG) \geq g(H) \cdot d(H,\cG)$.
\end{corollary}
\begin{proof}
    If $\ell(H,\cG)=\infty$, we are done; therefore, we can assume that there exists a closed walk $W$ in~$G$ that witnesses $\ell(H,\cG)<\infty$.
    Since $W$ traverses at least one~$s^e$, the graph $H^W$ is non-empty.
    Let $h$ be an arbitrary node of~$H^W$. 
    By the definition of~$H^W$, the closed walk $W$ traverses $s^e$ for some $e \in E(H)$ incident to~$h$.
    But as the adhesion graphs of $(H,\cG)$ are disjoint and $W$ traverses $s^e$ no more than once, $W$ leaves $G_h$ by traversing another $s^f$ with $f \in E(H)$ incident to $h$.
    So $h$ has degree $\ge 2$ in~$H^W$.
    As $G$ is finite and the adhesion graphs are disjoint, $H$ and $H^W$ are also finite.
    Hence, $H^W$ contains a cycle~$O$.

    Let $h$ be a node of~$O$. Let $e,f$ be the two edges on~$O$ that are incident to~$h$.
    The walk $W$ contains a subwalk $W_h$ in $G_h$ between two distinct adhesion graphs $G_e,G_f$ of~$(H,\cG)$, as $W$ traverses $s^e, s^f$ at most once.
    We can and do choose $W_h$ so that its internal vertices are not contained in the adhesion graphs~$G_e,G_f$.
    Then $W_h$ has length at least~$d(H,\cG)$.
    Since $h$ has degree~2 in~$O$, all indices of~$W$ that appear at the edges of~$W_h$ are good indices.
    As we find $W_h$ with these properties for each node $h$ of~$O$, we find that $f(W)\ge g(H)\cdot d(H,\cG)$.
\end{proof}

\begin{example} \label{ex:BlocksCoincideGraphClass}
    Let $(H, \cG)$ be a \gd\ (with parts~$G_h$) of a finite connected graph $G$ and an integer $r \geq 0$
    such that
    \begin{enumerate}
        \item \label{item:partsconnected} every part $G_h$ is connected and $\pi_1(G_h) = \pi_1^r(G_h)$,
        \item \label{item:disjointadhesion} its adhesion graphs $G_e$ are disjoint,
        \item \label{item:connectedadhesion} every adhesion graph $G_{e}$ is connected.
    \end{enumerate}
    Then $\mindist_r(G) \geq \ell (H,\cG) \geq g(H) \cdot d(H,\cG) \geq g(H)$.
\end{example}

\begin{proof}
    The inequality $\ell(H,\cG) \geq g(H) \cdot d(H,\cG)$ was proved in \cref{cor:HW}.
    The disjointness of the adhesion graphs given by~\cref{item:disjointadhesion} ensures that
    $d(H,\cG) \geq 1$, and thus $g(H) \cdot d(H,\cG) \geq g(H)$.
    So it remains to show $\mindist_r(G)\ge\ell(H,\cG)$.

    We say that a closed walk $W$ in~$G$ is \emph{generated by short cycles} if, for some (equivalently: every) vertex $x_0$ of $G$ and for some (equivalently: every) closed walk $W'$ based at $x_0$ which stems from $W$, we have $W' \in \pi_1^r(G,x_0)$.

    To show $\mindist_r(G)\ge\ell(H,\cG)$, it suffices by \cref{lem:midistKeyLemma}~(iii)$\to$(i) to show that every closed walk $W$ in~$G$ of length $< \ell(H, \cG)$ is generated by short cycles.    
    Since $f(W)$ is no larger than the length of~$W$, we may prove the stronger statement that
    \begin{equation}\phantomsection\label{eq:NiceExample}
        \textit{every closed walk $W$ in $G$ with $f(W) < \ell(H,\cG)$ is generated by short cycles.}
    \end{equation}
    We prove \cref{eq:NiceExample} by induction on the number of subwalks of (cyclic shifts of)~$W$ that are traversals of some~$s^e$ (for $e\in E(H)$).
    
    If $W$ traverses no $s^e$ (which is the case if $f(W) = 0$), then a single part $G_h$ contains~$W$, and thus the short cycles in~$G_h$ generate $W$ by~\cref{item:partsconnected}.
    If $W$ traverses at least one~$s^e$ and $W$ traverses $s^e$ at most once for every edge $e$ of~$H$, then~$W$ is itself a candidate in the definition of~$\ell(H, \cG)$ and hence~$f(W) \ge \ell(H, \cG)$, which contradicts that $W$ has length $<\ell (H,\cG)$ by assumption.

    Now assume that $W=: v_1e_1v_2 \dots v_ne_nv_1$ traverses some $s^e$ of $(H,\cG)$ at least twice.
    Let $P_i$ and $P_j$ be two such traversals where $P_i$ meets $G_e$ in a vertex $v_i$ and $P_j$ meets $G_e$ in a vertex $v_j$ with $i<j$.
    We may pick a $v_i$--$v_j$ path $P$ in $G_e$ by~\cref{item:connectedadhesion}.
    Let $W_1 := v_1Wv_iPv_jWv_1$ and $W_2 := v_j P^- v_iWv_j$.
    As $W$ is homotopic to $W_1 (v_1W^-v_j) W_2 (v_j W v_1)$, it suffices to show that $W_1$ and $W_2$ are generated by short cycles.
    Now both~$W_1$ and~$W_2$ have less traversals than~$W$: By construction, each traversal of some~$s^f$ of~$(H, \cG)$ in~$W_1$ or~$W_2$ is also a traversal in~$W$, except possibly for the new ones using~$P$ or~$P^-$, of which we may have created at most one in~$W_1$ and in $W_2$ (but these need not be traversals at all).
    While~$W$ has the two traversals~$P_i$ and~$P_j$ of~$s^e$, $W_1$ and~$W_2$ contain neither of these two and each possibly a single new one.
    Thus, $W_1$ and $W_2$ each have at least one traversal less than~$W$, and we can hence apply the induction hypothesis to both, completing the proof.
\end{proof}

\subsection{Local separations whose separator is a clique}

Recall from the example of \cref{fig:rlocally1sheetednecessary} that an $r$-tomic separator $\hat X$ of $G_r$ can be much larger than the $r$-local separator $X$ of $G$ to which it projects. This cannot happen, however, if $\hat X$ induces a clique -- an important case applied throughout~\cite{locallychordal,localGlobalChordal,CanTDChordalGraphs}:

\begin{liftlemma}\label{lift:cliqueseparations}
    Let $G$ be a connected graph and $r \geq 3$ an integer.
    The $r$-local separations of $G_r$ whose separator is a clique are precisely the lifts of the $r$-local separations of $G$ whose separator is a clique.
\end{liftlemma}

For the proof of \cref{lift:cliqueseparations}, we need a lemma:

\begin{lemma}\label{lem:CliquesAndLocalCover}
    Let $G$ be a connected graph and $r \geq 3$ an integer.
    Then:
    \begin{enumerate}
        \item \label{itm:CliqueProjInjective} For every clique $\hat X\se V(G_r)$, $p_r$ is injective on $\hat X$, and its projection $p_r(\hat X)$ is a clique of~$G$.
        \item \label{itm:CliquesLift} For every clique $X\se V(G)$ and every lift~$\hat x$ of some vertex~$x \in X$ to~$G_r$, there exists a unique clique $\hat X$ of $G_r$ containing~$\hat x$ such that $p_r$ restricts to a bijection from $\hat X$ to $X$.
        \item \label{itm:CliquesAreRToms} $(\hat X \cup N(\hat X)) \cap (\gamma(\hat X) \cup N(\gamma(\hat X))) = \emptyset$ for every clique $\hat X\se V(G_r)$ and every $\gamma \in \Gamma(p_r) \setminus \{\id_{G_r}\}$. 
    \end{enumerate}
    Moreover, for every clique $X\se V(G)$, the $r$-toms of $p_r^{-1}(X)$ in~$G_r$ are precisely the cliques $\hat X\se V(G_r)$ for which $p_r$ restricts to a bijection from $\hat X$ to $X$.
\end{lemma}

\begin{proof}[Proof of \cref{lem:CliquesAndLocalCover}]
    \cref{itm:CliqueProjInjective} holds by the definition of covering.
    \cref{itm:CliquesLift} follows from~$r \ge 3$ and the uniqueness of path-lifting in coverings.

    For~\cref{itm:CliquesAreRToms}, suppose for a contradiction that there are $\gamma \in \Gamma(p_r) \setminus \{\id_{G_r}\}$ and a clique $\hat X\se V(G_r)$ such that $\hat X \cup N(\hat X)$ and $\gamma(\hat X) \cup N(\gamma(\hat X))$ share a vertex~$\hat y$.
    Then there are a vertex $\hat x_0 \in \hat X$ and a vertex $\hat x_1 \in \gamma(\hat X)$ both with distance $\leq 1$ to $\hat y$ in~$G_r$.
    As $\hat x_0' \coloneqq \gamma (\hat x_0)$ has distance $\leq 1$ from $\hat x_1 \in \gamma(\hat X)$, the two vertices $\hat x_0$ and $\hat x_0'$ in the same fibre of $p_r$ thus have distance $\leq 3 \leq r$.
    This contradicts \cref{mindistGeR}.

    For the `moreover'-part, note first that any lift $\hat X$ of $X$ to $G_r$ that is a clique must be contained in some $r$-tom of~$p_r^{-1}(X)$.
    For the converse inclusion, consider any $r$-tom $\hat Y$ of~$p_r^{-1}(X)$.
    By \cref{itm:CliquesLift}, $\hat Y$ contains a lift~$\hat X$ of~$X$ that is a clique.
    Since~$\Gamma(p_r)$ acts transitively on the fibres of~$p_r^{-1}$, \cref{itm:CliquesAreRToms}~implies that~$\hat X$ cannot be joined by an edge to any vertex in~$p_r^{-1}(X) \sm \hat X$.
    Now if $\hat Y \sm \hat X$ is non-empty, then there is a short cycle in~$G_r$ meeting both~$\hat X$ and a vertex~$\hat y \in \hat Y \setminus \hat X$.
    Thus, $\hat y$ has distance at most $r-1$ to some vertex in~$\hat X$, and since $\hat X$ is a clique, $\hat y$ has distance at most $r$ to every vertex in~$\hat X$.
    But $\hat X$ is a lift of~$X$, and so $\hat X$ contains a vertex $\hat y'$ that lies in the same fibre as~$\hat y$.
    This contradicts~\cref{mindistGeR}.
\end{proof}

\begin{proof}[Proof of \cref{lift:cliqueseparations}]
    By \cref{lem:CliquesAndLocalCover}~\cref{itm:CliqueProjInjective,itm:CliquesLift}, $p_r$ induces a surjection $\hat X \mapsto p_r(\hat X)$ from the cliques of $G_r$ to the cliques of $G$ which is injective up to deck transformation.
    By the `moreover'-part of~\cref{lem:CliquesAndLocalCover}, the lifts $\hat X$ of $X$ to $G_r$ that are cliques are precisely the $r$-toms of $p_r\inv(X)$ in $G_r$; in particular, these lifts $\hat X$ are precisely the $r$-locally closed lifts of $X$. 
    Thus, the statement follows immediately from \cref{ProjectionLocalSeparation} and \cref{keylemma:liftofseparations}.
\end{proof}

\section{Links and crossing}\label{sec:LinksCrossing}

As the next step regarding $r$-local separations, we develop notions of \emph{crossing} and \emph{nestedness} for $r$-local separations.
We show that they agree with the respective notions for global separations on $r$-local coverings (\cref{corresLinks,corresCross}).
We then prove, assuming $k < K(G,r) := \frac{\mindist_r(G)}{r} + 1$, that two $\lek$-separations $s$ and $t$ of $G$ are nested if and only if all their lifts are nested (\cref{uniqueRcoupleLift} and \cref{LiftingLinks}).
Moreover, if $s$ and $t$ cross, then every lift of $s$ crosses exactly one lift of $t$ (\cref{NestedRtomic}).

\subsection{From global to local, and correspondence in the local cover}

A cycle $O$ \defn{alternates} between two disjoint vertex sets $X$ and $Y$ if there are four distinct vertices $x,x',y,y'$ on $O$ such that $x,x'\in X$ and $y,y'\in Y$ and $O$ visits these four vertices in the order of $x,y,x',y'$.
We emphasise that $O$ may intersect $X\cup Y$ in more than four vertices.
This definition is inspired by~\cite{Local2sep}.

Two vertex sets $X,Y\se V(G)$ are \defn{$r$-coupled} in~$G$ if either $X$ and $Y$ share a vertex or there is a short cycle in $G$ that alternates between $X$ and~$Y$.
Further, let $X \subseteq V(G)$ and let $\{F_1,Y,F_2\}$ be any $r$-local separation of~$G$.
The \defn{$X$-link for~$F_i$ in $\{F_1,Y,F_2\}$} is the set of all vertices $x\in X\sm Y$ such that there is a walk in $\cW_r(X)$ that starts at~$x$, visits~$Y$, and its first edge in $\partial Y$ is contained in~$F_i$.

Let $\{E_1,X,E_2\}$ and $\{F_1,Y,F_2\}$ be two $r$-local separations with $r$-tomic separators.
Inspired by \cref{char:NestedViaLinks}, we say that $\{E_1,X,E_2\}$ and $\{F_1,Y,F_2\}$ \defn{cross} if
\begin{itemize}
    \item $X$ and $Y$ are $r$-coupled, and
    \item for every pair of indices $i,j\in [2]$ at least one of the following three sets is non-empty:\\
    the $X$-link for $F_j$, the $Y$-link for~$E_i$, or the edge set $E_i\cap F_j\cap\partial (X\cap Y)$.
\end{itemize}
If $\{E_1,X,E_2\}$ and $\{F_1,Y,F_2\}$ do not cross, then they are \defn{nested}.

\begin{lemma}\label{oppositeNonemptyLinks}
    If two separations of a connected graph with empty centre cross, then some two opposite links are non-empty.
\end{lemma}
\begin{proof}
    Otherwise, the separator of some corner is empty, which contradicts \cref{char:NestedViaLinks}~(iii).
\end{proof}

\begin{lemma}\label{AlternatingCycle}
    Let $\{A_1,A_2\}$ and $\{C_1,C_2\}$ be crossing separations of a connected graph~$G$ with $r$-tomic separators $X$ and~$Y$, respectively. 
    Then $X$ and $Y$ are $r$-coupled.
\end{lemma}
\begin{proof}
    We are done if $X$ and $Y$ intersect.
    Thus, we may assume that $X\cap Y=\emptyset$.
    Then \cref{oppositeNonemptyLinks} yields that two opposite links, say both $X$-links, are non-empty.
    Since $X$ is $r$-tomic and $X\cap Y$ is empty, there is a short cycle $O\se G$ or an edge $e$ that meets both $X$-links, one in a vertex $x$ and the other one in a vertex $x'$.
    The edge $e$ would contradict the fact that $\{C_1, C_2\}$ is a separation.
    Thus, $O$ exists.
    Then both $x$--$x'$ paths in~$O$ must meet~$Y$, and therefore $O$ alternates between $X$ and~$Y$.
\end{proof}

\begin{corlemma}[Links]\label{corresLinks}
    Let $\{A_1,A_2\}$ and $\{C_1,C_2\}$ be separations of a connected graph~$G$ with $r$-tomic separators $X$ and $Y$, respectively, such that $X$ and $Y$ are $r$-coupled.
    Let $\{E_1,X,E_2\}$ and $\{F_1,Y,F_2\}$ be the $r$-local separations induced by $\{A_1,A_2\}$ and $\{C_1,C_2\}$, respectively.
    Then:
    \begin{enumerate}
        \item\label{item:corXlink} the $X$-link for $C_i$ is equal to the $X$-link for~$F_i$, for both $i=1,2$;
        \item\label{item:corYlink} the $Y$-link for $A_i$ is equal to the $Y$-link for~$E_i$, for both $i=1,2$;
        \item\label{item:corCentre} $E\big( X\cap Y,(A_i\sm A_{3-i})\cap (C_j\sm C_{3-j})\big)= E_i\cap F_j\cap\partial (X\cap Y)$ for all $i,j\in\{1,2\}$.
    \end{enumerate}
\end{corlemma}

\begin{proof}
    Since $X$ and $Y$ are $r$-coupled, there is a vertex $v\in X\cap Y$ or there is a short cycle $O\se G$ that alternates between $X$ and~$Y$.
    To show \cref{item:corXlink} and~\cref{item:corYlink}, by symmetry it suffices to show that the $X$-link for~$C_1$ is equal to the $X$-link for~$F_1$.
    The backward inclusion is straightforward.
    To see the forward inclusion, let $x$ be a vertex in the $X$-link for~$C_1$.
    If $v \in X \cap Y$ exists, we use that $X$ is $r$-tomic to find a walk $W\in\cW_r(X)$ that starts with $x$ and visits~$v\in Y$.
    Otherwise, we use that $X$ is $r$-tomic and the existence of $O$ to find a walk $W\in\cW_r(X)$ that starts in $X$ and visits~$Y$.
    Let $y$ denote the first vertex of $W$ in~$Y$.
    Since $x$ lies in the $X$-link for~$C_1$, all vertices of $W$ that precede $y$ are contained in $C_1\sm C_2$.
    Hence, $y$ is preceded on~$W$ by an edge in~$F_1$.
    In particular, the first edge of the walk $W$ in $\partial Y$ lies in~$F_1$.
    So $x$ lies in the $X$-link for~$F_1$.
    
    \cref{item:corCentre} is immediate from the assumption that $\{A_1,A_2\}$ induces $\{E_1,X,E_2\}$ and $\{C_1,C_2\}$ induces~$\{F_1,Y,F_2\}$.
\end{proof}

\begin{corlemma}[Crossing]\label{corresCross}
    Let $s$ and $t$ be separations of a connected graph~$G$ with $r$-tomic separators.
    The separations $s$ and $t$ cross if and only if their $r$-local counterparts $s_r$ and $t_r$ cross.
\end{corlemma}
\begin{proof}
    Let $X$ and $Y$ denote the separators of $s$ and~$t$, respectively.
    If $s$ and $t$ cross, then $X$ and $Y$ are $r$-coupled by \cref{AlternatingCycle}.
    If $s_r$ and $t_r$ cross, then $X$ and $Y$ are $r$-coupled by the definition of crossing.
    In either case, we may use \cref{corresLinks} to find that $s_r$ and~$t_r$ or $s$ and~$t$ also cross.
\end{proof}

\subsection{Interplay with the covering map}

\begin{lemma}\label{randomInequality}
    Whenever $G$ is a connected graph, $(k-1)r<\mindist_r(G)$ implies $kr/2<\mindist_r(G)$, for all integers $k,r\ge 0$.
\end{lemma}
\begin{proof}
    If $k\le 2$, then $kr/2 \leq r<\mindist_r(G)$, where the last inequality holds by~\cref{mindistGeR}.
    Otherwise, $k\ge 3$ gives $(k-1)r-kr/2=(k/2-1)r > 0$, so $kr/2 < (k-1)r$ and the desired implication follows.
\end{proof}

Recall that we assume that $r \geq 2$ or $k = 1$, if we add~~\ASSUMPTION\ to a statement. 
We assume~~\ASSUMPTION\ for the rest of \cref{sec:LinksCrossing}.

The next lemma addresses lifting the notion of `$r$-coupled'.
For this, recall that for an $r$-tomic vertex set~$X$ in a connected graph~$G$, the $r$-toms~$\hat X$ of~$p_r^{-1}(X)$ are the lifts of~$X$ by~\cref{randomInequality,RtomicGives1sheeted}.

\begin{liftlemmaASS}\label{uniqueRcoupleLift}
    Let $X$ and $Y$ be $r$-coupled, $r$-tomic vertex sets of size~$\le k$ in a connected graph~$G$ such that $(k-1)r<\mindist_r(G)$.
    Let $\hat X$ be an $r$-tom of $p_r\inv (X)$.
    Let $\hat H$ denote the union of all short cycles in $G_r$ that have at least two vertices in~$\hat X$.
    Then there exists a unique $r$-tom $\hat Y$ of $p_r\inv (Y)$ such that $\hat Y$ meets $\hat X\cup V(\hat H)$.
    Any $v \in X \cap Y$ then has unique lift $\hat v$ in $\hat X \cap \hat Y$, and
    $\hat Y$ is the unique $r$-tom of $p_r\inv (Y)$ that is $r$-coupled with~$\hat X$.
    Every short cycle $O$ in $G$ that alternates between $X$ and $Y$ then has a unique lift $\hat O$ that alternates between $\hat X$ and $\hat Y$.
\end{liftlemmaASS}
\begin{proof}
    First, assume that $r \leq 1$. Then $k = 1$ by~~\ASSUMPTION , so $|X|=|Y|=1$. Since there are no cycles of length at most $1$ and as $X$ and $Y$ are $r$-coupled, we thus have $H = \emptyset$ and $X = Y$, say $X=Y=\{z\}$.
    As $G$ has no loops (at $z$), $p_r\inv(z)$ is an independent set vertex set of $G_r$, whose $r$-toms are formed by its single vertices. 
    Every $r$-tom $\hat X = \{\hat z\}$ is the unique $r$-tom $\hat Y$ of $p_r\inv(Y)$ that meets $\hat X$.
    Hence, we may assume that $r \geq 2$.
    
    \casen{Existence of $\hat Y$.}
    Let an $r$-tom $\hat X$ of $p_r\inv(X)$ be given.
    Since $X$ is $r$-locally $1$-covered, by \cref{randomInequality} and \cref{RtomicGives1sheeted}, $\hat X$ is a lift of $X$. 
    As $X$ and $Y$ are $r$-coupled, there either is a vertex $v\in X\cap Y$, or $X$ and $Y$ are disjoint and there is a short cycle $O\se G$ that alternates between $X$ and~$Y$.
    If $v \in X \cap Y$ exists, let $\hat v$ be its unique lift in~$\hat X$.
    If $O$ exists, let $\hat O$ denote any lift of $O$ that meets $\hat X$.
    As $\hat O$ is short, it meets no other $r$-tom of $p_r\inv(X)$.
    Similarly, let $\hat Y$ denote the unique $r$-tom of $p_r\inv (Y)$ that contains $\hat v$ if $v$ exists.
    if not, let $\hat Y$ be the unique $r$-tom of $p_r\inv(Y)$ that meets $\hat O$.
    As earlier, $\hat Y$ is a lift of $Y$.
    Then $\hat v\in\hat X\cap\hat Y$ or $\hat O$ is a short cycle in $G_r$ that alternates between the disjoint vertex sets $\hat X$ and~$\hat Y$.
    So $\hat X$ and $\hat Y$ are $r$-coupled, and in particular $\hat Y$ meets $\hat X\cup V(\hat H)$, as desired.
    Finally, since $\hat X$ is a lift of $X$ and also $r$-locally closed as an $r$-tom of $p_r\inv(X)$, the short cycle $\hat O$ is the unique lift of $O$ that meets $\hat X$; in parituclar, it is the unique lift of $O$ alternating between $\hat X$ and $\hat Y$.
    
    \casen{Uniqueness of $\hat Y$.}
    Assume for a contradiction that there are two distinct $r$-toms $\hat Y_1,\hat Y_2$ of $p_r\inv (Y)$ that both meet $\hat X\cup V(\hat H)$ in vertices $\hat y_1,\hat y_2$, respectively. 
    If $\hat y_i\in V(\hat H)\sm\hat X$, then let $\hat O_i$ be a short cycle in $G_r$ that contains $\hat y_i$ and intersects $\hat X$ in some two vertices $\hat x_i,\hat x'_i$.
    Further, let $\hat y'_i$ denote the unique vertex in $\hat Y_{3-i}$ with $p_r(\hat y'_i)=p_r(\hat y_i)$.

    Suppose first that both $\hat y_i$ lie in~$\hat X$.
    Since $\hat X$ and $\hat Y_2$ are $r$-tomic, there are an $\hat y_1$--$\hat y_2$ path and an $\hat y_2$--$\hat y_1'$ path in $G_r$, both of length $\le (k-1)r/2$, by \cref{pathFlipping}.
    So $\hat y_1$ and $\hat y_1'$ have distance $\le (k-1)r$, contradicting our assumption that $(k-1)r<\mindist_r(G)$.

    Suppose second that $\hat y_1$ lies in~$\hat X$ and that $\hat y_2$ lies on~$\hat O_2$.
    Since $\hat X$ is $r$-tomic, there is a $\hat y_1$--$\{\hat x_2,\hat x'_2\}$ walk $Q(\hat X)\in\cW_r(\hat X)$.
    By choosing $Q(\hat X)$ of minimum length, we may assume that $Q(\hat X)$ ends in $\hat x_2$ and avoids~$\hat x'_2$, say.
    Since $Q(\hat X)$ avoids $\hat x'_2$, there is a $\hat y_1$--$\hat x_2$ path $Q'(\hat X)$ in $G_r$ of length~$\le (k-2)r/2$ by \cref{pathFlipping}.
    Let $Q(\hat O_2)$ be a path of length~$\le r/2$ in $\hat O_2$ from $\hat x_2$ to~$\hat y_2$.
    Since $\hat Y_2$ is $r$-tomic, there is an $\hat y_2$--$\hat y'_1$ path $Q(\hat Y_2)$ in $G_r$ of length~$\le (k-1)r/2$, by \cref{pathFlipping}.
    Then the concatenation of $Q'(\hat X)$, $Q(\hat O_2)$ and $Q(\hat Y_2)$ links $\hat y_1$ to~$\hat y'_1$ and has length~$\le (k-1)r$, contradicting our assumption that $(k-1)r<\mindist_r(G)$.

    The case that $\hat y_1$ lies on $\hat O_1$ while $\hat y_2$ lies in~$\hat X$ follows from the previous case by symmetry.

    Suppose finally that both $\hat y_i$ lie on~$\hat O_i$.
    Since $\hat X$ is $r$-tomic, there is a $\{\hat x_1,\hat x'_1\}$--$\{\hat x_2,\hat x'_2\}$ walk $Q(\hat X)\in\cW_r(\hat X)$; choose this of minimum length.
    Possibly by renaming some vertices, we may assume that $Q(\hat X)$ links $\hat x_1$ to~$\hat x_2$.
    If $\{\hat x_1,\hat x'_1\}$ and $\{\hat x_2,\hat x'_2\}$ share a vertex, then $Q(\hat X)$ has length~$0$, and we let $Q'(\hat X):=Q(\hat X)$.
    Otherwise $Q(\hat X)$ avoids two vertices in~$\hat X$, so there is a $\hat x_1$--$\hat x_2$ path $Q'(\hat X)$ in $G_r$ of length~$\le (k-3)r/2$ by \cref{pathFlipping}.
    For both~$i$ let $Q(\hat O_i)$ be a $\hat y_i$--$\hat x_i$ path in $\hat O_i$ of length~$\le r/2$.
    Using that $\hat Y_2$ is $r$-tomic and \cref{pathFlipping}, let $Q(\hat Y_2)$ be a $\hat y_2$--$\hat y'_1$ path in $G_r$ of length~$\le (k-1)r/2$.
    Then the concatenation of $Q(\hat O_1)$, $Q'(\hat X)$, $Q(\hat O_2)$ and $Q(\hat Y_2)$ links $\hat y_1$ to $\hat y'_1$ and has length $\le (k-1)r$, contradicting our assumption that $(k-1)r<\mindist_r(G)$.
\end{proof}

\begin{liftlemmaASS}[Links]\label{LiftingLinks}
    Let $G$ be a connected graph.
    Let $s=\{E_1,X,E_2\}$ and $t=\{F_1,Y,F_2\}$ be $r$-local separations of~$G$ with $r$-tomic separators of size $\le k$ such that $X$ and $Y$ are $r$-coupled and $(k-1)r<\mindist_r(G)$.
    Let $\hat s=\{\hat E_1,\hat X,\hat E_2\}$ be a lift of~$s$, and let $\hat t=\{\hat F_1,\hat Y,\hat F_2\}$ be the lift of~$t$ such that $\hat X$ and $\hat Y$ are $r$-coupled, which exists and is unique by \cref{uniqueRcoupleLift}.
    Then $p_r$ restricts to bijections between
    \begin{enumerate}
        \item \label{item:LiftXlink} the $\hat X$-link for~$\hat F_i$ and the $X$-link for~$F_i$, for both $i=1,2$;
        \item \label{item:LiftYlink} the $\hat Y$-link for~$\hat E_i$ and the $Y$-link for~$E_i$, for both $i=1,2$;
        \item \label{item:LiftXcapY} $\hat X\cap\hat Y$ and $X\cap Y$;
        \item \label{item:LiftCentre} $\hat E_i\cap\hat F_j\cap\partial (\hat X\cap\hat Y)$ and $E_i\cap F_j\cap\partial (X\cap Y)$, for all $i,j\in\{1,2\}$.
    \end{enumerate}
    Moreover, $s$ and $t$ cross if and only if $\hat s$ and $\hat t$ cross.
\end{liftlemmaASS}

\noindent In the context of \cref{LiftingLinks}, we recall that $s$ and $t$ can be lifted by \cref{sufficientForLifting} and \cref{randomInequality}.

\begin{proof}
    Since $X$ and $Y$ are $r$-coupled, there either is a vertex $v\in X\cap Y$ or $X$ and $Y$ are disjoint and there is a short cycle $O\se G$ that alternates between $X$ and~$Y$.
    If $v$ exists, it has a unique lift $\hat v$ in $\hat X\cap\hat Y$ by \cref{uniqueRcoupleLift}.
    If $O$ exists, it has a unique lift $\hat O$ that alternates between $\hat X$ and~$\hat Y$, again by \cref{uniqueRcoupleLift}.

    Since $\hat s$ is a lift of $s$, its separator $\hat X$ is a lift of $X$ and $r$-locally closed.
    Further, by \cref{RtomicGives1sheeted}, $\hat X$ is an $r$-tom of $p_r\inv(X)$.
    Analogously, $\hat Y$ is an $r$-locally closed lift of $Y$ and thus an $r$-tom of $p_r\inv(Y)$.
    By \cref{uniqueRcoupleLift}, $\hat Y$ is the unique $r$-tom of $p_r\inv(Y)$ which is $r$-coupled with $\hat X$.
    Thus, \cref{item:LiftXcapY} holds, and this immediately yields
    \cref{item:LiftCentre}.
    
    We now claim that 
    \begin{equation}\label{linkLiftInclusionProof}
        \text{the lift of the $X$-link for~$F_1$ that lies in $\hat X$ is included in the $\hat X$-link for~$\hat F_1$.}
    \end{equation}
    To see this, let $x$ be a vertex in the $X$-link for~$F_1$, and let $\hat x$ denote its unique lift in~$\hat X$.
    Let $W\in\cW_r(X)$ witness that $x$ lies in the $X$-link for~$F_1$; that is, $W$ is a concatenation of $r$-local $X$-paths $P_1,\ldots, P_n$ starting at $x$, and the first edge $e$ of $W$ in $\partial Y$ exists and lies in~$F_1$.
    Let $\hat W$ denote the lift of~$W$ to~$G_r$ that starts at~$\hat x$.
    By the path-lifting property of coverings, this defines lifts $\hat P_i$ of all the~$P_i$ and a lift $\hat e$ of~$e$.
    As $\hat X$ is $r$-locally closed, the $\hat P_i$ are $r$-local $\hat X$-paths. 
    Hence, $\hat W\in \cW_r(\hat X)$.
    It is immediate that $\hat e$ is the first edge of $\hat W$ in $\partial p_r\inv(Y)$.
    Thus, it suffices to show that $\hat e$ lies in~$\partial \hat Y$.
    
    Let $\hat Z$ be the $r$-tom of $p_r\inv (Y)$ with $\hat e\in\partial \hat Z$.
    Since $\hat e$ lies on some $r$-local $\hat X$-path $\hat P_i$, it either lies on a short cycle in $G_r$ that has two vertices in~$\hat X$ or is an $\hat X$-edge.
    Thus, $\hat X$ and $\hat Z$ are $r$-coupled.
    Since $\hat Y$ is the unique $r$-tom of $p_r\inv(Y)$ that is $r$-coupled with $\hat X$, \cref{uniqueRcoupleLift} yields $\hat Z=\hat Y$.
    Therefore, $\hat e\in \partial Y$.
    Then $\hat e\in\hat F_1$ as well.
    This completes the proof of~\cref{linkLiftInclusionProof}.

    Next, we claim that
    \begin{equation}\label{linkProjectionInclusionProof}
        \text{the projection of the $\hat X$-link for~$\hat F_1$ is included in the $X$-link for~$F_1$.}
    \end{equation}
    So let $\hat x$ be a vertex in the $\hat X$-link for~$\hat F_1$, and let $\hat W\in\cW_r(\hat X)$ starting at $\hat x$ whose first edge $\hat e$ in $\partial \hat Y$ lies in $\hat F_1$.
    If $e:= p_r(\hat e)$ is also first edge $f$ of $W := p_r(\hat W)$ in $\partial Y$, then the projection $x := p_r(\hat x)$ is in the $X$-link for $F_1$ and we are done.
    
    So let $f$ be the first edge of $W$ which is in $\partial Y$.
    Let $\hat f$ be its lift in $\hat W$, and let $\hat Z$ be the $r$-tom of $p_r\inv(Y)$ whose $\partial \hat Z$ contains $\hat f$.
    Since $\hat W \in \cW_r(\hat X)$, $\hat f$ either lies on a short cycle that meets $\hat X$ in at least two vertices or is an $\hat X$-edge.
    Since $\hat Y$ is the unique $r$-tom of $p_r\inv(Y)$ that is $r$-coupled with $\hat X$, \cref{uniqueRcoupleLift} yields $\hat Z=\hat Y$.
    In particular, $\hat f = \hat e$, as desired.
    
    Both \cref{linkLiftInclusionProof,linkProjectionInclusionProof} extend to all four links by analogue arguments, so \cref{item:LiftXlink,item:LiftYlink} hold.
    The `moreover'-part is immediate from \cref{item:LiftXlink,item:LiftYlink,item:LiftCentre}.
\end{proof}

\begin{liftlemmaASS}[Crossing]\label{NestedRtomic}
    Let $s$ and $t$ be two $r$-local separations of a connected graph $G$ with $r$-tomic separators $X$ and $Y$ of size~$\le k$, respectively, such that~$(k-1)r<\mindist_r(G)$.
    The following assertions are equivalent:
    \begin{enumerate}
        \item \label{item:sCrossest} $s$ crosses $t$;
        \item \label{item:liftsCrossesLiftt}some lift of $s$ crosses some lift of $t$;
        \item \label{item:CrossingMatching} every lift of $s$ crosses exactly one lift of $t$.
    \end{enumerate}
\end{liftlemmaASS}

\noindent In the context of \cref{NestedRtomic}, we recall that $s$ and $t$ can be lifted by \cref{sufficientForLifting} and \cref{randomInequality}.

\begin{proof}
\cref{item:sCrossest}$\to$\cref{item:liftsCrossesLiftt}.
This follows directly from the `moreover'-part of \cref{LiftingLinks}. (Note that \cref{item:sCrossest} implies, by definition of 'cross', that $s$ and $t$ are $r$-coupled, which \cref{LiftingLinks} requires.)

\cref{item:liftsCrossesLiftt}$\to$\cref{item:CrossingMatching}.
As the lifts of $s$ and the lifts of $t$ each form an orbit under the action of the group of deck transformations of $p_r$ on the $r$-locally closed lifts of $X$ (equivalently by \cref{RtomicGives1sheeted}: $r$-toms of $p_r\inv(X)$), it suffices to show that some lift of $s$ is crossed by precisely one lift of $t$.
Let $\hat s$ and $\hat t$ be the respective lifts of $s$ and $t$ that cross by \cref{item:liftsCrossesLiftt}.
By the definition of crossing, their respective separators $\hat X$ and $\hat Y$ are $r$-coupled.
Since $p_r$ isomorphically maps short cycles of $G_r$ to short cycles in $G$, $X$ and $Y$ are $r$-coupled too.
Thus, \cref{uniqueRcoupleLift} yields that $\hat t$ is the only lift of $t$ that crosses~$\hat s$, as desired.

\cref{item:CrossingMatching}$\to$\cref{item:sCrossest}.
By \cref{sufficientForLifting} and \cref{randomInequality}, $s$ and $t$ have lifts. 
By \cref{item:CrossingMatching}, some of these cross, say $\hat s$ and $\hat t$ with separators $\hat X$ and $\hat Y$.
As in the proof of \cref{item:liftsCrossesLiftt}$\to$\cref{item:CrossingMatching}, $X$ and $Y$ are $r$-coupled.
By the `moreover'-part of \cref{LiftingLinks}, $s$ and $t$ cross, as desired.
\end{proof}

We will not use the following lemma, but have included it for completeness.

\begin{lemmaASS}\label{nearPartition}
    Let $s$ and $t$ be $r$-local separations of a connected graph~$G$ with $r$-coupled, $r$-tomic separators $X$ and $Y$ of size~$\le k$, respectively, and assume that $(k-1)r<\mindist_r(G)$.
    Then the four links and $X\cap Y$ together form a near-partition of $X\cup Y$, and there are no edges in $G$ between opposite links.
\end{lemmaASS}
\begin{proof}
    Recall that $s$ and $t$ can be lifted by \cref{sufficientForLifting} and \cref{randomInequality}.
    Let $\hat s$ be an arbitrary lift of $s$ with separator~$\hat X$.
    By \cref{uniqueRcoupleLift}, there is a unique lift $\hat t$ of $t$ whose separator~$\hat Y$ is $r$-coupled with $\hat X$.
    Then by \cref{LiftingLinks}, $p_r$ restricts to bijections between the respective pairs of links in $G_r$ and~$G$ and between $\hat X\cap \hat Y$ and $X\cap Y$.
    Now let $\{\hat A_1,\hat A_2\}$ and $\{\hat C_1,\hat C_2\}$ be the separations of $G_r$ that induce $\hat s$ and~$\hat t$, respectively, by \cref{correspondenceofseparationswithfixedseparator}.
    The links with regard to $\{\hat A_1,\hat A_2\}$ and $\{\hat C_1,\hat C_2\}$ together with $\hat X\cap \hat Y$ form a near-partition of $\hat X\cup\hat Y$ by definition.
    By \cref{corresLinks}, the links with regard to $\hat s$ and $\hat t$ together with $\hat X\cap\hat Y$ form the same near-partition of $\hat X\cup\hat Y$.
    Finally, since by \cref{LiftingLinks} the covering map $p_r$ restricts to bijections between the pairs of links in $G_r$ and~$G$ and between $\hat X\cap \hat Y$ and $X\cap Y$, also the links and $X\cap Y$ together form a near-partition of $X\cup Y$.

    Finally, assume for a contradiction that there is an edge $x_1 x_2$ such that the $x_i$ are contained in opposite $X$-links; the $Y$-links-case then follows by symmetry.
    Let $\hat x_i$ denote the lifts of the $x_i$ in~$\hat X$.
    Then $\hat x_1 \hat x_2$ is an edge in $G_r$ between opposite $\hat X$-links.
    But then $\hat x_1 \hat x_2$ is an edge between $\hat C_1\sm\hat C_2$ and $\hat C_2\sm\hat C_1$, contradicting the fact that $\{\hat C_1,\hat C_2\}$ is a separation.
\end{proof}

\section{Local bottlenecks}\label{sec:LocalBottlenecks}

In this section, we introduce and study one of the centrepieces of our theory of $r$-local separations: \emph{$r$-local bottlenecks}.
For all relevant concepts, we shall follow a recurring scheme: we define the local notions, show that they agree with the respective global counterpart on $r$-local covers, and prove lifting/projection lemmas under a suitable assumption on the displacement of~$p_r$.
We work along this three-step process first for \tstar s (\cref{TriStarCorrespondence}, \cref{ProjectionLocalYStar} and \cref{LiftLocalYStar}), then for bottlenecks (\cref{correspondenceBottlenecks}, \cref{ProjectingLocalBottlenecks} and \cref{LiftingLocalBottlenecks}) and finally for a construction of nested sets of $r$-local separations that represent the $r$-local bottlenecks (\cref{construction:inductiveNestedSetLocal} , \cref{CorrespondenceNestedSet}, \cref{projectionliftNestedSet} and \cref{LocalNestedToT}).

\subsection{Local and global \texorpdfstring{\tstar s}{tri-stars} in the local cover}

We need a way to localise \tstar s. 
To do this, let us first characterise \tstar s in more local terms than by their definition:

\begin{lemma}\label{CharacterisationOfTStar}
    A set $\sigma=\{\,(A_i,B_i):i\in [3]\,\}$ of three distinct oriented separations $(A_i,B_i)$ of a graph~$G$ with separators~$X_i$ is a \tstar\ if and only if 
    \begin{enumerate}[label=\rm{(Y\arabic*)}]
        \item\label{Y1} the sets $A_i \setminus B_i$ for $i \in [3]$ form a near-partition of the vertex set $V(G) \setminus (\bigcup_{i \in [3]} X_i)$, and
        \item\label{Y2} every vertex in some $X_i$ is contained in at least one other $X_j$. 
    \end{enumerate}
\end{lemma}

\begin{proof}
    This is immediate from the definition of \tstar.
\end{proof}

An \defn{$r$-local \tstar } in~$G$ is a set $\sigma=\{\,(E_i,X_i,F_i):i\in [3]\,\}$ of three oriented $r$-local separations such that 
\begin{enumerate}[label=\rm{(LY\arabic*)}]
    \item\label{LY1} the $E_i$ form a near-partition of the edge set $\partial (\bigcup_{i \in [3]} X_i)$,
    \item\label{LY2} every vertex in some $X_i$ is contained in at least one other~$X_j$, and
    \item\label{LY3} every $r$-local component at $\bigcup_{i \in [3]} X_i$ is included in some~$E_i$.
\end{enumerate} 
The unoriented $r$-local separations $\{E_i,X_i,F_i\}$ are the \defn{constituents} of~$\sigma$.
We say that $\sigma$ has \defn{order~$\le k$} if all its elements have order~$\le k$.
If an $r$-local \tstar\ $\sigma$ is introduced as $\{\,(E_i,X_i,F_i):i\in [3]\,\}$, then we denote the union of the separators $X_i$ by $X$, the \defn{centre} $\bigcap_{i \in [3]} X_i$ by $Z$ and the \defn{$ij$-link} $(X_i \cap X_j) \setminus Z$ by $X_{ij} = X \setminus X_k$ whenever $\{i,j,k\} = [3]$.
We remark that $\{E_i,X, \partial X \setminus E_i\}$ is an $r$-local separation by \cref{LY1} and \cref{LY3}.

An $r$-local \tstar\ $\sigma = \{\,(E_i,X_i,F_i) : i\in [3]\,\}$ is \defn{relevant with base} $\{E_1,X_1,F_1\}$ if $\{E_1,X_1,F_1\}$ is tight and for every $x \in X_{23}$ and for every $i \in \{2,3\}$ there exists an $x$--$X_{1i}$ edge in $G$ or there exists an $r$-local component $F \subseteq E_i$ at $X$ which contains an edge incident to $x$ and one incident to $X_{1i} \cup Z$.

\begin{lemma}\label{relevantLocalTriStarRtomic}
    Let $\sigma=\{\,(E_i,X_i,F_i) : i\in [3]\,\}$ be a relevant $r$-local \tstar\ in a connected graph~$G$ with base $\{E_1,X_1,F_1\}$.
    Then $X:=\bigcup_{i\in [3]}X_i$ is $r$-tomic.
\end{lemma}
\begin{proof}
    Since $\{E_1,X_1,F_1\}$ is tight, the separator $X_1$ is $r$-tomic by definition (\cref{TightLocalSeparatorAreOConnected}).
    Now let $x$ be a vertex in the $23$-link.
    Since $\sigma$ is relevant with base $\{E_1,X_1,F_1\}$, there is an $x$--$X_{12}$ edge $e$ or there is an $r$-local component $F\se E_2$ at $X$ which contains an edge incident to $x$ and an edge incident to $X_{12}\cup Z$, where $Z$ is the centre $\bigcap_{i\in [3]}X_i$.
    If $e$ exists, then its two endvertices, i.e.\ $x$ and some vertex in $X_1$, lie in the same $r$-tom by the definition of $r$-tom, as desired.
    The same is ensured by the definition of $r$-local component, if $F$ exists.
\end{proof}

Given a set $\sigma$ of (possibly oriented) separations of $G$, we write $\defnMath{\sigma_r}:=\{s_r : s \in \sigma\}$ for the set of $r$-local separations of $G$ induced by the separations in $\sigma$.

\begin{corlemma}[Tri-star]\label{TriStarCorrespondence}
    Assume that the binary cycle space of $G$ is generated by cycles of length $\leq r$ and that $G$ is connected.
    Then $(\sigma,s)\mapsto (\sigma_r,s_r)$ defines a bijection between the set of all pairs $(\sigma,s)$ of relevant \tstar s $\sigma$ in~$G$ with base~$s$ and the set of all pairs $(\tau,t)$ of relevant $r$-local \tstar s $\tau$ in $G$ with base~$t$.
\end{corlemma}

\begin{proof}
    Recall that, by \cref{CharacterisationOfTStar}, \tstar s are characterised by \cref{Y1} and \cref{Y2}.

    \casen{Well-defined.}
    Let $\sigma=\{\,(A_i,B_i):i\in [3]\,\}$ be a relevant \tstar\ in~$G$ with base~$\{A_1,B_1\}$.
    For every~$i$, let $(E_i,X_i,F_i)$ denote the $r$-local separation induced by~$(A_i,B_i)$.
    Then $\{E_1,X_1,F_1\}$ is tight by \cref{keylemma:correspondenceofseparations}, and $X:=\bigcup_{i\in [3]}X_i$ is $r$-tomic by \cref{relevantTriStarRtomicInCovering}.

    Let us now show that $\sigma_r = \{\,(E_i,X_i,F_i) : i \in [3]\,\}$ is an $r$-local \tstar .
    \cref{LY1} and \cref{LY2} follows immediately from \cref{Y1} and \cref{Y2}.
    For \cref{LY3}, let $F$ be an $r$-local component at $X$.
    Since $X$ is $r$-tomic, there is a component $K$ of $G - X$ such that $F = E(K,X)$ by \cref{keylemma:correspondenceofcomponents}. 
    As the $(A_i,B_i)$ are separations of $G$ which satisfy \cref{Y1}, every component of $G - X$, in particular $K$, has its vertices in some $A_i \setminus B_i$.
    Thus, $F \subseteq E(X, A_i \setminus B_i) \subseteq E_i$ as $(E_i,X_i,F_i)$ is induced by $(A_i,B_i)$.
    Hence, \cref{LY3} holds.

    Since $\sigma$ is relevant with base $\{A_1,B_1\}$ and $X$ is $r$-tomic, it follows from \cref{keylemma:correspondenceofseparations} and \cref{keylemma:correspondenceofcomponents} (applied to the $r$-local separations $\{E_i,X, \partial X \setminus E_i\}$ for $i=2,3$) that $\sigma_r$ is relevant with base $\{E_1,X_1,F_1\}$.

    \casen{Injective.}
    Since $G$ is connected, this follows from the definition of $r$-local separations induced by separation, i.e.\ the map $s \mapsto s_r$.

    \casen{Surjective.}
    Let $\tau=\{\,(E_i,X_i,F_i):i\in [3]\,\}$ be a relevant $r$-local \tstar\ in $G$ with base~$t=\{E_1,X_1,F_1\}$.
    We have to find a pair $(\sigma,s)$ with $(\sigma_r,s_r)=(\tau,t)$.

    First, let us find separations $(A_i,B_i)$ of $G$ that induce the~$(E_i,X_i,F_i)$.
    Since $X:=\bigcup_{i\in [3]}X_i$ is $r$-tomic, by \cref{relevantLocalTriStarRtomic}, the $r$-local separations $(E_i, X, \partial X \setminus E_i)$ are induced by separations $(A_i', B_i')$ of~$G$,
    by \cref{correspondenceofseparationswithfixedseparator}, so, $E(X,A_i' \setminus B_i') = E_i$.
    Now let $A_i := (A_i' \setminus B_i') \cup X_i$ and $B_i := B_i'$.
    Then $A_i\sm B_i=A'_i\sm B'_i$ and $E_G(X_i, A_i \setminus B_i) = E_i$.
    Thus, $E(X_i, B_i \setminus A_i) = \partial X_i \setminus E_i = F_i$.
    Hence, $(A_i,B_i)$ induces $(E_i,X_i,F_i)$.

    Second, we claim that $\sigma := \{\,(A_i,B_i) : i \in [3]\,\}$ is a \tstar .
    The $(E_i, X, \partial X \setminus E_i)$ are $r$-local separations whose sets $E_i$ near-partition~$\partial X$, by \cref{LY1}.
    Hence, for the corresponding separations $(A'_i,B_i')$ the sets $A'_i\sm B'_i$ near-partition the components of $G-X$, by \cref{LY3} and \cref{keylemma:correspondenceofcomponents}.
    Since $A'_i\sm B'_i=A_i\sm B_i$, the sets $A_i\sm B_i$ near-partition $V(G)\sm X$, giving \cref{Y1}.
    As \cref{LY2} is identical to \cref{Y2}, \cref{CharacterisationOfTStar} implies that $\sigma$ is a \tstar .

    It remains to show that $\sigma$ is relevant with base $\{A_1,B_1\}$.
    Since $X$ is $r$-tomic, this follows from the fact that $\tau$ is relevant with base~$\{E_1,X_1,F_1\}$ by \cref{keylemma:correspondenceofseparations} and \cref{keylemma:correspondenceofcomponents} (applied to the $r$-local seps $\{E_i,X,\partial X \setminus E_i\}$ for $i=2,3$).
    We have $(\sigma,\{A_1,B_1\})\mapsto (\tau,\{E_1,X_1,F_1\})$ by construction.
\end{proof}

\subsection{Interplay of local \texorpdfstring{\tstar s}{tri-stars} with the covering map}

We define a function $f \colon \N \to \N$ such that $f(k) \cdot r/2 < \mindist_r(G)$ is a sufficient condition for the union of the separators of a relevant $r$-local \tstar\ to be $r$-locally $1$-covered, but which also satisfies $f(k) \leq 2(k-1)$ for every positive $k \in \N$: 
\begin{equation}
    \text{$\defnMath{f(k)} := \lfloor 3k/2 \rfloor$ for $k \geq 3$ and $f(k) := 0$ for $k \leq 2$.}
\end{equation}

\begin{lemma}\label{evenmoreRandomInequality}
    $f(k) \leq 2(k-1)$ for every positive $k \in \N$.
\end{lemma}

\begin{proof}
    For $k \leq 2$, the statement is trivial.
    For $k = 3$, we have $\lfloor 3k/2 \rfloor = 4 = 2(k-1)$, and for $k \geq 4$ we have $2(k-1)-\lfloor 3k/2\rfloor\ge k/2-2\ge 0$.
\end{proof}

\begin{lemma}\label{ThirdrandomInequality}
    $f(k) \cdot r/2<\mindist_r(G)$ implies $kr/2<\mindist_r(G)$, for all $k,r\in\N$.
\end{lemma}
\begin{proof}
    If $k\le 2$ then $kr/2 \leq r<\mindist_r(G)$, by~\cref{mindistGeR}, while $k\ge 3$ gives $\mindist_r(G) > f(k) \cdot r/2 = \lfloor 3k/2 \rfloor \cdot r/2 \geq kr/2$.
\end{proof}

The following \cref{ShortCycleThroughComponents} is of the same flavour as \cite[Lemma~3.10]{Local2sep}.

\begin{lemma}\label{ShortCycleThroughComponents}
    Let $\{E_1,X,E_2\}$ be a tight $r$-local $2$-separation of a connected graph~$G$.
    There exists a short cycle $O\se G$ that is a union of two $X$-paths $P_1,P_2$ such that $P_1$ starts and ends in edges in~$E_1$, while $P_2$ starts and ends in edges in $E_2$ or $P_2$ is an $X$-edge.
\end{lemma}
\begin{proof}
    Let $\{x,y\}:=X$.
    Let $F_i$ be a tight component in $E_i$ for $i=1,2$.
    For each $i$, pick a short cycle $O_i\se G$ that includes an $X$-path 
    starting and ending with edges in $F_i$. 
    If $O_1$ uses an edge in~$E_2$, or if the edge $xy$ exists and is used by~$O_1$ then $O:=O_1$ is the union of the desired paths, and similarly for $O_2$.
    Hence, we may assume that both $X$-paths in $O_i$ start and end with edges in~$E_i$, for both~$i$.
    Let $P_i$ be a shortest $X$-path in~$O_i$, for both~$i$, so the $P_i$ have length~$\le r/2$.
    Then $P_1,P_2$ and $O:=P_1\cup P_2$ are as desired.
\end{proof}

\begin{lemma}\label{SmallLocalYStarIs1Sheeted}
    Let $\sigma$ be a relevant $r$-local \tstar\ of order at most $2$ in a connected graph~$G$.
    Then the union $X$ of the separators of the elements of $\sigma$ has size $\leq 2$ or some short cycle in~$G$ contains~$X$. 
\end{lemma}

\begin{proof} 
    Let $\{\,(E_i, X_i, F_i) : i \in [3]\,\}:=\sigma$ and let $\{E_1,X_1,F_1\}$ be its base. 
    Assume that $X$ has size~$\geq 3$.

    We claim that all $X_i$ have size two, all links have size one, and the centre is empty.
    If an $X_i$ has size one, then it is included in another $X_j$ by~\cref{LY2}, and since the third $X_k$ must also be included in $X_i \cup X_j = X_j$ by~\cref{LY2}, this gives $|X|\le 2$, a contradiction.
    So all $X_i$ have size two.
    
    If two separators are equal, then also the third separator must be equal to these two by \cref{LY2}, contradicting our assumption that $|X|\ge 3$.
    Thus, all the~$X_i$ are distinct, and by \cref{LY2} every two $X_i$ share an element.
    Hence, the centre of the \tstar\ is empty and every $ij$-link consists of a unique vertex, which we denote by~$x_{ij}$.
    Thus, $X=\{x_{12},x_{23},x_{13}\}$.

    Since $\{E_1,X_1,F_1\}$ is tight, we may use \cref{ShortCycleThroughComponents} to find a short cycle $O\se G$ that is a union of two $X_1$-paths $P,Q$ such that $P$ starts and ends in edges in~$E_1$ or is the single edge~$x_{12}x_{13}$, and such that $Q$ starts and ends in edges in~$F_1$; in particular, $O$ meets both $x_{12}$ and $x_{13}$.
    Let $e$ denote the edge on~$P$ incident to~$x_{12}$.
    Let $f$ denote the edge on~$Q$ incident to~$x_{12}$.

    If the other endvertex of $f$ is~$x_{23}$, then $O$ contains~$X$ as desired.
    If not, then $f \in \partial X \setminus \partial X_3$.
    Since $f \in F_1$, we have $f \in \partial X \setminus (E_1 \cup E_3)$ and hence $f\in E_2$ by \cref{LY1}.
    Let us show that $e \in F_2$.
    If $e = x_{12}x_{13}$, then $e \in \partial X_2 \setminus \partial X$, so $e \in F_2$ by \cref{LY1}.
    Otherwise, $e$ lies in~$E_1\se\partial X$ as well as in $\partial X_2$, so $e\in F_2$ by \cref{LY1}.
    Altogether, $e \in F_2$ and $f \in E_2$ lie in different sides of the $r$-local separation $\{E_2,X_2,F_2\}$.
    As $O$ is a short cycle containing both $e$ and~$f$, it must meet $X_2$ in both its elements, and especially in~$x_{23}$.
    Hence, $O$ contains~$X$.
\end{proof}

\begin{lemma}\label{SmallSetsAre1Sheeted}
    Let $X$ be an $r$-tomic set of vertices in a connected graph~$G$.
    If $|X| \leq 2$ or $X$ is contained in a short cycle, then $X$ is $r$-locally $1$-covered.
\end{lemma}

\begin{proof}
    If $|X| \leq 2$, then $\mindist_r(G) > r \geq k \cdot r/2$ for $k\le 2$, \cref{mindistGeR}, so $X$~is $r$-locally $1$-covered by \cref{RtomicGives1sheeted}.

    Now assume that $X$ is contained in some short cycle $O\se G$.
    Let $\hat O$ be a lift of $O$.
    Then $\hat O$ is a short cycle too, so $\hat X := V(\hat O) \cap p_r\inv(X)$ is contained in an $r$-tom of $p_r\inv(X)$.
    Let us show that this consists of $\hat X$ only.
    
    If not, there exists a short cycle which joins some $\hat x \in \hat X$ to some $\hat y \in p_r\inv(X) \setminus \hat X$.
    In particular, $d_{G_r}(\hat y, \hat x) \leq r/2$.
    Let $\hat y' \in \hat X$ be in the same fibre as $\hat y$.
    Then 
    \[
        d_{G_r}(\hat y, \hat y') \leq d_{G_r}(\hat y, \hat x) + d_{G_r}(\hat x, \hat y') \leq r/2 + r/2 = r,
    \]
    which contradicts \cref{mindistGeR}.
\end{proof}

\begin{lemmaASS}\label{LocalYStarIs1SheetedLift}
    Let $\sigma$ be a relevant $r$-local \tstar\ of order $\leq k$ in a connected graph $G$ such that $f(k) \cdot r/2 < \mindist_r(G)$.
    Then the union $X$ of the separators of the elements of $\sigma$ is $r$-locally $1$-covered.
\end{lemmaASS}

\begin{proof}
    The set $X$ is $r$-tomic by \cref{relevantLocalTriStarRtomic}.
    If $k \leq 2$, then $X$ is $r$-locally $1$-covered by \cref{SmallLocalYStarIs1Sheeted} and \cref{SmallSetsAre1Sheeted}.
    So we may assume that $k \geq 3$.
    By \cref{LY2}, $X$~has size $\leq \lfloor 3k/2 \rfloor$.
    Then $X$ is $r$-locally $1$-covered by \cref{RtomicGives1sheeted}, as $\lfloor 3k/2 \rfloor\cdot r/2=f(k)\cdot r/2<\mindist_r(G)$.
\end{proof}

\begin{lemmaASS}\label{LocalYstarIs1SheetedProjection}
    Let $G$ be a connected graph such that $f(k) \cdot r/2 < \mindist_r(G)$.
    Let $\hat \sigma$ be a relevant $r$-local \tstar\ of order $\leq k$ in $G_r$.
    Then the projection $X:=p_r(\hat X)$ of the union $\hat X$ of the separators of the elements of $\hat \sigma$ is $r$-locally $1$-covered. Moreover, $\hat X$ is an $r$-tom of $p_r\inv(X)$, and thus a lift of $X$.
\end{lemmaASS}

\begin{proof}
    The set $\hat X$ is $r$-tomic in~$G_r$ by \cref{relevantLocalTriStarRtomic}.
    Hence, $X$ is $r$-tomic by \cref{ProjectionRtomic}.
    
    If $k \leq 2$ then $\hat X$ ,and thus also $X$, either has size $\leq 2$ or is contained in a short cycle, by  \cref{SmallLocalYStarIs1Sheeted}. 
    Hence $X$ is $r$-locally $1$-covered by \cref{SmallSetsAre1Sheeted}.

    So we may assume $k \geq 3$.
    By \cref{LY2}, $\hat X$ has size $\leq \lfloor 3k/2 \rfloor$, hence so does its projection $X$.
    Thus, $X$ is $r$-locally $1$-covered by \cref{RtomicGives1sheeted} and $\lfloor 3k/2 \rfloor\cdot r/2=f(k)\cdot r/2<\mindist_r(G)$.

    In both cases, $\hat X$ is an $r$-tom of $p_r\inv(X)$, since the $r$-tomic set $\hat X$ is contained in an $r$-tom of $p_r\inv(X)$ $r$-tomic and its projection $X = p_r(\hat X)$ is $r$-locally $1$-covered.
    By the definition of $r$-locally $1$-covered, $\hat X$ is a lift of $X$.
\end{proof}

\begin{projlemmaASS}[Local \tstar ]\label{ProjectionLocalYStar}
    Let $G$ be a connected graph such that $f(k) \cdot r/2 < \mindist_r(G)$.
    Let $\hat \sigma$ be a relevant $r$-local \tstar\ in $G_r$ of order $\leq k$ with base~$\hat s$. 
    Then $p_r(\hat \sigma)$ is a relevant $r$-local \tstar\ in~$G$ of order $\leq k$ with base $p_r(\hat s)$.
\end{projlemmaASS}

\begin{proof}
    Let $\{\,\hat s_i : i\in [3]\,\}:=\hat\sigma$ where $\hat s_i=:(\hat E_i,\hat X_i,\hat F_i)$ so that $\hat s=: \hat s_1$.
    Let $\hat X:=\bigcup_{i\in [3]}\hat X_i$.
    By \cref{LocalYstarIs1SheetedProjection}, $X := p_r(\hat X)$ is $r$-locally $1$-covered and $\hat X$ is an $r$-tom of~$p_r\inv (X)$, and thus a lift of $X$.
    Hence, the subsets $\hat X_i$ of $\hat X$ are $r$-locally closed lifts of their projections $X_i$.
    Thus, by \cref{ProjectionLocalSeparation}, each $\hat s_i$ projects to an $r$-local separation $s_i=(E_i,X_i,F_i)$ of~$G$, which is tight for $i =1$.
    
    Since $\hat X$ is a lift of $X$, \cref{LY1} and \cref{LY2} for $p_r(\hat \sigma) = \{\, s_i : i \in [3] \,\}$ follow immediately from \cref{LY1} and \cref{LY2} of~$\hat \sigma$.
    \cref{LY3} and the relevance of $\hat \sigma$ with base $\hat s_1$ imply the same for $\sigma$ and $s_1$ by \cref{keylemma:liftofcomponents}.
\end{proof}

\begin{liftlemmaASS}[Local \tstar ]\label{LiftLocalYStar}
    Let $G$ be a connected graph such that $f(k) \cdot r/2 < \mindist_r(G)$.
    Let $\sigma = \{\,(E_i, X_i, F_i) : i \in [3]\,\}$ be a relevant $r$-local \tstar\ in~$G$ of order $\leq k$ with base $\{E_1, X_1, F_1\}$.     
    Let $\hat X$ be any $r$-tom of $p_r\inv(X)$.
    Then there are unique lifts $(\hat E_i,\hat X_i,\hat F_i)$ of the $(E_i,X_i,F_i)$ with $\hat X_i\se \hat X$ via \cref{keylemma:liftofseparations}, and
    $\hat\sigma:=\{\,(\hat E_i,\hat X_i,\hat F_i) : i\in [3]\,\}$ is a relevant $r$-local \tstar\ in~$G_r$ of order~$\le k$ with base $\{\hat E_1,\hat X_1,\hat F_1\}$.
\end{liftlemmaASS}

\begin{proof}
    Recall that $X: = \bigcup_{i \in [3]} X_i$ is $r$-locally $1$-covered by \cref{LocalYStarIs1SheetedLift}, so $\hat X$ is a lift of $X$.
    Thus,
    each $X_i$ has a unique lift $\hat X_i\se \hat X$, which is $r$-locally closed since $\hat X$ is $r$-locally closed as an $r$-tom of $p_r\inv (X)$.
    Hence the lifts $\hat s_i = (\hat E_i,\hat X_i,\hat F_i)$ defined as in \cref{keylemma:liftofseparations} exist and are unique, and $\{\hat E_1,\hat X_1,\hat F_1\}$ is tight.
    As $\hat X$ is a lift of $X$, \cref{LY1} and \cref{LY2} transfer from $\sigma$ to $\hat \sigma := \{\, \hat s_i : i \in [3]\,\}$.
    \cref{keylemma:liftofcomponents} ensures that \cref{LY3} and relevance transfer from $\sigma$ to~$\hat \sigma$.
\end{proof}

\subsection{Local bottlenecks and nested sets of local separations}

Let $k\in\N$.
An \defn{$r$-local bottleneck of order~$k$} in~$G$, or \defn{$r$-local $k$-bottleneck}, is a non-empty set $\beta$ of tight $r$-local $k$-separations of $G$ satisfying:
\begin{enumerate}[label=(B$_r$)]
    \item\label{localEntangle} 
    Whenever $\beta$ contains the base of a relevant $r$-local \tstar\ $\sigma$ in $G$ of order~$\le k$, then $\beta$ contains at least one of the other two constituents of $\sigma$.
\end{enumerate}

\noindent We note that since the $r$-local separations in an $r$-local bottleneck are tight, their separators are $r$-tomic by \cref{TightLocalSeparatorAreOConnected}.

\begin{corlemma}[Local bottlenecks]\label{correspondenceBottlenecks}
    Assume that the binary cycle space of $G$ is generated by cycles of length~$\le r$ and that $G$ is connected.
    The map $\beta\mapsto \beta_r:=\{s_r:s\in \beta\}$ given by \cref{keylemma:correspondenceofseparations} is a bijection between the $k$-bottlenecks $\beta$ in~$G$ and the $r$-local $k$-bottlenecks in~$G$.
\end{corlemma}
\begin{proof}
    \cref{TriStarCorrespondence} ensures that $\beta_r$ satisfies \cref{localEntangle} if and only if $\beta$ satisfies \cref{Entangle1}.
\end{proof}

\begin{projlemmaASS}[Local bottlenecks]\label{ProjectingLocalBottlenecks}
    Let $G$ be a connected graph such that $f(k) \cdot r/2< \mindist_r(G)$.
    Let $\hat\beta$ be an $r$-local $k$-bottleneck in~$G_r$.
    Then $p_r(\hat\beta)$ is an $r$-local $k$-bottleneck in~$G$.
\end{projlemmaASS}
\begin{proof}
    Hence, by \cref{sufficientForProjection} and \cref{ThirdrandomInequality}, we may use \cref{ProjectionLocalSeparation} to find that $\beta:=p_r(\hat\beta)$ is defined and consists of tight $r$-local $k$-separations of~$G$.
    In particular, $\beta$ is non-empty.
    
    To show that $\beta$ exhibits~\cref{localEntangle}, let $\sigma$ be a relevant $r$-local \tstar\ of order $\leq k$ with base~$s \in \beta$.
    Take any $\hat s\in\hat\beta$ with $p_r(\hat s)=s$.
    Let $X$ denote the union of the separators of the elements of~$\sigma$, and let $\hat X$ be the $r$-tom of $p_r\inv(X)$ that includes the $r$-tomic separator of~$\hat s$.
    By \cref{LiftLocalYStar}, the elements of $\sigma$ have unique lifts to $G_r$ whose separators are included in $\hat X$, and these form a relevant $r$-local \tstar\ $\hat\sigma$ in~$G_r$ of order~$\le k$ with base~$\hat s$.
    By \cref{localEntangle} for~$\hat \beta$, some constituent $\hat t$ of $\hat\sigma$ other than~$\hat s$ is contained in~$\hat\beta$.
    So the corresponding constituent $p_r(\hat t\,)$ of $\sigma$ is in $\beta$ and distinct from~$s$.
    Hence, $\beta$ satisfies \cref{localEntangle}.
\end{proof}

\begin{liftlemmaASS}[Local bottlenecks]\label{LiftingLocalBottlenecks}
    Let $G$ be a connected graph.
    Let $\beta$ be an $r$-local $k$-bottleneck in $G$ such that $f(k) \cdot r/2 < \mindist_r(G)$.
    Then the elements of $\beta$ can be lifted to $G_r$ and the set $\hat\beta$ of all such lifts is an $r$-local $k$-bottleneck in~$G_r$.
\end{liftlemmaASS}
\begin{proof}
    By \cref{sufficientForLifting} and \cref{ThirdrandomInequality}, the elements of $\beta$ have $r$-locally $1$-covered separators and lift to~$G_r$.
    All these lifts, the elements of $\hat\beta$, are tight $r$-local $k$-separations of~$G_r$ by \cref{keylemma:liftofseparations}; in particular, $\hat \beta$ is non-empty.
    To show that $\hat\beta$ exhibits~\cref{localEntangle}, let $\hat \sigma$ be a relevant $r$-local \tstar\ in~$G_r$ of order $\leq k$ with base $\hat s \in \hat \beta$.
    Then $s := p_r(\hat s)$ is in~$\beta$, by definition of~$\hat \beta$.
    The set $\sigma$ of the projections of the elements of~$\hat\sigma$ is a relevant $r$-local \tstar\ in~$G$ of order $\leq k$ with base $s$, by \cref{ProjectionLocalYStar}.
    By \cref{localEntangle} for $\beta$, we know that $\beta$ contains a second constituent $t$ of $\sigma$ besides~$s$.
    Then $\hat\beta$ contains all lifts of $t$ to $G_r$, in particular, the one that is a constituent of~$\hat\sigma$.
    Hence, $\hat \beta$ satisfies~\cref{localEntangle}.
\end{proof}

\begin{construction}[Nested local separations from local bottlenecks]\label{construction:inductiveNestedSetLocal}
Let $G$ be a graph and $r \geq 0$ an integer, and let $\cB_r$ be any set of $r$-local bottlenecks.
For every $k \in \N$, let $\defnMath{\cB_r^k(G)}$ be the set of the $r$-local $k$-bottlenecks in $\cB$ and $S^k_r := \bigcup \cB_r^k(G)$ for every $k\in\N$.
Recursively for all $k\in\N$, define sets $N_r^k(G) \subseteq S^k_r$, as follows.
Let $\defnMath{x^k_r} \colon S^k_r \to \N$ assign to each $r$-local separation  $s \in X^k_r$ the number of $r$-local separations in $X^k_r$ that $s$ crosses assuming for now that this number is finite.
Write $\defnMath{N_r^{<k}(G)}:=\bigcup_{j<k}N_r^j(G)$.
For each $\beta\in\cB_r^k(G)$ let $\defnMath{N_r^k(G,\beta)}$ consist of the $r$-local separations $s\in\beta$ which are nested with $N_r^{<k}(G)$ and among those have minimal $x^k(s)$.
Finally, set $\defnMath{N_r^k(G)}:=\bigcup_{\beta\in\cB_r^k} N_r^k(G,\beta)$.
\end{construction}

\noindent Whenever we use the notation of $N^k_r(G)$ without reference to any set $\cB_r$ of $r$-local bottlenecks, we set $\cB_r$ to be the set of all $r$-local bottlenecks in~$G$. 

In full analogy to separations, we say that a set~$S$ of $r$-local separations of~$G$ is~\defn{$\Gamma$-canonical} for some subgroup~$\Gamma$ of~$\Aut(G)$ if $S$ is invariant under the action of~$\Gamma$ on the $r$-local separations of~$G$, as given by~$\{E_1, X, E_2\} \mapsto \{\phi(E_1), \phi(X), \phi(E_2)\}$ for~$\phi \in \Aut(G)$.
In the case of~$\Gamma = \Aut(G)$, we just say that~$S$ is \defn{canonical}.
Further, a set~$\cB_r$ of $r$-local bottlenecks in~$G$ is \defn{$\Gamma$-invariant} for a subgroup~$\Gamma$ of $\Aut(G)$ if the action of~$\Gamma$ on the $r$-local bottlenecks in~$G$ (which is inherited from the action of~$\Aut(G)$ on the $r$-local separations of~$G$) is invariant on~$\cB$.

\begin{corlemma}[Nested set]\label{CorrespondenceNestedSet}
    Let $G$ be a locally finite, connected graph, let $\cB$ be a set of bottlenecks in $G$, and let $\cB_r = \{\beta_r : \beta \in \cB\}$ be the set of $r$-local bottlenecks given by \cref{correspondenceBottlenecks}.
    Assume that the binary cycle space of $G$ is generated by cycles of length $\leq r$.
    Then the sets $N_r^k(G)$ are well-defined and consist of the $r$-local separations $s_r$ induced by the separations $s$ in $N^k(G)$ given by \cref{construction:inductiveNestedSet}.
    In particular, $N_r^k(G)$ is a nested set of $r$-local separations.
    Moreover, if $\cB_r$ is $\Gamma$-invariant for a subgroup~$\Gamma$ of $\Aut(G)$, then~$N_r^k(G)$ is $\Gamma$-canonical.
\end{corlemma}

\begin{proof}
    The map $\beta \mapsto \beta_r$ given by \cref{correspondenceBottlenecks} restricts for each $k \in \N$, to a bijection from the set $\cB^k$ of all $k$-bottlenecks in~$\cB$ to the set $\cB^k_r$ of the $r$-local $k$-bottlenecks in $\cB_r$.
    In particular, the map $s \mapsto s_r$ from \cref{keylemma:correspondenceofseparations} restricts to a bijection $\bigcup \cB^k \to \bigcup \cB^k_r$.
     
    We prove by induction on $k \in \N$ that $s \mapsto s_r$ restricts to a bijection from $N^k(G,\beta)$ to $N_r^k(G,\beta_r)$ for every $\beta \in \cB^k$.
    The base case $k=0$ is trivial as there are no $r$-local $0$-bottlenecks in a connected graph.
    Assume now that~$k>0$.
    Then \cref{keylemma:correspondenceofseparations} together with the induction hypothesis yields that the map $s \mapsto s_r$ restricts to a bijection $N^{<k}(G) \to N_r^{< k}(G)$.
    Hence, an $s \in \beta \in \cB^k$ is nested with $N^{<k}(G)$ if and only if its $r$-local counterpart $s_r \in \beta_r \in \cB_r^k$ is nested with $N_r^{<k}(G)$, by \cref{corresCross}.
    Moreover, \cref{corresCross} also yields that an $s \in \bigcup \cB^k$ crosses a $t \in \bigcup \cB^k$ if and only if $s_r \in \bigcup \cB_r^k$ crosses $t_r \in \bigcup \cB_r^k$.
    In particular, $s_r$ crosses only finitely many such $t_r$ by \cref{CrossingNumberFinite}, and thus $x^k_r$ maps only to finite numbers.
    Moreover, $s \mapsto s_r$ restricts to a bijection $N^k(G,\beta) \to N_r^k(G,\beta_r)$ for every $\beta \in \bigcup \cB^k$.

    For the `in particular'-part, recall that~$N^k(G)$ is nested by~\cref{BottleneckNestedSet}, so~\cref{corresCross} implies that~$N_r^k(G)$ is nested as well.
    The `moreover'-part is immediate from the canonicity of~\cref{construction:inductiveNestedSetLocal}.
\end{proof}

\begin{theoremASS}[Nested set]\label{projectionliftNestedSet}
    Let $G$ be a locally finite, connected graph such that $(k-1)r<\mindist_r(G)$.
    Let $\cB_r$ be a set of $r$-local bottlenecks in~$G$, and let $\hat \cB_r$ be the set of $r$-local bottlenecks in~$G_r$ given by~\cref{LiftingLocalBottlenecks} applied to~$\cB_r$.
    Then $N_r^k(G)$ and $N_r^k(G_r)$ are well-defined, and $N_r^k(G_r)$ consists precisely of the lifts to $G_r$ of the elements of~$N_r^k(G)$.
    In particular, $N_r^k(G)$ is nested.
    Moreover, if $\cB_r$ is $\Gamma$-invariant for a subgroup~$\Gamma$ of $\Aut(G)$, then~$N_r^k(G)$ is $\Gamma$-canonical.
\end{theoremASS}

\begin{proof}
    Recall that $(k-1)r<\mindist_r(G)$ implies $f(k)\cdot r/2<\mindist_r(G)$ by \cref{evenmoreRandomInequality}.
    For every $j \leq k$, $p_r$ restricts to a surjection $\hat \beta \mapsto p_r(\hat \beta)$ from the set $\hat \cB_r^j$ of all $r$-local $j$-bottlenecks in $\hat \cB_r$ to the set $\cB_r^j$ of all $r$-local $j$-bottlenecks in~$\cB_r$; the map is well-defined by \cref{ProjectingLocalBottlenecks} and is surjective by \cref{LiftingLocalBottlenecks}.
    Moreover, \cref{ProjectingLocalBottlenecks} and \cref{LiftingLocalBottlenecks} ensure that $\bigcup \hat \cB_r^j$ is the set of all the lifts of elements of $\bigcup \cB_r^j$, for every $j \leq k$.

    We prove the claim by induction on~$k$.
    The base case $k=0$ is trivial because there are no $r$-local $0$-bottlenecks in a connected graph.
    So suppose that~$k>0$.
    It follows from the induction hypothesis that $N_r^{<k}(G_r)$ is the set of all lifts of the elements of $N_r^{<k}(G)$.
    Hence, $\hat s \in \hat \beta \in \hat \cB_r^k$ is nested with $N_r^{<k}(G_r)$ if and only if $p_r(\hat s) \in p_r(\hat \beta) \in \cB_r^k$ is nested with $N_r^{<k}(G)$ by \cref{NestedRtomic}.
    Moreover, \cref{NestedRtomic} yields that $s \in \bigcup \cB_r^k$ crosses $t \in \bigcup \cB_r^k$ if and only if every lift $\hat s \in \bigcup \hat \cB_r^k$ of $s$ crosses precisely one lift $\hat t \in \bigcup \hat \cB_r^k$ of $t$.
    Thus, the numbers $x^k_r(s)$ and $x^k_r(\hat s)$ in \cref{construction:inductiveNestedSetLocal} applied to $G$ and $G_r$ coincide.
    In particular, the former is a finite number, since the latter is, by \cref{CorrespondenceNestedSet}.
    Hence, $\hat s \mapsto p_r(\hat s)$ restricts to a surjection $N_r^k(G_r,\hat \beta) \to N_r^k(G, p_r(\hat \beta))$ for every $\hat \beta \in \hat \cB_r^k$ and,
    conversely, $N_r^k(G_r, \hat \beta)$ is the set of all lifts of elements of $N_r^k(G_r, \beta)$, for every $\beta \in \cB_r^k$, where $\hat \beta \in \hat \cB_r^k$ is the set of all lifts of elements of $\beta$ (which is an $r$-local $k$-bottleneck by \cref{LiftingLocalBottlenecks}).
    All in all, $N_r^k(G_r)$ is the set of all lifts of elements of $N_r^k(G)$.

    For the `in particular'-part, recall that~$N_r^k(G_r)$ is nested by~\cref{CorrespondenceNestedSet}, so~\cref{NestedRtomic} implies that~$N_r^k(G)$ is nested as well.
    The `moreover'-part is immediate from the canonicity of~\cref{construction:inductiveNestedSetLocal}.
\end{proof}

\begin{theoremASS}\label{LocalNestedToT}
    Let $G$ be a locally finite, connected graph such that $(k-1)r<\mindist_r(G)$.
    Let $\cB_r$ be a set of $r$-local bottlenecks in $G$.
    Then $N_r^{\le k}(G)$ meets every $r$-local bottleneck in~$\cB_r$ of order~$\le k$.
\end{theoremASS}
\begin{proof}
    Let $\hat \cB_r = \{ \hat \beta_r : \beta_r \in \cB_r\}$ be the set of $r$-local bottlenecks in~$G_r$ given by~\cref{LiftingLocalBottlenecks} applied to~$\cB_r$, and let $\hat \cB = \{\hat \beta : \hat \beta_r \in \hat \cB_r\}$ be the set of bottlenecks in~$G_r$ given by \cref{correspondenceBottlenecks}.
    Let~$\beta_r$ be any $r$-local bottleneck in~$\cB_r$, and write~$\hat \beta_r \in \hat \cB_r$ and~$\hat \beta \in \hat \cB$ for the corresponding ($r$-local) bottlenecks in~$G_r$.
    By~\cref{BottleneckNestedSet} applied to~$\hat \cB$ in~$G_r$, there is~$\hat s\in\hat\beta\cap N^{\le k}(G_r)$.
    So $\hat s_r\in \hat\beta_r\cap N_r^{\le k}(G_r)$ by \cref{CorrespondenceNestedSet} for~$\hat \cB_r$.
    Since $N_r^{\le k}(G_r)$ for~$\hat \cB_r$ is equal to the set of all lifts of all elements of $N_r^{\le k}(G)$ for~$\cB_r$ by \cref{projectionliftNestedSet}, we have $p_r(\hat s_r)\in \beta_r \cap N_r^{\le k}(G)$.
\end{proof}

\section{Constructing graph-decompositions from nested sets of local separations}\label{sec:constructinggddirectly}

With the construction of nested sets of $r$-local separations representing the $r$-local bottlenecks of a graph $G$ at hand (see~\cref{sec:LocalBottlenecks}), it remains to reconstruct from just this local information the \gd s of $G$ which previously, in \cite{canonicalGraphDec}, required full knowledge of the (usually infinite) cover $G_r$ of $G$ and its (also usually infinite) tree of tangles from \cref{construction:inductiveNestedSet}.

This begins with the challenge to construct a locally defined analogue of the \td\ from \cref{construction:candidatetreedecomposition}, which was based on a (global) tree of tangles of~$G_r$.
Building the \td\ relied on the partial ordering of global separations, and it is not clear what a meaningful $r$-local version of this partial ordering could be.

To overcome this obstacle, we develop an alternative perspective on~\cref{construction:candidatetreedecomposition} by introducing the notion of \emph{cutouts} as an alternative to splitting stars (\cref{sec:SplittingStarPerspective}), which we then use to find a suitable $r$-local framework:
We give the respective $r$-local definitions in~\cref{sec:LocalCutouts}, where we also show that they agree with their global counterparts on $r$-local coverings (\cref{correspondenceCutouts});
in particular, these cutouts give rise to the same \td\ of~$G_r$ as the globally defined tree of tangles did.
We then prove the respective projection/lifting lemmas along the covering map~$p_r$ (\cref{projectionCutoutStuff}, \cref{liftCutoutStuff} and \cref{CutoutMaster}). 
Finally, we describe our construction of graph-decompositions from nested sets of $r$-local separations (\cref{construction:rlocalsepstographdecompviaG}), which we show to agree with \cref{construction:candidatetreedecomposition} for nested sets of global separations on $r$-local coverings (\cref{correspondenceofcHrandcT}).
Our construction also commutes with~$p_r$ (\cref{ProjectionGD}):
The graph-decomposition from any nested set~$N$ of $r$-local separations is isomorphic to that obtained by taking the set~$\hat N$ of all the lifts to $G_r$ of elements of~$N$, applying our construction to~$\hat N$ in $G_r$, and folding the resulting tree-decomposition back along~$p_r$ via~\cref{construction:orbittreedecompositionisgraphdecomposition}.

\subsection{Splitting stars from the perspective of a relation} \label{sec:SplittingStarPerspective}

For an oriented separation $s = (A_1,A_2)$ of $G$ we denote its separator by~$X = \defnMath{X(s)}$ and define
\[
    \defnMath{L(s)} := E_G(A_1 \setminus A_2, X) = E_G(A_1 \setminus A_2, A_2) \quad\text{and}\quad
    \defnMath{R(s)} := E_G(X, A_2 \setminus A_1) = E_G(A_1, A_2 \sm A_1).
\]
We think of $L(s)$ and $R(s)$ as the \defn{left} and \defn{right} sides of the local separations induced by the separation $s$ (whose names are purely based on the reading direction).
Note that, for elements $s > s'$ of $\orientN{N}$, any edge in $R(s) \cap R(s')$ has a vertex in $X(s) \cap X(s')$. 

Let $N$ be a nested set of separations of a graph~$G$. 
We define a relation $\Sigma_N^\vx$ on the set  $\orientN{N}$ of orientations of elements of $N$ by letting
\begin{align*}
    \text{$\defnMath{s_1 \Sigma_N^\vx s_2}$ $:\Leftrightarrow$ }&\text{$X(s_1)$ and $X(s_2)$ share a vertex $v$ and $s_1>s_2^*$,}\\
    &\text{and there is no $s\in \orientN{N}$ with $v\in X(s)$ and $s_1>s>s_2^*$.}
\end{align*}
Recall that $s^*$ denotes the inverse $(B,A)$ of a separation $s = (A,B)$.
The relation $\Sigma_N^\vx$ is symmetric.
Recall from \cref{construction:candidatetreedecomposition} the definition of~$\cT(N)$.
If $\cT(N)$ is a \td , $\orientN{N}$ contains no $(\omega +1)$-chain. 
By \cite[Corollary~3.3]{infinitetreesets}, this implies that $\orientN{N}$ is the disjoint union of its splitting stars.

\begin{lemma}\label{GlobalSigmaVXsplits}
    Let $N$ be a nested set of proper separations of a connected graph~$G$ such that $\cT(N)$ is a tree-decomposition.
    Let $s_1,s_2\in\orientN{N}$ with $s_1\Sigma^\vx_N s_2$. Then $s_1$ and $s_2$ lie in the same splitting star of~$N$.
\end{lemma}
\begin{proof}
    By assumption, there is a vertex $v\in X(s_1)\cap X(s_2)$, we have $s_1> s_2^*$, and there is no $s\in \orientN{N}$ with $v\in X(s)$ and $s_1>s>s_2^*$.
    Suppose, for a contradiction, that $s_1$ and $s_2$ lie in distinct splitting stars of~$N$.
    Then $s_1> s_2^*$ with \cref{lemma:withthehelpofTD} yields that there is $s\in\orientN{N}$ with $s_1>s>s_2^*$.
    By definition of the partial ordering~$\geq$, the vertex $v$ lies in $X(s)$, a contradiction.
\end{proof}

Let $N$ be a nested set of separations of a graph~$G$.
For $s\in\orientN{N}$ we write
\[
    \defnMath{R_N(s)}:=R(s)\sm \bigcup\,\{\,R(s'):s'\in\orientN{N}\text{ such that }X(s')\text{ meets }X(s)\text{ and }s>s'\,\}
\]

We remark that the clause of $X(s')$ meeting $X(s)$ in the definition of $R_N(s)$ is redundant, since any edge in $R(s) \cap R(s')$ has a vertex in $X(s) \cap X(s')$, as $s > s'$. 
We still include it as a reminder.

For the next lemmas, recall that the \defn{interior} of a splitting star~$\sigma$ of $N$ is defined as $\defnMath{\interior(\sigma)} := \bigcap_{(A,B)\in\sigma} B$.
Whenever $\cT(N)$ from \cref{construction:candidatetreedecomposition} is a \td\ $(T,\cV)$ of $G$, its bags $V_\sigma$ are $\interior(\sigma)$ and we choose its parts $G_\sigma$ as $G[V_\sigma]$.

\begin{lemma}\label{RNsFromPartPerspective}
    Let $N$ be a nested set of proper separations of a connected graph~$G$ such that $\cT(N)$ is a tree-decomposition $(T,\cV)$.
    Let $\sigma$ be a splitting star of~$N$ and $s\in\sigma$.
    Then $R_N(s)=E(X(s),V_\sigma \setminus X(s))$.
\end{lemma}
\begin{proof}   
    Let $E:=E(X(s),V_\sigma \setminus X(s))$.
    For the backward inclusion $R_N(s)\supseteq E$, note first that $R(s) \supseteq E$ by the definition of~$V_\sigma$.
    Consider $e = xy \in R(s) \sm R_N(s)$ with $x \in X(s)$; we show that $e \notin E$.
    By the choice of~$e$, there is~$s' = (A, B) \in \orientN{N}$ with $e \in R(s')$ such that $X(s)$ and $X(s')$ meet and $s > s'$; we choose $s'$ maximal with respect to $>$.
    Then there is no $s''\in\orientN{N}$ with $s > s'' > s'$.
    By \cref{lemma:withthehelpofTD}, we thus have $(s')^* \in \sigma$.
    But $x \in X(s')$, as $e\in R(s')$, and hence $y \in B \setminus A$.
    Thus, $y \notin V_\sigma$ and so $e \notin E$.
    
    For the forward inclusion $R_N(s)\se E$, let $e=xy\in R(s)\sm E$ with $x\in X(s)$; we have to show $e\notin R_N(s)$.
    Let $\tau\in V(T)$ be closest to~$\sigma$ with $y\in V_\tau$.
    As $\tau\neq\sigma$, by \cref{lemma:withthehelpofTD} there is $s'=(A,B)\in\sigma$ with $y\in A\sm B$.
    Since $x\in X(s)\se V_\sigma\se B$, the fact that $s'$ is a separation implies $x\in X(s')$, so $e\in L(s')=R((s')^*)$.
    Then $e\notin R_N(s)$ as $s,s'\in\sigma$ gives $s>(s')^*$.
\end{proof}

Let $N$ be a nested set of separations of a graph~$G$.
We denote by $\defnMath{\cX:=\cX(N)}$ the set of all separators of elements of~$N$.
We denote by $\defnMath{\cK:=\cK(N)}$ the set of components of $G - \bigcup \cX$ and define the relation $\Sigma_N^\cK$ on $\orientN{N}$ by letting
\begin{center}
    $\defnMath{s_1 \Sigma_N^\cK s_2}$ $:\Leftrightarrow$ there is a component $K \in \cK$ such that $\partial K$ meets $R_N(s_i)$ for both $i=1,2$.
\end{center}
Note that $\Sigma_N^\cK$ is symmetric.

\begin{lemma}\label{GlobalSigmaKinPart}
    Let $N$ be a nested set of proper separations of a connected graph~$G$ such that $\cT(N)$ is a tree-decomposition $(T,\cV)$ with parts~$G_\sigma = G[V_\sigma]$.
    Let $K \in \cK(N)$.
    Then there is a unique node~$\sigma$ of $T$ whose part $G_\sigma$ contains $K$.
    Moreover, every $s\in\orientN{N}$ such that $R_N(s)$ meets $\partial K$ lies in~$\sigma$.
\end{lemma}
\begin{proof}

    As $K$ is a component of $G-\bigcup \cX(N)$ and the edges of $T$ are the separations in~$N$, we can orient every edge of $T$ in the direction of~$K$ to find a sink $\sigma$ of~$T$ whose part $G_\sigma$ contains $K$.
    Then $\sigma$ is unique with this property.
    
    Now let $s \in \orientN{N}$ be given as stated.
    As $\partial K$ meets $R_N(s) \subseteq R(s)$, our $s$ points towards~$K$.
    So we find, by the definition of splitting star, a unique $t\in\sigma$ with $s \geq t$.
    By assumption, $R_N(s)$ and $\partial K$ share an edge $xk$ with $x \in X(s)$ and $k \in K$.
    Since $X(t)$ is disjoint from $K$ but $X(t)$ separates $X(s)$ from $K$, the edge $xk$ is in $R(t)$.
    Hence, $s \not> t$ by $xk \in R_N(s)$, giving $s=t \in \sigma$ as desired.
\end{proof}

\begin{corollary}\label{GlobalSigmaKsplits}
    Let $N$ be a nested set of proper separations of a connected graph~$G$ such that $\cT(N)$ is a tree-decomposition.
    Let $s_1,s_2\in\orientN{N}$ with $s_1\Sigma^\cK_N s_2$. 
    Then $s_1$ and $s_2$ lie in the same splitting star of~$N$.\qed
\end{corollary}

An edge $e$ of $G$ \defn{bridges} a collection $\cX$ of vertex sets of $G$ if both ends of $e$ are contained in $\bigcup \cX$ but every $X\in\cX$ contains at most one end of~$e$.
If $N$ Is a set of separations of~$G$ and $\cX=\cX(N)$, we define the relation $\Sigma_N^\cX$ on $\orientN{N}$ by letting
\begin{center}
    $\defnMath{s_1 \Sigma_N^\cX s_2}$ $:\Leftrightarrow$ there is an $\cX$-bridging edge of $G$ that lies in $R_N(s_i)$ for both $i=1,2$.
\end{center}
Note that $\Sigma_N^\cX$ is symmetric.

\begin{lemma}\label{GlobalSigmaXinPart}
    Let $N$ be a nested set of proper separations of a connected graph~$G$ such that $\cT(N)$ is a tree-decomposition $(T,\cV)$ with parts~$G_\sigma = G[V_\sigma]$.
    Let $e$ be an $\cX(N)$-bridging edge of $G$.
    Then there is a unique node $\sigma$ of $T$ whose part $G_\sigma$ contains $e$.
\end{lemma}
\begin{proof}
    As $e$ bridges $\cX(N)$ and the edges of $T$ are the separations in~$N$, we can orient every edge of $T$ in the direction of~$e$ to find a sink $\sigma$ of~$T$ whose part $G_\sigma$ contains $e$.
    Then $\sigma$ is unique with this property.
\end{proof}

\begin{corollary}\label{GlobalSigmaXsplits}
    Let $N$ be a nested set of proper separations of a connected graph~$G$ such that $\cT(N)$ is a tree-decomposition.
    Let $s_1,s_2\in\orientN{N}$ with $s_1\Sigma^\cX_N s_2$. Then $s_1$ and $s_2$ lie in the same splitting star of~$N$.
\end{corollary}

\begin{proof}
    This follows immediately from \cref{RNsFromPartPerspective} and \cref{GlobalSigmaXinPart}.
\end{proof}

We denote by $\defnMath{\Sigma_N}$ the transitive reflexive closure of the union of the three symmetric relations $\Sigma^\vx_N$, $\Sigma^\cK_N$ and~$\Sigma^\cX_N$. 
Then $\defnMath{\Sigma_N}$ is an equivalence relation on~$\orientN{N}$.
The equivalence classes of $\Sigma_N$ are the \defn{cutouts} of~$N$.

\begin{proposition}\label{CutoutsAreSplittingStars}
    Let $N$ be a nested set of proper separations of a connected graph~$G$ such that $\cT(N)$ is a tree-decomposition $(T,\cV)$ with parts~$G_\sigma = G[V_\sigma]$.
    Then the cutouts of $N$ are precisely the splitting stars of~$N$.
\end{proposition}
\begin{proof}
    Every cutout of~$N$ is included in a splitting star of~$N$ by \cref{GlobalSigmaVXsplits}, \cref{GlobalSigmaKsplits} and \cref{GlobalSigmaXsplits}.

    For the converse, let $\sigma$ be a splitting star of~$N$, and assume for a contradiction that $\sigma$ is not included in a cutout of~$N$.
    Then, using that every cutout of~$N$ is included in a splitting star of~$N$, we find a bipartition $\{\sigma_1,\sigma_2\}$ of $\sigma$ such that the elements of $\sigma_1$ are not related to any elements of $\sigma_2$ (with regard to the relation~$\Sigma_N$).
    Let $X_i$ denote the union of all the separators of separations in~$\sigma_i$, for $i=1,2$.
    Let $P$ be a shortest $X_1$--$X_2$ path in~$G$ with ends~$x_i\in X_i$.
    For each~$i$, choose $s_i\in\sigma_i$ with $x_i\in X(s_i)$.
    We distinguish three cases.
    For all of them, let $(T,\cV):=\cT(N)$ and $\cX:=\cX(N)$.

    First, assume that $P$ has at least one internal vertex.
    Then $X_1$ and $X_2$ are disjoint and there is no $X_1$--$X_2$ edge.
    It follows from $\sigma$ being a star that all internal vertices of $P$ lie in $V_\sigma \setminus (X_1\cup X_2)$, and thus avoid, by \cref{lemma:withthehelpofTD}, all other $X \in \cX$.
    So there is a (unique) component $K$ of $G \setminus \bigcup\cX$ that contains all the internal vertices of~$P$.
    Let $e_i$ denote the edge on~$P$ that is incident to~$x_i$ for both~$i$; in particular, $e_i$ is in $\partial X_i$ and $\partial K$.
    Then $e_i\in R_N(s_i)$ by \cref{RNsFromPartPerspective}, so $K \subseteq G_\sigma$ by \cref{GlobalSigmaKinPart} and $s_i \in \sigma$ for $i=1,2$.
    Hence, the component $K$ witnesses $s_1\Sigma_N^\cK s_2$, which contradicts the choice of $\{\sigma_1,\sigma_2\}$.

    Second, assume that $P$ has a single edge~$x_1 x_2$.
    Then $X_1$ and $X_2$ are disjoint.
    Hence, $x_1 x_2$ is $\cX$-bridging and is an edge of the part~$G_\sigma$.
    Then $x_1 x_2\in R_N(s_i)$ for both~$i$ by \cref{RNsFromPartPerspective}.
    So the edge $x_1 x_2$ witnesses $s_1\Sigma_N^\cX s_2$, which contradicts the choice of $\{\sigma_1,\sigma_2\}$.

    Finally, assume that $P$ has only one vertex $x_1=x_2$.
    Then $X(s_1)$ and $X(s_2)$ share this vertex, and since $\sigma$ is a splitting star there is no $s\in\orientN{N}$ with $s_1>s>s_2^*$ by \cref{lemma:withthehelpofTD}.
    So $s_1 \Sigma_N^\vx s_2$ contradicts the choice of $\{\sigma_1,\sigma_2\}$.

    As we have derived contradictions in all cases, the proof is complete.
\end{proof}

Let $\sigma$ be a cutout of a set $N$ of separations of~$G$.
The \defn{$\sigma$-part} of~$G$, which we denote by $\defnMath{\apart(\sigma)}$ in formulas, is the subgraph of $G$ that is the union of the graphs listed under (i) and (ii) below plus the edges listed under (iii):
\begin{enumerate}
    \item the induced subgraph $G[X(s)]$ for every~$s\in\sigma$;
    \item the components $K\in\cK(N)$ such that $\partial K$ shares an edge with $R_N(s)$ for some~$s\in\sigma$; and
    \item the edges in $R_N(s)$, for every~$s\in\sigma$.
\end{enumerate}
Note that the edges listed in (iii) have both ends in graphs listed in (i) and (ii), by \cref{RNsFromPartPerspective}.

\begin{proposition}\label{TDCpartVScutoutPart}
    Let $N$ be a nested set of proper separations of a connected graph~$G$ such that $\cT(N)$ is a tree-decomposition~$(T, \cV)$ with parts~$G_\sigma = G[V_\sigma]$.
    Let $\sigma$ be a splitting star and, equivalently, a cutout of~$N$.
    Then $G_\sigma = \apart(\sigma)$.
\end{proposition}
\begin{proof}
    Let $I:=G_\sigma$ and $P:=\apart(\sigma)$, and recall that~$V_\sigma = \interior(\sigma)$ for all nodes~$\sigma \in T$ by the definition of~$\cT(N)$.
    
    We first show the inclusion $P\se I$.
    Since $\sigma$ is a star,
    we have $X(s) \subseteq \interior(\sigma)$ for every $s \in \sigma$, 
    and in particular $G[X(s)] \subseteq I$.
    If $K$ is a component in $\cK(N)$ such that $\partial K$ shares an edge with $R_N(s)$ for some $s\in\sigma$, then $K\se G_\sigma= I$ by \cref{GlobalSigmaKinPart}.
    The edges in $R_N(s)$ for any $s\in\sigma$ are edges in~$I$ by \cref{RNsFromPartPerspective}.
    
    We now show the converse inclusion, $I\se P$.
    The induced subgraphs $G[X(s)]$ for $s\in\sigma$ are included in~$P$ by definition.
    Let $\cK_\sigma$ denote the set of components of $I - \bigcup_{s \in \sigma} X(s)$.
    Fix any $K\in\cK_\sigma$.
    By \cref{lemma:withthehelpofTD}, $K \in \cK(N)$ and, by \cref{RNsFromPartPerspective}, $\partial K \subseteq \bigcup_{s \in \sigma} R_N(s)$.
    Since $G$ is connected, there is an edge $e\in\partial K$, so $K \subseteq P$ by (ii).
    It remains to show that every edge $e$ of $I$ that is not contained in a $G[X(s)]$ or $\partial K$ for a $K\in\cK_\sigma$ lies in~$P$.
    By \cref{lemma:withthehelpofTD}, such an edge $e$ bridges~$\cX$.
    Moreover, $e$~joins a vertex in $X(s_1)$ to a vertex in $X(s_2) \setminus X(s_1)$ for some two $s_1,s_2 \in \sigma$. 
    By  \cref{RNsFromPartPerspective}, $e$ is in $R_N(s_1)$.
    Hence, $e \in P$ by (iii).
\end{proof}

\subsection{Cutouts yield a local version of splitting stars} \label{sec:LocalCutouts}

Let $s = (E_1,X,E_2)$ and $t = (F_1,Y,F_2)$ be $r$-local separations of $G$ with $r$-tomic separators.
Inspired by \cref{char:NestedViaLinks}, we let \defn{$s \succeq t$} if
\begin{itemize}
    \item $X$ and $Y$ are $r$-coupled, and
    \item all three of the following sets are empty:
    the $X$-link for $F_2$, the $Y$-link for~$E_1$, and the edge set $E_1\cap F_2\cap\partial (X\cap Y)$.
\end{itemize}
We remark that $s \succeq t$ if and only if $s^* \preceq t^*$.

\begin{corlemma}[$\geq$]\label{correspondenceGEQ}
    Assume that the binary cycle space of $G$ is generated by cycles of length~$\le r$ and that $G$ is connected.
    Let $s$ and $t$ be oriented separations of~$G$ with $r$-tomic separators $X$ and $Y$, respectively, such that $X$ and $Y$ are $r$-coupled.
    Then $s \geq t$ if and only if $s_r \succeq t_r$.
\end{corlemma}

\begin{proof}
    This follows from \cref{correspondenceofseparationswithfixedseparator,corresLinks} and the `moreover'-part of \cref{char:NestedViaLinks}.
\end{proof}

\begin{liftlemmaASS}[$\geq$]\label{LiftGEQ}
    Let $G$ be a connected graph such that $(k-1)r<\mindist_r(G)$, and let $s$ and $t$ be oriented $r$-local separations of~$G$ with $r$-coupled $r$-tomic separators $X$ and $Y$ of size~$\le k$, respectively.
    Let $\hat s$ be any lift of~$s$ to $G_r$, with separator $\hat X$ say.
    Let $\hat t$ be the unique lift of~$t$ with separator $\hat Y$ such that $\hat X$ and $\hat Y$ are $r$-coupled, as considered in \cref{LiftingLinks}.
    Then $s \succ t$ if and only if $\hat s \succ \hat t$.
\end{liftlemmaASS}

\begin{proof}
    The equivalence $s\succeq t\Leftrightarrow \hat s\succeq\hat t$ follows from \cref{LiftingLinks}.
    Using the uniqueness of $r$-coupled lifts as in~\cref{uniqueRcoupleLift}, we derive the stronger equivalence $s\succ t\Leftrightarrow \hat s\succ\hat t$.
\end{proof}

For an $r$-local separation $s = (E_1,X,E_2)$ we define $\defnMath{X(s)} := X$, $\defnMath{L(s)} := E_1$ and $\defnMath{R(s)} := E_2$.
Let $N$ be a set of $r$-local separations of a graph~$G$.
We denote by $\defnMath{\orientN{N}}$ the set of all orientations of $r$-local separations in~$N$.
We define a relation $\Sigma_N^\vx$ on $\orientN{N}$ by letting
\begin{align*}
    \text{$\defnMath{s_1 \Sigma_N^\vx s_2}$ $:\Leftrightarrow$ }&\text{$X(s_1)$ and $X(s_2)$ share a vertex $v$ and $s_1\succ s_2^*$,}\\
    &\text{and there is no $s\in \orientN{N}$ with $v\in X(s)$ and $s_1\succ s\succ s_2^*$.}
\end{align*}
Note that $\Sigma_N^\vx$ is symmetric.
For $s\in\orientN{N}$ we write
\[
    \defnMath{R_N(s)}:=R(s)\sm \bigcup\,\{\,R(s'):s'\in\orientN{N}\text{ such that }X(s')\text{ meets }X(s)\text{ and }s\succ s'\,\}.
\]
We remark that stating that $X(s')$ meets $X(s)$ is technically redundant, since it holds for any $s \succ s'$ such that $R(s) \cap R(s') \neq \emptyset$.

For the remainder of this section~(\ref{sec:LocalCutouts}), let us assume that the binary cycle space of $G$ is generated by cycles of length~$\le r$ and that $G$ is connected.
Let $N$ be a set of oriented separations of $G$ with $r$-tomic separators and write $\defnMath{N_r}:=\{s_r:s\in N\}$.

\begin{corlemma}\label{correspondenceRN}
    For every $s\in\orientN{N}$ we have $R_N(s)=R_{N_r}(s_r)$.
\end{corlemma}
\begin{proof}
    This follows from \cref{correspondenceofseparationswithfixedseparator,correspondenceGEQ}.
\end{proof}

We denote by $\cX=\cX(N)$ the set of all separators of the $r$-local separations in~$N$.
We denote the set of components of $G - \bigcup \cX$ by $\defnMath{\cK:=\cK(N)}$ and define the relation $\Sigma_N^\cK$ on $\orientN{N}$ by letting
\begin{center}
    $\defnMath{s_1 \Sigma_N^\cK s_2}$ $:\Leftrightarrow$ there is a component $K \in \cK$ such that $\partial K$ meets $R_N(s_i)$ for both $i=1,2$.
\end{center}
Note that $\Sigma_N^\cK$ is symmetric.
We define the relation $\Sigma_N^\cX$ on $\orientN{N}$ by letting
\begin{center}
    $\defnMath{s_1 \Sigma_N^\cX s_2}$ $:\Leftrightarrow$ there is an $\cX$-bridging edge of $G$ that lies in $R_N(s_i)$ for both $i=1,2$.
\end{center}
Note that $\Sigma_N^\cX$ is symmetric.
As for ordinary separations, we denote by $\defnMath{\Sigma_N}$ the transitive reflexive closure of the union of all three symmetric relations $\Sigma^\vx_N$, $\Sigma^\cK_N$ and~$\Sigma^\cX_N$. 
Then $\defnMath{\Sigma_N}$ is an equivalence relation on $\orientN{N}$.
We refer to the equivalence classes of $\Sigma_N$ as \defn{cutouts} of~$N$.
The \defn{$\sigma$-part} of~$G$ for a cutout $\sigma$ of~$N$ is defined just like for a cutout of a set of usual separations before~\cref{TDCpartVScutoutPart}, and we denote it by~$\defnMath{\apart(\sigma)}$.

\begin{corlemma}\label{correspondenceCutouts}
    Assume that the binary cycle space of $G$ is generated by cycles of length~$\le r$ and that $G$ is connected.
    Let $N$ be a set of proper separations of $G$ with $r$-tomic separators.
    \begin{enumerate}
        \item $s\Sigma_N^* \,t\Longleftrightarrow s_r\Sigma_{N_r}^*\, t_r$ for all $s,t\in\orientN{N}$ and $*\in\{\vx,\cK,\cX\}$.
        \item\label{correspondenceCutoutsItem2} $\sigma\mapsto \sigma_r=\{s_r:s\in\sigma\}$ is a bijection between the cutouts of~$N$ and the cutouts of~$N_r$.
        \item\label{correspondenceCutoutsItem3} The $\sigma$-parts of~$G$ coincide with its~$\sigma_r$-parts.
    \end{enumerate}
\end{corlemma}
\begin{proof}
    Note that $\cX(N) = \cX(N_r)$ and $\cK(N) = \cK(N_r)$.
    
    (i).~For $*=\vx$ this is immediate from \cref{correspondenceGEQ}.
    For $*\in\{\cX,\cK\}$ it follows from \cref{correspondenceRN}.
    
    (ii) follows from (i).
    
    (iii) follows from (ii) by \cref{correspondenceRN}.
\end{proof}

\subsection{Interplay with the covering map} \label{sec:CutousAndCoveringMap}

\begin{projlemmaASS}\label{projectionRNs}
    Let $G$ be a connected graph such that $(k-1)r < \mindist_r(G)$, and let $N$ be a nested set of $r$-local separations of $G$ with $r$-tomic separators of size $\leq k$.
    Let~$\hat N$ be the set of all lifts of elements of $N$ from~\cref{keylemma:liftofseparations} and \cref{sufficientForLifting}.
    Then, for every $\hat s \in \orientN{\hat N}$ and $s:=p_r(\hat s)$, the $r$-local covering~$p_r$ restricts to a bijection $R_{\hat N}(\hat s) \to R_N(s)$.
\end{projlemmaASS}
\begin{proof}
    Let $\hat s\in\orientN{\hat N}$ be given and let $s:=p_r(\hat s)$.
    Let $\hat S$ be the set of all $\hat t\in \orientN{\hat N}$ such that $X(\hat t)$ meets $X(\hat s)$ and $\hat s\succ \hat t$.
    Define $S\se\orientN{N}$ analogously, with $N$ in place of~$\hat N$ and $s$ in place of~$\hat s$.
    So $\hat S$ and $S$ contain precisely the local separations considered in the definition of $R_N(\hat s)$ and $R_N(s)$.
    By \cref{uniqueRcoupleLift} and \cref{LiftGEQ}, the map~$p_r$ induces a bijection $\hat S\to S$, and for every $\hat t\in\hat S\cup\{\hat s\}$ it restricts to a bijection $R(\hat t)\to R(t)$.
    Hence, $p_r$ restricts to a bijection $R_{\hat N}(\hat s)\to R_N(s)$.
\end{proof}

\begin{projlemmaASS}[$\Sigma_N^*$]\label{projectionCutoutStuff}
    Let $G$ be a connected graph such that $(k-1)r < \mindist_r(G)$, let $N$ be a nested set of $r$-local separations of $G$ with $r$-tomic separators of size $\leq k$, and let $\hat N$ be the set of all lifts of elements of~$N$.
    \begin{enumerate}
        \item For all $\hat s, \hat t \in \orientN{\hat N}$ and $* \in \{\vx,\cK,\cX\}$ we have
        $\hat s \Sigma_{\hat N}^* \hat t\Longrightarrow p_r(\hat s) \Sigma_N^*\, p_r(\hat t)$.
        \item\label{projectionCutoutStuffItem2} For every cutout $\hat\sigma$ of~$\hat N$ there exists a unique cutout $\sigma$ of~$N$ such that $p_r(\hat\sigma)\se\sigma$.
        Moreover, $p_r(\apart(\hat\sigma))\se \apart(\sigma)$.
    \end{enumerate}
\end{projlemmaASS} 

\begin{proof}
    (i).~Let $\hat s,\hat t\in\orientN{\hat N}$ be as in the statement, and abbreviate $s:=p_r(\hat s)$ and $t:=p_r(\hat t)$.
    For $*=\vx$ we are done by~\cref{LiftGEQ}.
    For $*=\cX$ the assertion follows from \cref{projectionRNs} and $\bigcup\cX(\hat N)=p_r\inv (\bigcup\cX(N))$.
    
    For $*=\cK$ there is a component $\hat K$ of $G_r-\bigcup\cX(\hat N)$ such that $\partial\hat K$ meets $R_{\hat N}(\hat s)$ and $R_{\hat N}(\hat t)$.
    By path lifting, the projection $K:=p_r(\hat K)$ is an entire component of $G-\bigcup\cX(N)$.
    Then both $R_N(s)$ and $R_N(t)$ meet~$\partial K$ by \cref{projectionRNs}, giving~$s\Sigma_N^\cK t$.
    
    (ii).~The existence of $\sigma$ given $\hat\sigma$ is immediate from~(i).
    Every component $\hat K\in\cK(\hat N)$ projects to a component $p_r(\hat K)\in\cK(N)$, as observed in the proof of~(i).
    Hence, the `moreover'-part follows with \cref{projectionRNs}.
\end{proof}

\begin{liftlemmaASS}[$\Sigma_N^*$]\label{liftCutoutStuff}
    Let $G$ be a connected graph such that $(k-1)r < \mindist_r(G)$, let $N$ be a nested set of $r$-local separations of $G$ with $r$-tomic separators of size $\leq k$, and let $\hat N$ be the set of all lifts of elements of~$N$.
    \begin{enumerate}
        \item Let $s,t \in \orientN{N}$ and $* \in \{\vx,\cK,\cX\}$ be such that $s \Sigma_N^* t$.
            For every lift $\hat s$ of $s$ there exists a lift $\hat t$ of~$t$ such that $\hat s \Sigma_{\hat N}^* \hat t$.
        \item\label{liftCutoutStuffItem2} For every cutout $\sigma$ of $N$ and every lift $\hat s$ of some $s\in\sigma$, the cutout $\hat \sigma$ of~$\hat N$ with $\hat s\in\hat\sigma$ satisfies $\sigma\se p_r(\hat\sigma)$.
            Moreover, $\apart(\sigma)\se p_r(\apart(\hat\sigma))$.
    \end{enumerate}
\end{liftlemmaASS}

\noindent We remark that the lift $\hat t$ of $t$ in (i) will not normally be unique when $* = \cK$.

\begin{proof}
    (i).~Let $s,t$ be as in the statement of~(i), and let $\hat s$ be an arbitrary lift of~$s$.
    For $*=\vx$ we assume that $X(s)$ and $X(t)$ meet in a vertex~$v$, and we consider the unique lift $\hat t$ of $t$ such that~$X(\hat t)$ meets $X(\hat s)$ (in a lift $\hat v$ of $v$) from the statement of~\cref{LiftingLinks}.
    Then $\hat s \succ \hat t^*$ by \cref{LiftGEQ}.
    If there is~$\hat s' \in \orientN{\hat N}$, which is a lift of some~$s' \in \orientN{N}$, with $\hat v \in X(\hat s')$ and~$\hat s \succ \hat s' \succ \hat t^*$, then $v \in X(s')$ and~\cref{LiftGEQ} yields~$s \succ s' \succ t^*$, contradicting~$s \Sigma_{N}^\vx t$.
    Hence we have $\hat s\Sigma_{\hat N}^\vx \hat t$.
    
    For $* = \cK$ consider any $K \in \cK(N)$ such that both $R_N(s)$ and $R_N(t)$ meet~$\partial K$.
    Since $p_r$ restricts to a bijection $R_{\hat N}(\hat s) \to R_N(s)$ by \cref{projectionRNs}, there is a component $\hat K$ of $p_r\inv(K)$ such that $R_{\hat N}(\hat s)$ meets~$\partial\hat K$.
    As $K$ is connected, $p_r$ restricts to a surjection $\partial \hat K \to \partial K$.
    Choose $e\in \partial K\cap R_N(t)$ and let $\hat e\in\partial\hat K$ be a lift of~$e$.
    Let $\hat t$ be a lift of~$t$ with $\hat e\in\partial X(\hat t)$.
    Then $\hat e\in R_{\hat N}(\hat t)$ by \cref{projectionRNs}, so $\hat K$ witnesses that $\hat s\Sigma_{\hat N}^\cK\hat t$.
    
    For $* = \cX$ there is an $\cX(N)$-bridging edge $e$ of $G$ in $R_N(s) \cap R_N(t)$.
    Its lift $\hat e$ at the given $\hat s$ bridges~$\cX(\hat N)$, because $e$ bridges~$\cX(N)$, and it lies in~$R_{\hat N}(\hat s) \cap R_{\hat N}(\hat t)$ by~\cref{projectionRNs}.
    Hence, $\hat s \Sigma_{\hat N}^\cX \hat t$ as desired.

    (ii).~Let $\sigma,s,\hat\sigma,\hat s$ be as in the statement of~(ii).
    For every $t \in \sigma$, there are $n \in \N$ and $s_1, \dots ,s_{n-1} \in \sigma$ such that for every $i \in [n]$ there is some $*(i) \in \{\vx,\cK,\cX\}$ with $s_{i-1} \Sigma_N^{*(i)} s_i$, where $s_0 := s$ and $s_n := t$.
    Let $\hat s_0 := \hat s$.
    Applying~(i) iteratively for $i=1,\ldots,n$ yields lifts $\hat s_i$ of $s_i$ such that $\hat s_{i-1} \Sigma_N^{*(i)} \hat s_i$. Thus, $\hat t=\hat s_n \in \hat \sigma$, as desired.
    The `moreover'-part follows with \cref{projectionRNs}, since components in~$\cK(\hat N)$ project to components in~$\cK(N)$.
\end{proof}

\begin{propositionASS}[Cutout]\label{CutoutMaster}
    Let $G$ be a connected graph such that $(k-1)r < \mindist_r(G)$, let $N$ be a nested set of $r$-local separations of $G$ with $r$-tomic separators of size $\leq k$, and let $\hat N$ be the set of all lifts of elements of~$N$.
    Then $\hat \sigma \mapsto p_r(\hat \sigma)$ is a well-defined surjection from the cutouts of $\hat N$ to the cutouts of~$N$. Moreover, $p_r(\apart (\hat \sigma)) = \apart (p_r(\hat \sigma))$.
\end{propositionASS}

\begin{proof}
    \casen{Well-defined.}
    Let $\hat\sigma$ be a cutout of~$\hat N$.
    Then $p_r(\hat\sigma)\se\sigma$ for a unique cutout $\sigma$ of~$N$ by \cref{projectionCutoutStuff}~\cref{projectionCutoutStuffItem2}.
    The converse inclusion follows from \cref{liftCutoutStuff}~\cref{liftCutoutStuffItem2}, so~$\sigma = p_r(\hat \sigma)$.

    \casen{Surjective.} Let $\sigma$ be any cutout of~$N$. There exists a cutout $\hat\sigma$ of $\hat N$ with $\sigma\se p_r(\hat\sigma)$ by \cref{liftCutoutStuff}~\cref{liftCutoutStuffItem2}.
    As we just saw, $p_r(\hat \sigma)$ is itself a cutout of~$N$, which must be $\sigma$ since the cutouts of $N$ partition $\Sigma_N$.

    \casen{Parts.}
    We have `$\se$' by
    \cref{projectionCutoutStuff}~\cref{projectionCutoutStuffItem2} and `$\supseteq$' by \cref{liftCutoutStuff}~\cref{liftCutoutStuffItem2}.
\end{proof}

\subsection{From nested local separations to graph-decompositions} \label{sec:NestedLocalToGraphDec}

\begin{construction}[Graph-decomposition]\label{construction:rlocalsepstographdecompviaG}
    Let $N$ be a nested set of $r$-local separations of~$G$.
    We define the \defn{candidate} $\defnMath{\cH(N)} = (H, \cG)$ for a \gd\ with parts, as follows.
    If $N = \emptyset$, let $(H,\cG)$ with $\cG = (G_h : h \in V(H))$ be the trivial graph-decomposition whose decomposition graph $H$ is the graph on a single vertex $h$ and $G_h := G$.
    Assume now that $N$ is non-empty.
    The nodes of~$H$ are the cutouts of~$N$ and the edges of~$H$ are the $r$-local separations in~$N$, where a separation $\{E_1,X,E_2\} \in N$ connects the two (possibly equal) cutouts $\sigma_1,\sigma_2$ of~$N$ with $(E_{3-i},X,E_i)\in\sigma_i$.
    We define the decomposition part $G_h$ for a node $h$ of $H$ to be the $h$-part of $G$, with $h$ viewed as a cutout of~$N$.
\end{construction}

Note that $\cH(N)$ is uniquely determined by $G$ and $N$.
Thus, if $N$ is $\Gamma$-canonical for some subgroup $\Gamma$ of $\Aut(G)$, then so is $\cH(N)$.
Additional assumptions on~$N$ are required to ensure that $\cH(N)$ is a graph-decomposition.

Recall that an isomorphism $\psi\colon G\to G'$ between (multi-)graphs is a pair $\psi=(\psi_V,\psi_E)$ of bijections $\psi_V\colon V(G)\to V(G')$ and $\psi_E\colon E(G)\to E(G')$ that respect incidences of edges with vertices in both directions. 

\begin{corlemma}[Graph-decomposition]\label{correspondenceofcHrandcT}
    Assume that the binary cycle space of $G$ is generated by cycles of length $\leq r$ and that $G$ is connected.    
    Let $N$ be a nested set of proper separations of $G$ with $r$-tomic separators such that $\cT(N)$ is a \td\ $(T,\cV)$ with parts~$P_t = G[V_t]$.
    Let $(H,\cG):=\cH(N_r)$ with parts~$G_h = \apart(h)$.
    Then the map $s \mapsto s_r$ induces a canonical isomorphism $\psi\colon T\to H$ such that $G_{\psi(\sigma)}=P_\sigma$ for all nodes $\sigma$ of~$T$.
\end{corlemma}

\begin{proof}
    Let $\sigma$ be a node of~$T$.
    Recall that $\sigma$ is a splitting star of~$N$, by the definition of~$\cT(N)$.
    The splitting stars of $N$ are precisely the cutouts of~$N$ by \cref{CutoutsAreSplittingStars}, and $s\mapsto s_r$ defines a bijection between the cutouts of $N$ and the cutouts of~$N_r$ by \cref{correspondenceCutouts}~\cref{correspondenceCutoutsItem2}.
    So $\sigma\mapsto \sigma_r=\{s_r:s\in \sigma\}$ is a bijection $V(T)\to V(H)$ that combines with the $E(T) \to E(H)$ map $s \mapsto s_r$ to a graph isomorphism $\psi\colon T \to H$.
    We have $\apart(\sigma) = P_\sigma$ by \cref{TDCpartVScutoutPart}, and we have $\apart(\sigma)=\apart(\sigma_r)$ by \cref{correspondenceCutouts}~\cref{correspondenceCutoutsItem3}.
    Furthermore, $\apart(\sigma_r) = G_{\sigma_r}$ by~\cref{construction:rlocalsepstographdecompviaG}.
    With~$\psi(\sigma) = \sigma_r$ we find that $P_\sigma=G_{\psi(\sigma)}$.
\end{proof}

To apply \cref{correspondenceofcHrandcT}, we need a sufficient condition for when a nested set of proper separations of $G$ with $r$-tomic separators defines a tree-decomposition.

\begin{lemma}\label{rtomicTDCsufficient}
    Let $G$ be a locally finite, connected graph, let $k\in\N$, and let $N$ be a nested set of proper separations of~$G$ with $r$-tomic separators of size~$\le k$.
    \begin{enumerate}
        \item For every vertex $v$ of $G$, there are only finitely many separations in $N$ whose separators contain~$v$.
        \item $\cT(N)$ is a tree-decomposition of~$G$.
    \end{enumerate}
\end{lemma}
\begin{proof}
    (i).~Let $N_v$ consist of all $s\in N$ with $v\in X(s)$.
    Let $U:=B_{G}(v, (k-1)\lfloor r/2\rfloor + 1)$.
    Note that $U$ is finite.    
    Since the separators of the separations in $N$ are $r$-tomic, we have $X(s)\cup N_G(X(s))\se U$ for every $s\in N_v$.
    Hence, the map $\{A,B\}\mapsto \{A\cap V(U),B\cap V(U)\}$ that restricts each $\{A,B\}\in N_v$ to a proper separation of~$U$ is injective.
    Since $U$ is finite, there are only finitely many separations of~$U$, and so $N_v$ is finite.

    (ii).~By \cite[Lemma 2.7]{infinitesplinter} it suffices to show that $N$ has the following property: for every strictly decreasing sequence $(A_0,B_0)>(A_1,B_1)>\cdots$ of oriented separations in $\orientN{N}$, we have $\bigcap_{i\in\N}B_i=\emptyset$.
    To prove this, use (i) to find an infinite subsequence with disjoint separators as in the proof of \cite[Lemma~7.5]{canonicalGraphDec}.
\end{proof}

\begin{lemma}\label{nestedLocalInCoverGivesTDC}
    Let $G$ be a locally finite, connected graph, and assume that the binary cycle space of $G$ is generated by cycles of length $\leq r$.
    Let $N$ be a nested set of $r$-local separations of~$G$ with $r$-tomic separators all of size $\leq k$ for some $k\in\N$.
    Then $\cH(N)$ is a tree-decomposition (with parts) of~$G$.
    Moreover, if $N$ is $\Gamma$-canonical for some subgroup $\Gamma$ of $\Aut(G)$, then $\cH(N)$ is also $\Gamma$-canonical.
\end{lemma}
\begin{proof}
    By \cref{correspondenceofseparationswithfixedseparator}, there exists a unique set $M$ of separations of $G$ with $r$-tomic separators such that $M_r = N$.
    Since all separations in $M$ have $r$-tomic separators of size~$\le k$, we have that $\cT(M)$ is a tree-decomposition by \cref{rtomicTDCsufficient} (ii).
    Thus, $\cH(N)$ is as desired by~\cref{correspondenceofcHrandcT}.
\end{proof}

\begin{theoremASS}\label{ProjectionGD}
    Let $G$ be a locally finite, connected graph such that $(k-1)r < \mindist_r(G)$.
    Let $N$ be a nested set of $r$-local separations of $G$ with $r$-tomic separators of size $\leq k$, and let $(H,\cG):=\cH(N)$, with~$\cG = (G_h : h \in V(H))$ say.
    Let $\hat N$ be the set of all lifts of all elements of~$N$ to $G_r$, and recall that~$\cH(\hat N) =: (\hat H, \hat \cG)$ is a $\Gamma_r$-canonical tree-decomposition (with parts) of~$G_r$ by~\cref{nestedLocalInCoverGivesTDC}.
    Let $\bar\cH(N):=(\bar H,\bar \cG)$ with family~$\bar \cG = (\bar G_h : h \in V(\bar H))$ of parts be the graph-decomposition of $G$ that is obtained from~$\cH(\hat N)$ by folding via~$p_r$; see \cref{construction:orbittreedecompositionisgraphdecomposition}.
    
    For each cutout~$\sigma$ of~$N$, let $\tau(\sigma)$ denote the set of all cutouts $\hat\sigma$ of~$\hat N$ with $p_r(\hat\sigma)=\sigma$.
    For each~$s\in N$, let $\tau(s)$ denote the set of all lifts of~$s$.
    Then $\tau$ witnesses that $\cH(N)$ and $\bar \cH(N)$ are isomorphic, that is,
    \begin{enumerate}
        \item the pair of bijections $\tau := (\sigma\mapsto\tau(\sigma),s\mapsto \tau(s))$ is an isomorphism $H\to \bar H$, and
        \item $G_\sigma=\bar G_{\tau(\sigma)}$ for every $\sigma \in V(H)$.
    \end{enumerate}
\end{theoremASS}

\noindent In the context of~\cref{ProjectionGD}, we recall that the lifts of the $s \in N$ exist by \cref{sufficientForLifting} and \cref{randomInequality} and have $r$-tomic separators by~\cref{RtomicGives1sheeted}, and $\cH(\hat N)$ is a tree-decomposition (with parts) of~$G_r$ by \cref{nestedLocalInCoverGivesTDC}.

\begin{proof}
    (i). We first show that the map $\sigma \mapsto \tau(\sigma)$ is a bijection $V(H) \to V(\bar H)$.
    The vertices of $\bar H$ are the orbits of the vertices of $\hat H$ under the action of the group $\Gamma_r := \Gamma_r(G)$ of deck transformations of $p_r$ on $V(\hat H)$.
    The vertices of $\hat H$, in turn, are the cutouts of $\hat N$.
    The group $\Gamma_r$ acts naturally on $\hat N$, and since deck transformations preserve $\Sigma_N$, $\Gamma_r$ also acts naturally on the cutouts of $\hat N$.
    By~\cref{CutoutMaster} $p_r$ maps the cutouts of $\hat N$ onto the cutouts of $N$, and this map is well-defined on the $\Gamma_r$-orbits of cutouts of $\hat N$.
    By \cref{sufficientForLifting} the set of all lifts of a separation in~$N$ forms a $\Gamma_r$-orbit of~$\hat N$, and by \cref{liftCutoutStuff} (ii) this transfers to cutouts of~$N$ and their lifts to $\hat N$.
    Thus, the $\Gamma_r$-orbits of the set of cutouts of $\hat N$ are the sets $\tau(\sigma)$ for cutouts $\sigma$ of $N$.
    All in all, $p_r$ maps the sets~$\tau(\sigma)$, the vertices of $\bar H$, bijectively to the cutouts of $N$, the vertices of $H$.

    Similarly, the edges of $H$ are the separations in $N$, while the edges of $\bar H$ are the $\Gamma_r$-orbits of the edges of $\hat H$, the separations in $\hat N$.
    The group $\Gamma_r$ acts naturally on $\hat N$, and by \cref{sufficientForProjection} $p_r$ maps the $\Gamma_r$-orbits of $\hat N$ to $N$.
    As above the $\Gamma_r$-orbits of $\hat N$ are the sets $\tau(s)$ for separations~$s$ in $N$.
    Thus, $p_r$ maps the sets $\tau(s)$, the edges of $\bar H$, bijectively to the separations in $N$, the edges of $H$.
    
    The incidences between edges and vertices are clearly preserved.

    (ii).\
    Let $\sigma\in V(H)$ be a cutout.
    We have to show $G_\sigma=\bar G_{\tau(\sigma)}$.
    We have already seen in~(i) that $\tau(\sigma)$ is a $\Gamma_r$-orbit of the cutouts of~$\hat N$.
    As $\bar\cH(N)$ is obtained from the tree-decomposition $\cH(\hat N)$ (with parts that are induced subgraphs of~$G_r$) by folding via~$p_r$ as in \cref{construction:orbittreedecompositionisgraphdecomposition}, we therefore have for any cutout~$\hat\sigma\in\tau(\sigma)$ of~$\hat N$ that $\bar G_{\tau(\sigma)}=p_r(\hat G_{\hat\sigma})=p_r(\apart(\hat\sigma))$.
    Finally, $p_r(\apart(\hat\sigma))=\apart(\sigma)=G_\sigma$ by~\cref{CutoutMaster} using $\sigma=p_r(\hat\sigma)$.
\end{proof}

\begin{corollaryASS}\label{nestedLocalSeparationsDetermineGD}
    Let $G$ be a locally finite, connected graph such that $(k-1)r < \mindist_r(G)$.
    Let $N$ be a nested set of $r$-local separations of $G$ with $r$-tomic separators of size $\leq k$.
    Then $\cH(N)$ is a graph-decomposition with parts.
    Moreover, if $N$ is $\Gamma$-canonical for a subgroup $\Gamma$ of $\Aut(G)$, then so is~$\cH(N)$.\qed
\end{corollaryASS}

Let $\cH=(H,\cG)$ be a graph-decomposition of some graph~$G$, and let $F$ be a set of edges of~$H$.
We define $\cH/F:=(\tilde H,\tilde\cG)$ where $\tilde H$ is obtained from $H$ by contracting the edges in~$F$ (keeping loops and parallel edges), and $\tilde\cG$ contains for every vertex $\tilde h$ of~$\tilde H$ the part $\tilde G_{\tilde h}$ of $G$ that is the union of all parts $G_h$ with $h$ contained in the branch-set of~$\tilde h$.
We say that $\cH$ \defn{refines} $\cH/F$.
If $\cH=\cH(N)$ for a set of $r$-local separations and $N'\se N$, then $\cH/N'$ is well-defined, since the edge set of $H$ is, formally, the set $N$, of which $N'$ is a subset.

\begin{lemmaASS}\label{RefinementGD}
    Let $G$ be a locally finite, connected graph such that $(k-1)r < \mindist_r(G)$.
    Let $N\se M$ be two nested sets of $r$-local separations of $G$ with $r$-tomic separators of size $\leq k$.
    Then $\cH(N)=\cH(M)/(M\sm N)$.
\end{lemmaASS}
\begin{proof}
    Let $\hat N$ be the set of all lifts to $G_r$ of the elements of~$N$, which is possible by \cref{sufficientForLifting} and \cref{randomInequality}.
    We define $\hat M$ analogously.
    Then $N\se M$ implies $\hat N\se \hat M$.
    By \cref{nestedLocalInCoverGivesTDC}, both $\cH(\hat N)$ and $\cH(\hat M)$ are tree-decompositions of~$G_r$.
    Hence, $\cH(\hat N)=\cH (\hat M)/(\hat M\sm\hat N)$ by standard arguments, see e.g.~\cite{confing}.
    Since $\hat N$ is a $\Gamma_r$-invariant subset of $\hat M$, this construction is invariant under $\Gamma_r$ itself and hence $p_r$ is defined on $\Gamma_r$-orbits, to yield by \cref{construction:orbittreedecompositionisgraphdecomposition} graph-decompositions $\bar \cH(N)$ and $\bar \cH(M)$ of $G$ such that $\bar \cH(N) = \bar \cH(M) / (M \setminus N)$.
    By \cref{ProjectionGD}, we then get $\cH(N) = \bar \cH(N) = \bar \cH(M) / (M\sm N) = \cH(M) / (M\sm N).$
\end{proof}

\section{Proof of \texorpdfstring{\cref{MainIntro} and \cref{MainAlgo}}{the main results}}\label{sec:MainProof}

Recall that $N_r^{\le k}(G)$ was defined in \cref{construction:inductiveNestedSetLocal}, which for this section we take to refer to the set of \emph{all} $r$-local bottlenecks of order $\le k$ in $G$ as the set~$\cB_r^k(G)$.
We define $\cH_r^{\le k}(G):=\cH(N_r^{\le k}(G))$ with the operator $\cH$ taken from \cref{construction:rlocalsepstographdecompviaG}.
Recall that the edges of the decomposition graph of $\cH_r^{\le k}(G)$ are the $r$-local separations in~$N_r^{\le k}(G)$.
We say that these edges \defn{represent} all the $r$-local bottlenecks in~$G$ of order~$\le k$ if every $r$-local bottleneck in~$G$ of order~$\le k$ contains such an edge (viewed as the $r$-local separation that it formally is).

\begin{customthm}{\cref{MainIntro}}
    Let $G$ be a finite connected graph and $r,k \in \N$ with $k<K(G,r) = \frac{\mindist_r(G)}{r} + 1$.
    \begin{enumerate}
        \item\label{MainIntro:ToT} $\cH_r^{\le k}(G)$ is a canonical graph-decomposition of~$G$ that represents all the $r$-local bottlenecks in~$G$ of order~$\le k$.
        \item\label{MainIntro:Default} $\cH_r^{\le k}(G_r)$ is a canonical tangle tree-decomposition of scope~$k$ of the $r$-local covering~$G_r$.
        \item\label{MainIntro:Benchmark} $\cH_r^{\le k}(G)$ is canonically isomorphic to the graph-decomposition of $G$ that is obtained from the canonical tree-decomposition $\cH_r^{\le k}(G_r)$ of the $r$-local covering~$G_r$ by folding via~$p_r$.
        \item\label{MainIntro:Refinement} $\cH_r^{\le k}(G)$ refines $\cH_r^{\le \ell}(G)$ for all $\ell\in\N$ with $\ell\le k$.
    \end{enumerate}
\end{customthm}

\begin{proof}
    By assumption we have $(k-1)r<\mindist_r(G)$.
    Further, we claim that we may assume without loss of generality that \textsuperscript{\ref{ASSUMPTION}} is satisfied; that is, $r \ge 2$ or $k = 1$.
    To see this, suppose that $r\le 1$.
    Then all tight $r$-local separators of $G$ have size $1$ by \cref{lem:Tight2LocalSepAreSize1}, and also all tight separators of $G_r$ have size~$1$, as in this case $G_r$ is the universal cover of $G$ which is a tree.
    Thus, all $r$-local bottlenecks in $G$ and all bottlenecks in $G_r$ have order~$1$, and we may assume without loss of generality that $k=1$.

    \cref{MainIntro:Default}.\
    By~\cref{CorrespondenceNestedSet}, $N_r^{\le k}(G_r)$ is well-defined and nested, and by~\cref{construction:inductiveNestedSetLocal}, the $r$-local separations in $N_r^{\le k}(G_r)$ were selected from $r$-local bottlenecks.
    So they are tight and, in particular, $r$-tomic by~\cref{TightLocalSeparatorAreOConnected}.
    Hence, $\cH_r^{\le k}(G_r)$ is a tree-decomposition of~$G_r$ by \cref{nestedLocalInCoverGivesTDC}, and it is canonical because $N_r^{\le k}(G_r)$ is canonical by~\cref{CorrespondenceNestedSet}.
    Every pair of $\lek$-distinguishable tangles in~$G_r$ induces a bottleneck in~$G_r$ by \cref{TanglesInduceBottleneck}.
    The map $s\mapsto s_r$ induces a bijection between the bottlenecks in $G_r$ and the $r$-local bottlenecks in~$G_r$, by \cref{correspondenceBottlenecks}.
    Hence, $N_r^{\le k}(G_r)$ is a tree of tangles of $G_r$ of scope~$k$, and so $\cH_r^{\le k}(G_r)$ is a tangle tree-decomposition of scope $k$ of $G_r$.

    \cref{MainIntro:Benchmark}.\
    By \cref{projectionliftNestedSet}, $N_r^{\le k}(G)$ is well-defined and nested, and $N_r^{\le k}(G_r)$ equals the set of all lifts to~$G_r$ of all elements of~$N_r^{\le k}(G)$.
    Further, the elements of $N_r^{\le k}(G)$ are $r$-tomic by~\cref{TightLocalSeparatorAreOConnected} because they come from $r$-local bottlenecks by~\cref{construction:inductiveNestedSetLocal}.
    Hence, \cref{MainIntro:Benchmark} follows from \cref{MainIntro:Default} and \cref{ProjectionGD}.
    
    \cref{MainIntro:ToT}.\
    By~\cref{MainIntro:Benchmark} $\cH_r^{\le k}(G)$ is a graph-decomposition of~$G$.
    Since $N_r^{\le k}(G)$ is canonical by~\cref{projectionliftNestedSet} and $\cH_r^{\le k}(G) = \cH(N_r^{\le k}(G))$, also $\cH_r^{\le k}(G)$ is canonical by~\cref{construction:rlocalsepstographdecompviaG}.
    So it remains to show that the edges of its decomposition graph~$H$ represent all the $r$-local bottlenecks in~$G$ of order~$\le k$.
    Recall that $E(H) = N_r^{\le k}(G)$ by~\cref{construction:rlocalsepstographdecompviaG}.
    And indeed, every $r$-local bottleneck $\beta_r$ in $G$ of order~$\le k$ shares an $r$-local separation with~$N_r^{\le k}(G)$ by \cref{LocalNestedToT}.

    \cref{MainIntro:Refinement}.\
    Since $N_r^{\le\ell}(G)\se N_r^{\le k}(G)$ by definition, this follows from \cref{RefinementGD}.
\end{proof}

\begin{proof}[Proof of \cref{MainAlgo}]
    Given $G,r,k$ as in \cref{MainIntro}, the algorithm first performs \cref{construction:inductiveNestedSetLocal} to determine $N_r^{\le k}(G)$, and in the second step applies \cref{construction:rlocalsepstographdecompviaG} to $N_r^{\le k}(G)$ to determine~$\cH_r^{\le k}(G)$.
\end{proof}

\newpage
\appendix

\section{Corners of local separations}\label{secCorners}

We did not require an analogue of corners for local separations in this paper.
Corners of genuine separations are frequently used in the literature, however, and we believe that their local counterparts can be defined as follows.
The following definitions are supported by \cref{fig:LocalCorner}.
Let $\{A_1,A_2\}$ and $\{C_1,C_2\}$ be two separations of a graph~$G$ with separators $X$ and~$Y$, respectively.
Let $i,j\in [2]$.
Recall that the \defn{corner-separator for $A_i$ and $C_j$} is the union of the $Y$-link for $A_i$ with the $X$-link for $C_j$ with $X\cap Y$.
And the \defn{corner for $A_i$ and $C_j$} is the separation $\{A_i\cap C_j,A_{3-i}\cup C_{3-j}\}$.
As $i,j$ range over~$[2]$, there are four corner-separators and four corners.
Now let $\{E_1,X,E_2\}$ and $\{F_1,Y,F_2\}$ be two oriented $r$-local separations of~$G$.
The \defn{corner-separator for $E_i$ and $F_j$} is the union of the $Y$-link for~$E_i$ with the $X$-link for~$F_j$ with $X\cap Y$.
Let $L$ denote the corner-separator for $E_i$ and~$F_j$.
The \defn{corner for $E_i$ and $F_j$} is the triple 
\[\big\{\,\partial L\sm (E_{3-i}\cup F_{3-j}),\,L,\,\partial L\cap (E_{3-i}\cup F_{3-j})\,\big\}.\]

\begin{figure}[ht]
    \centering
    \includegraphics[height=6\baselineskip]{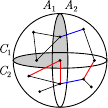}
    \caption{The edges in $E_2$ are coloured blue, while the edges in $F_2$ are coloured red. The corner-separator~$L$ is coloured grey}
    \label{fig:LocalCorner}
\end{figure}

\begin{corlemma}[Corners]\label{keylemma:correspondenceofcorner}
    Assume that the binary cycle space of $G$ is generated by cycles of length $\le r$ and that $G$ is connected.
    Let $s=\{A_1,A_2\}$ and $s'=\{C_1,C_2\}$ be crossing separations of $G$ with $r$-tomic separators.
    Let $\{E_1,X,E_2\}:=s_r$ and $\{F_1,Y,F_2\}:=s'_r$, as well as $i,j\in [2]$.
    The corner $d$ of $s_r,s'_r$ for $E_i$ and $F_j$ is equal to $c_r$ where $c$ is the corner of $s,s'$ for $A_i$ and~$C_j$.
    In particular, the four corners of $s_r,s'_r$ are $r$-local separations of~$G$.
\end{corlemma}
\begin{proof}
Without loss of generality $i,j=1$.
Then $c=\{A_1\cap C_1,A_2\cup C_2\}$.
The separator $L$ of $c$ is equal to the separator of~$d$ by \cref{corresLinks}.
It therefore remains to show the following two equalities:
\begin{align*}
    \partial L\sm (E_2\cup F_2)&=E(L,(A_1\cap C_2)\sm L)\\
    \partial L\cap (E_2\cup F_2)&=E(L,(A_2\cup C_2)\sm L)
\end{align*}
These are straightforward to check, for example by using \cref{fig:LocalCorner}.
\end{proof}

\newpage
\printbibliography
\end{document}